\documentclass{amsart}
\usepackage{dsfont}
\usepackage[usenames,dvipsnames,svgnames,table]{xcolor}
\usepackage[utf8]{inputenc}
\usepackage[T1]{fontenc}
\usepackage{ tipa }
\usepackage{helvet}
\usepackage{color}
\usepackage{mathrsfs}  
\usepackage{stmaryrd}  
\usepackage{amsmath,amsxtra,amsthm,amssymb,xr,enumerate,fullpage}
\usepackage[all]{xy}
\usepackage[bbgreekl]{mathbbol}
\usepackage{bbm}

\DeclareMathSymbol{\invques}{\mathord}{operators}{`>}
\DeclareUnicodeCharacter{00BF}{\tmquestiondown}
\DeclareRobustCommand{\tmquestiondown}{%
  \ifmmode\invques\else\textquestiondown\fi
}
\usepackage{verbatim}
\numberwithin{equation}{section}

\makeatletter
\newcommand{\mylabel}[2]{#2\def\@currentlabel{#2}\label{#1}}
\makeatother
\newtheorem{theorem}{Theorem}[section]
\newtheorem{lemma}[theorem]{Lemma}

\newtheorem{proposition}[theorem]{Proposition}
\newtheorem{corollary}[theorem]{Corollary}
\newtheorem{defn}[theorem]{Definition}

\newtheorem{remark}[theorem]{Remark}
\newtheorem{remdef}[theorem]{Remark/Definition}

\newcommand{\Gal}{\operatorname{Gal}}
\newcommand{\Fil}{\operatorname{Fil}}

\newcommand{\bD}{\mathbf{D}}
\newcommand{\DD}{\mathbb{D}}
\newcommand{\BB}{\mathbb{B}}

\newcommand{\QQ}{\mathbb{Q}}
\newcommand{\Qp}{\mathbb{Q}_p}
\newcommand{\Zp}{\mathbb{Z}_p}
\newcommand{\ZZ}{\mathbb{Z}}

\newcommand{\FF}{\mathbb{F}}
\newcommand{\FFF}{\mathcal{F}}
\newcommand{\fF}{\mathfrak{F}}

\newcommand{\ord}{\mathrm{ord}}

\newcommand{\fp}{\mathfrak{p}}

\newcommand{\vp}{\varphi}
\newcommand{\cL}{\mathcal{L}}
\newcommand{\cH}{\mathcal{H}}
\newcommand{\cO}{\mathcal{O}}

\newcommand{\HIw}{H^1_{\mathrm{Iw}}}

\newcommand{\GL}{\mathrm{GL}}

\newcommand{\Brig}{\BB_{{\rm rig},F}^+}

\newcommand{\cyc}{\textup{cyc}}

\newcommand{\ac}{\textup{ac}}
\newcommand{\LL}{\Lambda}
\newcommand{\TT}{\mathbb{T}}

\newcommand{\RR}{\mathcal{R}}
\newcommand{\f}{\textup{\bf f}}

\newcommand{\lra}{\longrightarrow}
\newcommand{\ra}{\lra}
\newcommand{\res}{\textup{res}}

\newcommand{\ur}{\textup{ur}}
\newcommand{\fP}{\mathfrak{P}}

\newcommand{\cP}{\mathcal{P}}

\newcommand{\Dcris}{\mathbb{D}_{\rm cris}}

\newcommand{\Tw}{\mathrm{Tw}}

\newcommand{\cor}{\mathrm{cor}}

\newcommand{\red}{\color{red}}

\newcommand{\p}{\mathfrak{p}}

\newcommand{\m}{\mathfrak{m}}

\newcommand{\cD}{\mathcal{D}}

\newcommand{\cG}{\mathcal{G}}
\newcommand{\sA}{\mathscr{A}}
\newcommand{\sF}{\mathscr{F}}
\newcommand{\sL}{\mathscr{L}}

\newcommand{\cE}{\mathcal{E}}
\newcommand{\cR}{\mathcal{R}}
\newcommand{\EXP}{\mathrm{EXP}}

\newcommand{\fN}{\mathfrak{N}}
\newcommand{\Cp}{\mathbb{C}_p}
\newcommand{\eG}{\varepsilon_\cG}
\newcommand{\Bcris}{\mathbb{B}_{\rm cris}}

\newcommand{\fe}{\mathfrak{e}}
\newcommand{\sW}{\mathscr{W}}
\newcommand{\Ig}{\mathrm{Ig}}
\newcommand{\x}{\mathrm{\bf x}}

\newcommand{\fa}{\mathfrak{a}}
\newcommand{\Pic}{\mathrm{Pic}}
\newcommand{\rec}{\mathrm{rec}}
\newcommand{\Aut}{\mathrm{Aut}}
\newcommand{\cA}{\mathcal{A}}
\newcommand{\bz}{\mathbf{z}}

\usepackage{hyperref}

\begin{document}

\title{Interpolation of generalized Heegner cycles in Coleman families}

\begin{abstract}
Kobayashi recently proved that the generalized Heegner cycles of Bertolini--Darmon--Prasanna can be interpolated along the anticyclotomic tower, giving rise to distribution valued cohomology classes with expected growth rate. We interpolate these classes along Coleman families. This construction plays a role in the proofs of $p$-adic Gross--Zagier formulae at for non-ordinary eigenforms and consequentially, also in the proof of a conjecture of Perrin-Riou. 
\end{abstract}

\author{K\^az\i m B\"uy\"ukboduk}
\address{K\^az\i m B\"uy\"ukboduk\newline UCD School of Mathematics and Statistics\\ University College Dublin\\ Ireland}
\email{kazim.buyukboduk@ucd.ie}

\author{Antonio Lei}
\address{Antonio Lei\newline
D\'epartement de Math\'ematiques et de Statistique\\
Universit\'e Laval, Pavillion Alexandre-Vachon\\
1045 Avenue de la M\'edecine\\
Qu\'ebec, QC\\
Canada G1V 0A6}
\email{antonio.lei@mat.ulaval.ca}


\maketitle
\tableofcontents
\section{Introduction}
Fix forever a prime $p\geq 5$ and an imaginary quadratic field $K$ where $(p)=\fp\fp^c$ splits (the superscript $c$ will always stand for the action of a fixed complex conjugation). We let $g \in S_\kappa(\Gamma_0(N)\cap \Gamma_1(p))$ denote a $p$-stabilized cuspidal eigenform of level $Np$ which is crystalline at $p$ and has even weight $\kappa\ge2$. We assume that every prime dividing the level $N$ splits in $K$. Let us write $g_{/K}$ for the base change of $g$ to $K$ and let $\psi$ denote an anticyclotomic Hecke character of $K$ of infinity type $(j,-j)$ with $-\frac{\kappa}{2}<j<\frac{\kappa}{2}$. Kobayashi showed in \cite{kobayashiGHC} that the generalized Heegner cycles of Bertolini--Darmon--Prasanna~\cite{bertolinidarmonprasanna13} associated to Rankin--Selberg convolutions $g_{/K}\otimes\psi$  can be interpolated as $\psi$ varies among anticyclotomic Hecke characters. The main purpose of the current article is to show that we may interpolate these classes  as $g$ varies in a Coleman family  (Theorem~\ref{thm_main_GHCinterpolate_Intro} below). This extends the work of Howard~\cite{howard2007} in the case of slope-zero (Hida) families.

Motivations for this interpolation result include its consequences towards the $p$-adic Gross--Zagier formula at critical slope, Perrin-Riou's conjecture comparing Beilinson--Kato elements to Heegner points and Birch and Swinnerton--Dyer formulae for elliptic curves of rank one. These applications are  discussed in the preprint~\cite{BPSI}. 

Our approach exploits the $p$-adic construction of rational points, a theme first observed by Rubin~\cite{Rubin92} and further explored by Perrin-Riou~\cite{PR93RubinsFormula}, Bertolini--Darmon~\cite{BertoliniDarmon2007} and Bertolini--Darmon--Prasanna \cite{bertolinidarmonprasanna12, bertolinidarmonprasanna13}. More precisely, our argument to interpolate generalized Heegner cycles dwells on the formula of Bertolini--Darmon--Prasanna, which relates the Bloch--Kato logarithms of {these cycles} to appropriate Rankin--Selberg $p$-adic $L$-values.  We briefly outline our strategy, which consists of three steps:
\begin{itemize}
\item { In Appendix~\ref{sec:biglogalaonganticyclotower}, using the theory of $(\vp,\Gamma)$-modules, we explain how to extend the work of Perrin-Riou  to construct a big exponential map  which interpolates Bloch--Kato exponential maps for $g_{/K}\otimes\psi$ as the newform $g$ varies in a Coleman family and $\psi$ varies among anticyclotomic Hecke characters. }
\item We outline in Section~\ref{subsec_anticyclo_padic_L_in_families} (where we adapt Brako\v{c}evi\'{c}'s work \cite{miljan2} for Hida families) the construction of a two-variable $p$-adic $L$-function (one variable parametrizing the variation of $g$ in a Coleman family, the other variable accounting for  the anticyclotomic variation), which interpolates the Rankin--Selberg $p$-adic $L$-function {of} Brako\v{c}evi\'{c}/Bertolini--Darmon--Prasanna as $g$ varies in a Coleman family. 
\item The $\p$-local candidate for the ``universal'' generalized Heegner cycle is then defined as the image of the big exponential map on the two-variable Brako\v{c}evi\'{c}/Bertolini--Darmon--Prasanna $p$-adic $L$-function. Relying on a $\LL(\widetilde{\Gamma}_\ac)$-adic version of the Bertolini--Darmon--Prasanna formula (Theorem~\ref{thm_GHC_rec_law_g} in the main text; here $\LL(\widetilde{\Gamma}_\ac)$ denotes the anticyclotomic Iwasawa algebra), we prove in 
Section~\ref{subsec_GHC_interpolated} that the $\p$-local cohomology class we construct arises as the restriction of a uniquely determined global cohomology class\footnote{We shall work under the hypothesis \ref{item_Int_PR}, which concerns integrality properties of the Perrin-Riou map given in Appendix~\ref{sec:biglogalaonganticyclotower}. This property is valid for slope-zero (Hida) families; in the general setup, it is a work in progress of Ochiai. Without this hypothesis, we can still prove the existence of families of $\p$-local classes and show that they do arise from  global classes, c.f. Remark~\ref{rem_JLZ}. Under \ref{item_Int_PR}, not only can we prove the existence of the global classes independently of \cite{JLZ}, but we can also show that these classes enjoy certain integrality properties. In Appendix~\ref{appB}, we explain that our method gives rise to global classes as $\psi$ varies over a proper wide open subdisc of the anticyclotomic weight space ${\rm Spm}(\LL(\Gamma_\ac)[1/p])$ without assuming \ref{item_Int_PR}.} which necessarily interpolates the generalized Heegner cycles of Bertolini--Darmon--Prasanna.
\end{itemize}

Before we give  precise statements of our results and discuss related prior and forthcoming works, let us first fix our notation and set the hypotheses we shall work with.

\subsection{Set up}\label{subsec_setup}
We fix once and for all an embedding $\iota_p: \overline{\QQ}\hookrightarrow \mathbb{C}_p$ and suppose that the prime $\fp$ of $K$ lands inside the maximal ideal of $\cO_{\mathbb{C}_p}$. We fix also an embedding $\iota_\infty: \overline{\QQ}\hookrightarrow \mathbb{C}$ and an isomorphism $\frak{j}:\mathbb{C}\stackrel{\sim}{\ra}  \mathbb{C}_p$ such that $\frak{j}\circ \iota_\infty =\iota_p$. The ring of integers of the completion of the maximal unramified extension $\QQ_p^\ur$ of $\Qp$ is denoted by $\sW$. We assume until the end of this article that the strong Heegner hypothesis (relative to the imaginary quadratic field $K$ and integer $N$) holds true. We fix a factorization $N\cO_K=\fN\fN^c$.

Let $f \in S_k(\Gamma_0(N))$ be a normalized cuspidal eigen-newform of level $N$ and even weight $k\ge2$. We assume that the fixed prime $p$ does not divide $ N$. Let $\alpha$ and $\beta$ be the two roots of the Hecke polynomial $X^2-a_p(f)X+p^{k-1}$. We assume that $\alpha\neq \beta$ and let $f^\alpha$ and $f^\beta$ denote the two $p$-stabilizations of $f$. We fix a finite extension $L$ of $\Qp$ that contains the Hecke field of $f$ as well as $\alpha$ and $\beta$. Throughout the article,  we shall write $\lambda$ for either one of the two roots. We assume that neither $f^\alpha$ nor $f^\beta$ is in the image of the operator  on the space of overconvergent modular forms
$$\theta_{k-2}: M^\dagger_{2-k}(\Gamma_0(N)\cap\Gamma_1(p))\lra M^\dagger_{k}(\Gamma_0(N)\cap\Gamma_1(p))$$ 
 given by $\left(q\frac{d}{dq}\right)^{k-1}$ on $q$-expansions.

We write $W_f$ for Deligne's $2$-dimensional representation with coefficients in $L$, whose Hodge--Tate weights are $0$ and $1-k$ (with the convention that the Hodge-Tate weight of the cyclotomic character is $+1$) and write $V_f:=W_f(k/2)$ for its central critical twist. We fix a Galois-stable $\cO_L$-lattice $T_f$ inside $V_f$ {given as in \S\ref{subsec_duality_deligne_poincare} below}.  We define and notate similar objects associated to a $p$-stabilized cuspidal eigenform $g\in S_\kappa(\Gamma_0(N)\cap\Gamma_1(p))$ of even weight $\kappa\geq 2$.

For a positive integer $c$, we denote by $K_c$ the ring class extension of $K$ modulo $c$. 
Let $\Sigma$ denote a {finite} set of places of $\QQ$, which contains all primes dividing $Np$ as well as the archimedean place. For any extension $F/\QQ$, we shall abuse notation to denote the set of places of $F$ lying above $\Sigma$  with the same symbol $\Sigma$. We write $F_\Sigma$ for the maximal extension of $F$ unramified outside $\Sigma$ and define $G_{F,\Sigma}:=\Gal(F_{\Sigma}/F)$.
\subsubsection{Distributions}
Let $K_{p^\infty}$ denote the ring class field of $K$ of conductor $p^\infty$. We write $\widetilde{\Gamma}_\ac:=\Gal(K_{p^\infty}/K)=\Gamma_\ac\times\Delta_\ac$, where $\Gamma_\ac\cong\Zp$ and $\Delta_\ac$ is a finite group. We shall write $\Lambda(\Gamma_\ac):=\ZZ_p[[{\Gamma}_\ac]]$ and $\Lambda(\widetilde{\Gamma}_\ac):=\ZZ_p[[\widetilde{\Gamma}_\ac]]$ for the completed group rings of ${\Gamma}_\ac$ and $\widetilde{\Gamma}_\ac$, respectively.

Fix $h\in\mathbb{R}\ge0$. { As in \cite[­\S I]{colmez98},} we define a valuation on the power series ring $\Qp[[X]]$ given by
\[
v_h\left(\sum_nc_nX^n\right)=\inf\left\{\ord_p(c_n)+h\ell(n):n\ge0\right\},
\]
where $\ell(0)=0$ and $\ell(n)= \frac{\log(n)}{\log(p)}+1$ for $n\ge1$. We may then define the set of \emph{power series of logarithmic order $h$} by
\[
\cH_h=\{F\in \Qp[[X]]:v_h(F)>-\infty\},
\]
and the subset of power series which are integral with respect to $v_h$ by
\[
\cH_h^+=\{F\in \cH_h:v_h(F)\ge0\}.
\]
 On choosing a topological generator $\gamma_\ac$ of ${\Gamma}_\ac$ and replacing $X$ by $\gamma_\ac-1$ in the definitions above, we define the subsets $\cH_h^+(\Gamma_\ac)$ and $\cH_h(\Gamma_\ac)$ of $\QQ_p[[{\gamma_\ac-1}]]$. We set $\cH_h(\widetilde{\Gamma}_\ac):=\cH_h(\Gamma_\ac){\otimes}_{\LL(\Gamma_\ac)}\LL(\widetilde{\Gamma}_\ac)$ and similarly define $\cH_h^+(\widetilde{\Gamma}_\ac)$. We note that $\cH_0(\widetilde{\Gamma}_\ac)=\LL(\widetilde{\Gamma}_\ac)\otimes_{\ZZ_p}\QQ_p$. Finally, if $W$ is a ring that contains $\Zp$, we set $\cH_h(\Gamma_\ac)_{W}:=\cH_h(\Gamma_\ac)\otimes_{\ZZ_p}{W}$ and similarly define $\cH_h^+(\Gamma_\ac)_{W}$.
\subsubsection{Coleman families}
\label{subsec_Coleman_families_revisited}
Let $\mathcal{W}={\rm Sp}\,\ZZ_p[[\Zp^\times]]$ denote the weight space. For each  $r=p^{v}<p^{\frac{p-1}{p-2}}$ with $v=\ord_p(e)$ for some $e\in L$, we let $B(r)\subset \mathcal{W}$ denote the closed affinoid disc about $k$ of radius $r$. Let $\mathscr{A}(r)$ denote the ring of $L$-valued analytic functions on $B(r)$ and let $\sA^\circ(r)\subset \mathscr{A}(r)$ be the subring of power-bounded elements. Both rings $\mathscr{A}(r)$ and $\mathscr{A}^\circ(r)$ are Noetherian for each $r$. We also consider the open disc $B^\circ(r)\subset B(r)$ of radius $r$ about $k$, which we think of as an $L$-rigid analytic space (see \cite[\S7]{deJong95} for a detailed description of the rigid analytic open ball). This inclusion induces an injective ring homomorphism
$$\sA^\circ(r)\lra \LL_{(k,r)}:=\cO_L\left[\left[\frac{X-k}{e}\right]\right],$$
where we think of $\LL_{(k,r)}$ as functions on $B^\circ(r)$ which are bounded by $1$; and where the image of $\sA^\circ(r)$ is given by $$\left\{\sum_{n=0}^\infty c_n\left(\frac{X-k}{e}\right)^n\in \cO_L\left[\left[\frac{X-k}{e}\right]\right]:\lim_{n\rightarrow\infty}c_n=0\right\}.$$ This map also induces a ring homomorphism 
$$\sA(r)\lra \LL_{(k,r)}[1/p]\,.$$

Fix $r_0= p^{v_0}$, with $v_0:=\ord_p(e_0)$ for some $e_0\in L$. Let $\f$ be a Coleman family over the affinoid disc $B(r_0)$ of fixed slope $v(\lambda):=\ord_p(\lambda)$ (which we can always ensure on shrinking $B(r_0)$ appropriately), through the $p$-stabilization $f^\lambda$ of $f$. The Coleman family $\f$ admits a formal $q$-expansion $\f=\sum_{n=1}^\infty A_nq^n$ with $A_n\in \LL_{(k,r_0)}$  thanks to the work of Coleman \cite{coleman97}. Let $\bblambda\in \LL_{(k,r_0)}$ denote the eigenvalue with which $U_p$ acts on $\f$. As explained in {Remark~\ref{remark_abstract_vs_physical} below (see also Remark~\ref{rem_appendix_explain_C_again})}, there exists a free $\LL_{(k,r_0)}$-module $\TT_\f$ of rank two equipped with a continuous  $G_{\QQ,\Sigma}$-action, interpolating the self-dual twists of the Galois representations of classical specializations of $\f$. It is self-dual in the sense that we have an isomorphism 
\begin{equation}
\label{eqn_symplectic_Coleman}
\TT_\f\stackrel{\sim}{\lra}{\rm Hom}_{\LL_{(k,r_0)}}(\TT_\f,\LL_{(k,r_0)})(1).
\end{equation}
of $G_{\QQ,\Sigma}$-representations. 

We set $\TT_\f^\ac:=\TT_{\f}\,\widehat{\otimes}_{\ZZ_p}\ZZ_p[[\widetilde{\Gamma}_\ac]]^\iota$, which is a free $\LL_{(k,r_0)}\widehat{\otimes}\ZZ_p[[\widetilde{\Gamma}_\ac]]$-module of rank two, on which we allow $G_{\QQ,\Sigma}$ act diagonally and $\ZZ_p[[\widetilde{\Gamma}_\ac]]^\iota$ denotes $\ZZ_p[[\widetilde{\Gamma}_\ac]]$, where the action of $\gamma\in\widetilde{\Gamma}_\ac$ is given by multiplication by $\gamma^{-1}$. 

 We define $B^\circ(r_0)_{\rm cl}$ to be the set of  classical points in the disc $B^\circ(r_0)$, so that we have for $\kappa \in B^\circ(r_0)_{\rm cl}$
$$\f(\kappa):=\sum_{n=1}^\infty A_n(\kappa)q^n\in S_\kappa(\Gamma_0(N)\cap\Gamma_1(p)).$$ 
We call $\f(\kappa)$ the weight-$\kappa$ specialization of $\f$. {On shrinking $B(r_0)$ if necessary, one can ensure that} the $p$-stabilized eigenform $\f(\kappa)=\sum_{n=1}^\infty A_n(\kappa)q^n$ is $p$-old, and it arises as the $p$-stabilization of a newform $\f(\kappa)^\circ \in S_\kappa(\Gamma_0(N))$ with respect to the eigenvalue $\bblambda(\kappa)$. There is a corresponding ring homomorphism 
$$\pi_\kappa: \LL_{(k,r)} \lra \cO_L$$
which induces morphisms of $G_{\QQ,\Sigma}$-modules
$$\pi_\kappa:\TT_\f^\ac\lra T_{\f(\kappa)}^\ac\,\,\,\,\,\,\,\,,\,\,\,\,\,\, \pi_\kappa:\TT_\f\lra T_{\f(\kappa)}$$
where $T_{\f(\kappa)}$ is an $\cO_L$-lattice in the central critical twist of Deligne's representation attached to $\f(\kappa)^\circ$. For each $\kappa$ as above, let us fix a generator $P_\kappa\in \LL_{(k,r_0)}$ of $\ker(\pi_\kappa)$.

{ The following hypotheses will be in effect for the main results of our article.
\begin{itemize}
\item[\mylabel{item_BI}{\bf(BI)}] $\{M\in \GL_2(\Zp):\det(M)\in (\Zp^\times)^{k-1}\}\subset {\rm im}\left(G_\QQ\stackrel{\rho_f}{\longrightarrow} \Aut_{\cO_L}(T_f)\cong \GL_2(\cO_L)\right)$;
\item[\mylabel{item_irr}{\bf(Irr)}] { If $v(\lambda)>0$, then} the residual representation of $\rho_\f$ when restricted to $G_{\Qp}$ is absolutely irreducible. 
\end{itemize}
}
\begin{remark}
\label{remark_regularity_mod_p_rep}
\item[i)] Suppose that the hypothesis {\ref{item_BI}} holds. Then the residual representation $\overline{\rho}_f$ is full, in the sense that ${\rm im}(\overline{\rho}_f)$ contains a conjugate of ${\rm SL}_2(\FF_p)$. This implies that the projective image of $\overline{\rho}_f$ contains ${\rm PSL}_2(\FF_p)$ up to conjugation. As explained in \cite[Remark 2.28]{ContiLangMedvedovsky}, it follows that the regularity condition in \cite[Theorem 6.2]{ContiIT2016} for $\overline{\rho}_f$, as well as \cite[Corollary 7.2]{ContiIT2016} for its deformation to a Coleman family hold true. 
\item[ii)]{ The hypothesis \ref{item_irr} holds true very often. For example, it holds whenever $\f$ admits a crystalline specialization of weight $k$ where $2\leq k\leq p+1$;  c.f. \cite[Theorem~2.6]{FontaineEdixhoven92}. We refer the reader to \cite[Th\'eor\`eme 3.2.1]{Berger2010} for further sufficient conditions.
}
\end{remark}

\subsection{Results} We are now ready to record our main result (the interpolation of generalized Heegner cycles of Bertolini--Darmon--Prasanna in Coleman families).
\begin{theorem}
\label{thm_main_GHCinterpolate_Intro} Suppose $c_0$ is a positive integer prime to $pN$. For each $\kappa\in B^\circ(r_0)_{\rm cl}$, we let 
$${\frak{z}}_{\f(\kappa),c_0}^\ac\in H^1(G_{K_{c_0},\Sigma},T_{\f(\kappa)}^\ac)\,\widehat{\otimes}_{ \LL(\widetilde{\Gamma}_\ac)}\,p^{c(\lambda)}\mathcal{H}_{v(\lambda)}^+(\widetilde{\Gamma}_\ac)$$ denote the $\Lambda(\widetilde{\Gamma}_\ac)$-adic cycle, interpolating the generalized Heegner cycles of Bertolini--Darmon--Prasanna along the anticyclotomic tower, given as in {Definition~\ref{defn_pstabilizedGHC_levelNp} $($see also Theorem~\ref{thm_kobayashi_ota_ac_interpolation}$)$.}

{ There exists a class }
$${{\frak{z}}_{\f,c_0}^\ac}\in H^1(G_{K_{c_0},\Sigma},\TT_{\f}^\ac{ [1/p]})\,\widehat{\otimes}_{ \mathcal{H}_{0}(\widetilde{\Gamma}_\ac)}\,\mathcal{H}_{v(\lambda)}(\widetilde{\Gamma}_\ac)
$$ 
such that 
${{\frak{z}}_{\f,c_0}^\ac}(\kappa)=
{\frak{z}}_{\f(\kappa),c_0}^\ac$ for every $\kappa\in B^\circ(r_0)_{\rm cl}$.  If moreover \ref{item_Int_PR} holds, 
then there exists an integer $c(\lambda)$ that depends only on $v(\lambda)$ such that
$${{\frak{z}}_{\f,c_0}^\ac}\in H^1(G_{K_{c_0},\Sigma},\TT_{\f}^\ac)\,\widehat{\otimes}_{\Zp}\,p^{c(\lambda)}\mathcal{H}_{v(\lambda)}^+(\widetilde{\Gamma}_\ac)\,.
$$
\end{theorem}
This statement corresponds to Theorem~\ref{thm_main_GHC_in_families_with_ac}(i) in the main body of our article, combined with Proposition~\ref{prop_main_GHC_in_families_with_ac_descended_to_L}  together with the discussion in Remark~\ref{rem_JLZ}. Theorem~\ref{thm_main_GHCinterpolate_Intro} may be recast in terms of classical Heegner cycles:
\begin{corollary}
\label{cor_compareGHCtononpstabilizedHeegCycle_intro}
We let $\frak{z}_{\f,c_0}\in H^1(K_{c_0},\TT_\f)\otimes\QQ_p$ denote the image of the class $\frak{z}_{\f,c_0}^{\ac}$ under the natural projection
$$ H^1(K_{c_0},\TT_{\f}^\ac)\widehat{\otimes}_{ \LL(\widetilde{\Gamma}_\ac)} \mathcal{H}_{v(\lambda)}(\widetilde{\Gamma}_\ac)
\lra  H^1(K_{c_0},\TT_{\f})\otimes \Qp
\,.$$ 
\item[i)] For any $r<r_0$, let us define $\TT_{\f}\vert_{B(r)}:=\TT_{\f}\otimes_{\LL_{(k,r_0)}} \mathscr{A}(r)$, which is a  free $\mathscr{A}(r)$-module of rank $2$. We have 
$$\frak{z}_{\f,c_0}\in \widetilde{H}^1 (G_{K_{c_0},\Sigma},\TT_{\f}\vert_{B(r)};\Delta_{\bblambda}),
$$
where $\widetilde{H}^1 (G_{K_{c_0},\Sigma},\TT_{\f}\vert_{B(r)};\Delta_{\bblambda})$ is the Pottharst Selmer group interpolating Bloch--Kato Selmer families along $B(r)$ $($see Definition~\ref{defn_pottharst_selmer_on_B_r} for its precise definition$)$.
\item[ii)] For any crystalline classical point $\kappa\in B(r)$ with $r<r_0$, we let ${z}_{\f(\kappa)^\circ,c_0} \in H^1_{\rm f}(K_{c_0},V_{\f(\kappa)})$ denote the classical Heegner class $($given as in \cite[\S 3]{NekovarGZ}$)$ of conductor $c_0$. {We put ${z}_{\f(\kappa)^\circ}:={\rm cor}_{K_1/K}{z}_{\f(\kappa)^\circ,1}$. Then for all crystalline classical points $\kappa\in B(r)$, we have
 $${\frak{z}}_{\f}\vert_{B(r)}(\kappa)=\left(1-\dfrac{p^{\frac{\kappa}{2}-1}}{\bblambda(\kappa)}\right)^2\cdot\dfrac{u_K^{-1}}{(2\sqrt{-D_K})^{\frac{\kappa}{2}-1}}\cdot\dfrac{W_{Np}\circ({\rm pr}^{\bblambda(\kappa)})^*(z_{\f(\kappa)^\circ})}{\bblambda(\kappa)\,\lambda_N(\f(\kappa)^\circ)\cE(\f(\kappa))\cE^*(\f(\kappa))}\,.$$
 Here, $W_{Np}$ is the Atkin--Lehner operator $($which we have normalized as in \cite[\S2.5]{KLZ2}$)$, $({\rm pr}^{\bblambda(\kappa)})^*$ is a certain linear combination of the degeneracy maps given as in Section~\ref{subsec_duality_deligne_poincare}, $\lambda_N(\f(\kappa)^\circ)$ is the Atkin--Lehner psuedo-eigenvalue, the multipliers $\cE(\f(\kappa))$ and $\cE^*(\f(\kappa))$ are defined as in Lemma~\ref{lemma_KLZ_2_Prop_10_1_1_last_par}.}
\end{corollary}
Corollary~\ref{cor_compareGHCtononpstabilizedHeegCycle_intro}(i) corresponds to Theorem~\ref{thm_main_GHC_in_families_with_ac}(ii) in the main body of our article, combined with Proposition~\ref{prop_main_GHC_in_families_with_ac_descended_to_L}. Corollary~\ref{cor_compareGHCtononpstabilizedHeegCycle_intro}(ii) follows from combining Theorem~\ref{thm_main_GHC_in_families_with_ac}(i) and Proposition~\ref{prop_compareGHCtononpstabilizedHeegCycle} below. 

\begin{remark}
\label{remark_abstract_vs_physical}

\item[i)]
We follow here the discussion in \cite[\S4]{LZ1} and strictly follow the notation in op. cit., {except for the fact that the specialization of the weight variable $\kappa_U$ is the weight of the corresponding overconvergent modular form in our article, whereas in op. cit., the specialization of $\kappa_U$ is the weight of the corresponding overconvergent modular form minus $2$}.  We note that in the notation of \cite{LZ1}, $U:=B^\circ(r)$ {(once we translate $U$ in op. cit. via ``$x\mapsto x+2$'' in the weight space, as we have indicated above)}, $\LL_U:=\LL_{(k,r_0)}$ and $B_U:=\LL_{(k,r_0)}[1/p]$. One constructs a free $B_U$-module $M_U(\f)$ of rank two in the overconvergent \'etale cohomology $($of Andreatta--Iovita--Stevens$)$ of the modular curves $($see Theorem 4.6.6 in op.cit.$)$, which interpolates Deligne's representations associated to classical points in the disc $U$. Moreover, one may consider a $\LL_U$-lattice $M_U^\circ(\f)^*\subset M_U(\f)^*$, which is given as the reflexive hull of the image of $M_{U,0}^\circ(\mathscr{H}_0^\prime)$  in  $M_U(\f)^*$ $($see Definition 4.4.6 in op. cit. for the definition of $M_{U,0}^\circ(\mathscr{H}_0^\prime)$$)$. The $\LL_U$-module $M_U^\circ(\f)^*$ is free of rank $2$. {This is because all reflexive modules over the ring $\LL_U$ are free since $\LL_U$ has Krull dimension $2$, and duals of finitely generated modules over a commutative normal Noetherian domain are reflexive; see \cite[\href{https://stacks.math.columbia.edu/tag/0AUY}{Section 0AUY}]{stacks-project} for a detailed discussion on reflexive modules.}  The $\LL_U$-module $M_U^\circ(\f)^*$ interpolates {homothetic copies of the} canonical lattices in Deligne's representations which are realized in the cohomology of modular curves.  
\item[ii)]
{In our main results, we take $\TT_\f$ to be the $\LL_{(k,r_0)}$-adic lattice ${p^{-\mathscr{C}}}M_U^\circ(\f)^*(1-\kappa_U/2)$ {for a suitably large integer $\mathscr{C}$ and small $r_0$; see Remark~\ref{rem_appendix_explain_C_again} how we choose an integer $\mathscr{C}$ and radius $r_0$ that are suitable for our purposes. We note once again that the universal weight variable $\kappa_U$ here is the universal weight variable in \cite{LZ1} plus 2.}} 
\item[iii)] {Even though it is not relevant to our arguments in this article, we remark that the symplectic structures on Deligne's representations induced from the Poincar\'e duality do not necessarily interpolate integrally, to yield the isomorphism \eqref{eqn_symplectic_Coleman}. It does interpolate once we invert $p$ $($see \cite[Theorem 4.6.6]{LZ1}$)$: There is a perfect Poincar\'e duality  
\begin{equation}\label{eqn_Poincare_in_families}
\{\,,\,\}: \, M_U(\f) \otimes M_U(\f)^* \lra B_U,\end{equation}
where $\kappa_U$ is the universal weight character. However, the pairing \eqref{eqn_Poincare_in_families} does \emph{not} restrict to a perfect pairing on the $\LL_U$-lattices whose construction we have outlined in $($i$)$.}
\end{remark}

\subsection{Related prior and forthcoming work}
In the case of slope-zero (Hida) families, Theorem~\ref{thm_main_GHCinterpolate_Intro} and Corollary~\ref{cor_compareGHCtononpstabilizedHeegCycle_intro} follow from combining the results of \cite{howard2007, CastellapadicvariationofHeegnerpoints, CastellaHsiehGHC}{; we note that these results were extended by Ota in~\cite{ota2020} relying on global methods}. A more general version of these results (in that it is not necessary to assume that the prime $p$ splits in $K$, and that the strong Heegner hypotheses can be relaxed to generalized Heegner hypothesis) is due to Disegni~\cite{disegniuniversalHeegcycle}.

In the positive slope case, Jetchev--Loeffler--Zerbes released a preprint~\cite{JLZ} shortly before the current article was made publicly available, where they utilize the techniques of \cite{LZ1} (which in turn dwells on the overconvergent \'etale cohomology of Andreatta--Iovita--Stevens). Their approach is very different from ours and does not require the prime $p$ be split in $K$  nor the integrality property \ref{item_Int_PR} of the Perrin-Riou maps. We {meticulously keep track of} the integrality properties, dwelling crucially on \ref{item_Int_PR}, of the universal generalized Heegner cycles, which is important for potential applications towards main conjectures for families. Moreover, the  description of the $p$-local properties of the interpolated classes in Corollary~\ref{cor_compareGHCtononpstabilizedHeegCycle_intro}(i) plays a crucial role in \cite{BPSI}. The precise relation of the universal generalized Heegner cycles to classical Heegner cycles in Corollary~\ref{cor_compareGHCtononpstabilizedHeegCycle_intro}(ii) is a generalization of work of Castella~\cite{CastellapadicvariationofHeegnerpoints} and Ota~\cite{ota2020} to the positive-slope case, and it is another key ingredient in \cite{BPSI}.

We note that there is yet a third independent approach to interpolate Heegner Cycles in positive-slope families: Based on a strategy similar to one employed in \cite{howard2007,disegniuniversalHeegcycle}, the forthcoming work~\cite{BPSII} of Pollack, Sasaki and the first named author gives another construction of universal Heegner cycles. The main technical input in op. cit. is Emerton's completed cohomology theory; more particularly, his realization of the eigencurve within the completed cohomology for $\GL2_{/\QQ}$.

\subsection*{Acknowledgments}The authors would like to thank Ming-Lun Hsieh, Zheng Liu and Filippo Nuccio for answering their questions during the preparation of this article. Further, they thank David Loeffler heartily for his comments and suggestions. They also thank the anonymous referee for very helpful and constructive comments and suggestions on earlier versions of the article, which led to many improvements.  The first named author thanks the Department of Mathematics at Harvard University for their hospitality, where the approach we follow here was conceived. He also thanks heartily Barry Mazur for numerous discussions and suggestions on this project. The first named author acknowledges support from European Union's Horizon 2020 research and innovation programme under the Marie Sk{\l}odowska-Curie Grant Agreement No. 745691. Parts of this work were carried during the second named author's visits to Harvard University in June 2018 and to University College Dubin in November 2018. He would like to thank these institutes for their hospitality. The second named author's research is supported by the NSERC Discovery Grants Program  RGPIN-2020-04259 and RGPAS-2020-00096.

\section{$p$-adic $L$-functions in families}

\subsection{One-variable anticyclotomic $p$-adic $L$-functions}\label{S:padicL-one-var}

In this section, we review the one-variable anticyclotomic $p$-adic $L$-function of  Brako\v{c}evi\'{c} \cite{miljan} and Bertolini--Darmon--Prasanna \cite{bertolinidarmonprasanna13}. For our purposes, it is more convenient  to follow the exposition presented in \cite{CastellaHsiehGHC}.

Let $\Ig(N)_{/\ZZ_{(p)}}$ be the Igusa scheme of level $N$ over $\ZZ_{(p)}$, parameterizing elliptic curves with with $\Gamma_1(Np^\infty)$-level structure. Given a point $\x\in\Ig(N)\left(\overline{\mathbb{F}}_{p}\right)$, we write $\widehat{S}_\x\hookrightarrow \Ig(N)_{/\sW}$ for the local deformation space of $\x$ over $\sW$ (that is, isomorphism classes of elliptic curves whose mod $p$ representations coincide with that of $\x$). As explained in \cite[\S3.1]{CastellaHsiehGHC}, results of Katz \cite{katz81} on the canonical Serre--Tate coordinates yields the identification
\[
\cO_{\widehat{S}_\x}=\sW[[t-1]].
\]
Consequently, given any $p$-adic modular form $F$  of level $N$ defined over $\sW$, we may evaluate it at the Serre--Tate coordinate $t$ of $\x$:
\begin{equation}
    F(t)=F|_{\widehat{S}_\x}\in\sW[[t-1]]\,,
\label{eq:ST}
\end{equation}
resulting in a $\sW$-valued measure $dF$ on $\Zp$ via the Amice transform:
\begin{equation}
    \int_{\Zp}t^xdF(x)=F(t).
\label{eq:Amice}
\end{equation}
If $\phi:\Zp\lra \cO_{\Cp}$ is any continuous function, we define
\[
F\otimes\phi(t)=\int_{x\in\Zp}\phi(x)t^xdF\in\cO_{\Cp}[[t-1]].
\]

\begin{defn}
Fix a positive integer $c_0$ with $p\nmid c_0$. Let $\fa$ be a prime-to-$c_0\fN p$ integral ideal of $\cO_{c_0}$. This defines a CM point $\x_\fa\in \Ig(N)(\overline{\mathbb{F}}_p)$, as explained in \cite[\S3.2]{CastellaHsiehGHC}. For a $p$-adic modular form $F$ as above, we define the power series
\[
F_\fa(t)=F\left(t^{\frac{1}{N(\fa)\sqrt{-D_K}}}\right)\in\sW[[t-1]]
\]
using the Serre--Tate coordinate $t$ attached to $\x_\fa$. Here $N(\fa)=\#\cO_{c_0}/\fa$.
\end{defn}
{
\begin{defn}
Let ${\displaystyle f(q)=\sum_{n>0}a_n(f)q^n}$ be the $q$-expansion of $f$. The $p$-depletion of $f$, denoted by $f^\flat$, is the modular form of level $Np^2$ whose $q$-expansion is given by ${\displaystyle \sum_{p\,\nmid\, n > 0}a_n(f)q^n}$.
\end{defn}
}

\begin{defn}
Given a continuous function $\rho:\widetilde{\Gamma}_\ac\lra \cO_{\Cp}$, we write $\rho|[\fa]$ for the function $\Zp^\times\lra \cO_{\Cp}$ defined by
\[x\mapsto \rho\left(\rec_\p(x)\sigma_\fa^{-1}\right),\]
where $\rec_\p$ denotes the local reciprocity law $\Qp^\times=K_\p^\times\lra \Gal(K^{\mathrm{ab}}/K)\lra\widetilde{\Gamma}_\ac$ and $\sigma_\fa$ denotes the image of $\fa$ in $\Gal(K^\fa/K)$ under the Artin map with $K^\fa$ being the maximal abelian $\fa$-ramified extension of $K$.
\end{defn}
\begin{defn}\label{defn:BDP}
Let $\psi$ be an anticyclotomic Hecke character of infinity type $(k/2,-k/2)$ whose prime-to-$p$ conductor is $c_0\cO_K$.  The anticyclotomic $p$-adic $L$-function attached to $f$ and $\psi$ is defined by  the $p$-adic measure on $\widetilde{\Gamma}_\ac$:
\[
\sL_{\p,\psi}(f)(\rho)=\sum_{[\fa]\in \Pic\cO_{c_0}}\psi(\fa)N(\fa)^{-k/2}\left(\left(\hat{f^\flat}\right)_\fa\otimes \psi_\p\rho|[\fa]\right),
\]
where $\hat{f^\flat}$ denotes the $p$-adic avatar of $f^\flat$. We may identify it with an element in the Iwasawa algebra $\cO_L\otimes\sW[[\widetilde{\Gamma}_\ac]]$.
\end{defn}
{Note that we may replace $f$ by $f^\lambda$, where $\lambda\in\{\alpha,\beta\}$, and this results in the same $p$-adic $L$-function since the construction goes through the  $p$-depletion of $f$ and $f^\flat=(f^\lambda)^\flat$. }


The $p$-adic $L$-function $\sL_{\p,\psi}(f)$ satisfies the following interpolation properties.
\begin{theorem}\label{thm:interpolationBDP}
Let  $\psi$ be  an anticyclotomic Hecke character of infinity type $(k/2,-k/2)$ with prime-to-$p$ conductor $c_0$. If $\phi$ is a Hecke character factoring through $\widetilde{\Gamma}_\ac$ with infinity type $(m,-m)$ with $m\ge0$ and $\widehat\phi$ its $p$-adic avatar, then
\[
\left(\frac{\sL_{\p,\psi}(f)(\widehat\phi)}{\Omega_p^{k+2m}}\right)^2=\fe(f,\psi\phi)\phi(\fN^{-1})2^{\#A(\psi)+3}c_o\epsilon(f)u_K^2\sqrt{D_K}\frac{L(f_{/K} ,\psi\phi,k/2)}{\Omega_K^{2(k+2m)}},
\]
where  $\fe(f,\psi,\phi)$ is given by
$$
\frac{(k+m-1)!m!}{(4\pi)^{k+2m+1}(\mathrm{Im}\sqrt{-D_K}/2)^{k+2m}}\left(1-a_p(f)p^{-k/2}\psi\phi_{\p^c}(p)+\psi\phi_{\p^c}(p^2)p^{-1}\right)^{2},
$$
$\epsilon(f)$ is the root number of $f$ and $u_K=\#\cO_K^\times/2$.
\end{theorem}
\begin{proof}
This is \cite[Proposition~3.8]{CastellaHsiehGHC}.  {Note that the hypothesis (ST) in op. cit. holds automatically in our setting since we assume the strong Heegner hypothesis, which implies that the set of primes $A(\chi)=\{ q|D_K:\chi_q\text{ unramified and }q|N\}$ is empty. }
\end{proof}

\subsection{Anticyclotomic $p$-adic $L$-functions in families}
\label{subsec_anticyclo_padic_L_in_families}
We construct an anticyclotomic $p$-adic $L$-function for a Coleman family $\f\in {\Lambda_{(k,r_0)}}[[q]]$ passing through a $p$-stabilization $f^\lambda$ of $f$. This $p$-adic $L$-function belongs to $ {\Lambda_{(k,r_0)}}\,\widehat\otimes\sW[[\widetilde{\Gamma}_\ac]]$. The construction we present here is based on Brako\v{c}evi\'{c}'s work for Hida families in \cite{miljan2}. A similar construction is  also indicated in \cite[Theorem~7.3.3]{JLZ}.

{Since Katz' moduli interpretation of $p$-adic modular forms applies to any  $p$-adically complete ring, we may apply the theory of Serre--Tate coordinates  to $\Lambda_{(k,r_0)}$-adic modular forms, allowing us to give the following generalizations of \eqref{eq:ST} and \eqref{eq:Amice}:}
\begin{defn} 
For any $\Lambda_{(k,r_0)}$-adic modular form $\fF$ and a point $\x\in\Ig(N)(\overline{\mathbb{F}}_p)$, we define the power series
\[
\fF(t)=\fF|_{\widehat{S}_\x}\in {\Lambda_{(k,r_0)}}\,\widehat\otimes\sW[[t-1]],
\]
as in Section~\ref{S:padicL-one-var}, which gives rise to a $ {\Lambda_{(k,r_0)}}\otimes \sW$-valued measure $d\fF$ via the Amice transform.
\end{defn}

\begin{remark}
We note that our family of modular forms $\fF$ is a priori an element of the space $G(N,\Lambda_{(k,r_0)})$ of $\Lambda_{(k,r_0)}$-adic modular forms, in the notation of \cite{miljan2}. In the scenario when $\fF$ has slope zero, the containment $G(N,\Lambda_{(k,r_0)})\subset V(N;\cO_L)\widehat{\otimes}_{\cO_L} \Lambda_{(k,r_0)}$ of \cite[Proposition~7.1]{miljan2}  (which Brako\v{c}evi\'{c} explains in the second half of the proof of this proposition) amounts to saying that we may regard $\fF$ as a $\Lambda_{(k,r_0)}$-adic (geometric) modular form (see also \cite[Theorem~3.2.16]{hida00}; note that rather than the whole branch ${\mathbb{I}}$ in op. cit., we work with its restriction $\Lambda_{(k,r_0)}$ to an open disc in the weight space). In particular, it makes sense to evaluate $\fF$ at a point on the Igusa tower and $\fF(t)$ is indeed an element of $\Lambda_{(k,r_0)}\,\widehat\otimes\sW[[t-1]]$. 

When $\fF$ has positive slope, the argument in the second half of the proof of \cite[Proposition~7.1]{miljan2} goes through verbatim, in light of the discussion in the paragraph preceding Equation (6.1) in op. cit., where Brako\v{c}evi\'{c} observes that the ordinarity of $\fF$ is not needed in the construction of the geometric-modular-form valued measures. 
\end{remark}

\begin{defn}
Given any continuous function $\phi:\Zp\lra \cO_{\Cp}$, we define 
\[
\fF\otimes\phi(t)=\int_{\Zp}\phi(x)t^xd\fF\in  {\Lambda_{(k,r_0)}}\,\widehat\otimes\ \cO_{\Cp}[[t-1]]
\]
as before.
\end{defn}

\begin{defn}
Let $\psi$ be an anticyclotomic Hecke character of infinity type $(k/2,-k/2)$ whose prime-to-$p$ conductor is $c_0\cO_K$. Then, there exists a family of anticyclotomic Hecke characters $\Psi$ which admits $\psi$ as its weight $k$ specialization and such that its weight-$m$ specialization $\Psi_m$ is of infinity type $(m/2,-m/2)$. Furthermore, if $m\equiv k\mod (p-1)$, $\Psi_m$ has the same conductor as $\psi$. Otherwise, the finite-type of $\Psi_m$ differ from $\psi$ by powers of the Teichmuller character at $\p$ and $\p^c$. Such a CM family can be constructed in the same way as in \cite[Lemma~4.3.1]{collins}. Namely,
\[
\Psi=\psi \theta^{-k/2}(\theta^c)^{k/2}\cA^{1/2}(\cA^{c})^{-1/2},
\]
where $\theta$ is the Hecke character denoted by $\alpha$ in \cite[\S4.2]{collins} (it is of conductor $\p$ and infinity type $(1,0)$) and $\cA$ is a $\Lambda_{(k,r_0)}$-adic character passing through $\theta$ as given in \cite[\S4.3]{collins}.
\end{defn}

\begin{defn}
We let $\f^\flat$ denote the $p$-depletion of the Coleman family $\f$. We define the $\Lambda_{(k,r_0)}\otimes \sW$-valued measure on $\widetilde{\Gamma}_\ac$ by setting
\[
\sL_{\p,\psi}(\f)(\rho)=\sum_{[\fa]\in \Pic\cO_{c_0}}\Psi(\fa)N(\fa)^{-\frac{\kappa}{2}}\left(\left(\f^\flat\right)_\fa\otimes \Psi_\p\rho|[\fa]\right),
\]
where $\kappa$ is the weight variable and $\rho$ is a character if $\widetilde{\Gamma}_\ac$. 
\end{defn}
\begin{remark}
\label{rem_compare_BDP_of_JLZ}
{ Suppose that $\psi$ is the trivial character (in particular, $c_0=1$). On comparing the interpolation formulae given in Theorem~\ref{thm:interpolationBDP} and \cite[Theorem~7.2.2]{JLZ}, we deduce that $\sL_{\p,\mathbb{1}}(\f)$ agrees with the two variable $p$-adic $L$-function given in Theorem~7.3.3 in op. cit. after multiplying by a  unit in $\cO_L$.}
\end{remark}
\begin{remark}
If $m\equiv k\mod (p-1)$  the weight-$m$ specialization of $\sL_{\p,{\Psi}}(\f)$ coincides with $\sL_{\p,\Psi_m}(\f(m))$, where $\f(m)$ is the weight-$m$ specialization of $\f$. When $m=k$, this is precisely the $p$-adic $L$-function $\sL_{\p,\psi}(f)$ of  Brako\v{c}evi\'{c} and Bertolini--Darmon--Prasanna given in Definition~\ref{defn:BDP}.
\end{remark}


\section{generalized Heegner cycles and a formula of Bertoini--Darmon--Prasanna}
\label{sec:GHC}

In this section, we give an overview of the work of Bertolini--Darmon--Prasanna and its enhancement by Castella and Hsieh  \cite{CastellaHsiehGHC} {and Kobayashi \cite{kobayashiGHC},  which give an explicit relation between generalized Heegner cycles and the $p$-adic $L$-function given in Definition~\ref{defn:BDP}}. Throughout this section $\kappa\in \ZZ_{\geq 2}$ denotes an even integer and $g=\sum_{n\geq 1}a_n(g)q^n\in S_\kappa(\Gamma_0(N))$ is a fixed normalized cuspidal eigenform, new of level $N$ satisfying \textbf{(BI)}. Let $\alpha_g,\beta_g$ denote the roots of the $p$-Hecke polynomial $X^2-a_p(g)X+p^{\kappa-1}$. We assume that the field $L$ is large enough to contain the Hecke field of $g$ as well as $\alpha_g$ and $\beta_g$. Throughout, we fix a choice of $\lambda\in\{\alpha,\beta\}$.  We write $g^\lambda$ for the $p$-stabilization of $g$ on which $U_p$ acts by $\lambda_g$.

\subsection{Deligne's representation, degeneracy maps and duality}
\label{subsec_duality_deligne_poincare}
In this subsection, we review various incarnations of Deligne's representation. 

 Let {$V_g$}  denote the self-dual twist of Deligne's $p$-adic representation attached to $g$ and {$T_g\subset V_g$ the Galois-stable lattice such that $T_g(\frac{\kappa}{2}-1)$ is realized as the $g$-isotypic quotient of $H^1_{\textup{\'et}}(Y_0(N)_{\overline{\QQ}},{\rm TSym}^{\kappa-2}\mathscr{H}_{\ZZ_p})$, where the \'etale sheaf ${\rm TSym}^{\kappa-2}\mathscr{H}_{\ZZ_p}$ is as in \cite{KLZ2}.}
{\begin{remark}
\label{rem_phsical_realizations_of_Deligne}
\item[i)] When it is important to emphasize this dependence, we  write $X_g(\mathscr{H}_{\ZZ_p})$, where $X\in\{T,V\}$, in place of $X_g$. The reason why this becomes relevant is that one may also realize $T_g$ as the $g$-isotypic direct summand of $H^1_{\textup{\'et}}(Y_0(N)_{\overline{\QQ}},{\rm Sym}^{\kappa-2}\mathscr{H}_{\ZZ_p}^\vee)(\frac{\kappa}{2})$,  where ${\rm Sym}^{\kappa-2}\mathscr{H}_{\ZZ_p}^\vee$ is the dual sheaf as in \cite{KLZ2}. When this realization is used and we need to keep track of this choice, we write $X_g(\mathscr{H}_{\ZZ_p}^\vee)$, where $X\in\{T,V\}$,  in place of $X_g$. 
\item[ii)] We note that the classical Heegner Cycles are realized with coefficients in $V_g(\mathscr{H}_{\ZZ_p}^\vee)$, whereas we will interpolate generalized Heegner cycles taking coefficients in $V_g(\mathscr{H}_{\ZZ_p})$. These two spaces are interchanged by the Atkin--Lehner operator $W_N$.
\end{remark}}

{We shall denote by $V_{g^\lambda}$ the $g^\lambda$-isotypic quotient of the cohomology of $H^1_{\textup{\'et}}(Y_0(Np)_{\overline{\QQ}},{\rm TSym}^{\kappa-2}\mathscr{H}_{\QQ_p})(1-\frac{\kappa}{2})$ and similarly define the lattice $T_{g^{\lambda}}$. As in Remark~\ref{rem_phsical_realizations_of_Deligne}, we will denote this representation by $V_{g^\lambda}(\mathscr{H}_{\ZZ_p})$ whenever necessary. Analogously, we also have the realization $V_{g^\lambda}(\mathscr{H}_{\ZZ_p}^\vee)$ as the $g^\lambda$-isotypic direct summand of $H^1_{\textup{\'et}}(Y_0(Np)_{\overline{\QQ}},{\rm Sym}^{\kappa-2}\mathscr{H}_{\QQ_p}^\vee)(\frac{\kappa}{2})$. The Atkin--Lehner operator $W_{Np}$ interchanges $V_{g^\lambda}(\mathscr{H}_{\ZZ_p})$ and $V_{g^\lambda}(\mathscr{H}_{\ZZ_p}^\vee)$.}

\begin{remark}
{In what follows, the shorthand $X_{g}^*(1)$ and $X_{g^\lambda}^*(1)$ will be reserved for the realization of the Galois representations attached to $g$ as $X_{g}(\mathscr{H}_{\ZZ_p}^\vee)$ and $X_{g^\lambda}(\mathscr{H}_{\ZZ_p}^\vee)$, respectively.}
\end{remark}

{There is a natural isomorphism 
$$({\rm pr}^\lambda)^*: V_{g}(\mathscr{H}_{\ZZ_p})\stackrel{\sim}{\lra} V_{g^\lambda}(\mathscr{H}_{\ZZ_p}),$$
given by $({\rm pr}^\lambda)^*={\rm pr}_1^*-\frac{1}{\lambda_g}{\rm pr}_2^*$, where 
$${\rm pr}_i^*: H^1_{\textup{\'et}}(Y_0(N)_{\overline{\QQ}},{\rm TSym}^{\kappa-2}\mathscr{H}_{\QQ_p})\lra H^1_{\textup{\'et}}(Y_0(Np)_{\overline{\QQ}},{\rm TSym}^{\kappa-2}\mathscr{H}_{\QQ_p})$$ 
($i=1,2$) are the degeneracy maps given as in \S2.4 of op. cit. We likewise have a natural isomorphism
$$({\rm pr}^\lambda_\vee)^*: V_{g}(\mathscr{H}_{\ZZ_p}^\vee)\stackrel{\sim}{\lra} V_{g^\lambda}(\mathscr{H}_{\ZZ_p}^\vee)$$
as well as isomorphisms $({\rm pr}^\lambda)_*$ and $({\rm pr}^\lambda_\vee)_*$ defined in the evident manner.} 

{\subsubsection{Poincar\'e duality}
\label{subsubsec_poincare}
There exist natural Poincar\'e duality pairings
$$\{\,,\,\}_{Np}:\,\,H^1_{\textup{\'et}}(Y_0(Np)_{\overline{\QQ}},{\rm TSym}^{\kappa-2}\mathscr{H}_{\QQ_p})\otimes H^1_{\textup{\'et}}(Y_0(Np)_{\overline{\QQ}},{\rm Sym}^{\kappa-2}\mathscr{H}_{\QQ_p}^\vee)\lra \QQ_p,$$
$$\{\,,\,\}_{N}:\,\,H^1_{\textup{\'et}}(Y_0(N)_{\overline{\QQ}},{\rm TSym}^{\kappa-2}\mathscr{H}_{\QQ_p})\otimes H^1_{\textup{\'et}}(Y_0(N)_{\overline{\QQ}},{\rm Sym}^{\kappa-2}\mathscr{H}_{\QQ_p}^\vee)\lra \QQ_p,$$
which induce a commutative diagram
$$\xymatrix@C=.1cm{
V_{g^\lambda}(\mathscr{H}_{\ZZ_p}) \ar[d]_{({\rm pr}^\lambda)_*}&\otimes& V_{g^\lambda}(\mathscr{H}_{\ZZ_p}^\vee)\ar[rrrrr]^(.6){\{\,,\,\}_{Np}}&&&&& L(1)\ar@{=}[d]\\
V_{g}(\mathscr{H}_{\ZZ_p}) &\otimes& V_{g}(\mathscr{H}_{\ZZ_p}^\vee)\ar[u]_{({\rm pr}^\lambda_\vee)^*}\ar[rrrrr]_(.6){\{\,,\,\}_{N}}&&&&& L(1).
}$$
Moreover, Atkin--Lehner operators give rise to the isomorphisms
$$W_{Np}:\,\,V_{g^\lambda}(\mathscr{H}_{\ZZ_p})\stackrel{\sim}{\lra} V_{g^\lambda}(\mathscr{H}_{\ZZ_p}^\vee),$$
$$W_{N}:\,\,V_{g}(\mathscr{H}_{\ZZ_p})\stackrel{\sim}{\lra} V_{g}(\mathscr{H}_{\ZZ_p}^\vee
).$$
This combined with the Poincar\'e duality diagram yields
$$\xymatrix@C=.1cm{
V_{g^\lambda}(\mathscr{H}_{\ZZ_p}) \ar[d]_{({\rm pr}^\lambda)_*}&\otimes& 
V_{g^\lambda}(\mathscr{H}_{\ZZ_p}) \ar[rrrrr]^{{\rm id}\otimes W_{Np}}&&&&&V_{g^\lambda}(\mathscr{H}_{\ZZ_p}) \ar[d]_{({\rm pr}^\lambda)_*}&\otimes& V_{g^\lambda}(\mathscr{H}_{\ZZ_p}^\vee)\ar[rrrrr]^(.6){\{\,,\,\}_{Np}}&&&&& L(1)\ar@{=}[d]\\
V_{g}(\mathscr{H}_{\ZZ_p}) &\otimes& 
V_{g}(\mathscr{H}_{\ZZ_p}) \ar[u]_{({\rm pr}^\lambda)^*}\ar[rrrrr]_{{\rm id}\otimes W_{N}}&&&&&V_{g}(\mathscr{H}_{\ZZ_p}) &\otimes& V_{g}(\mathscr{H}_{\ZZ_p}^\vee)\ar[u]_{({\rm pr}^\lambda_\vee)^*}\ar[rrrrr]_(.6){\{\,,\,\}_{N}}&&&&& L(1),
}$$ 
which simplifies to the following self-duality diagram:
\begin{equation}
\label{eqn_xymatrix_self_duality}
    \xymatrix@C=.1cm{
V_{g^\lambda}(\mathscr{H}_{\ZZ_p}) \ar[d]_{({\rm pr}^\lambda)_*}&\otimes& 
V_{g^\lambda}(\mathscr{H}_{\ZZ_p})\ar[rrrrrrrr]^(.6){\{\,\circ\,,\,\bullet\,\}_{Np}^\prime}&&&&&&&& L(1)\ar@{=}[d]\\
V_{g}(\mathscr{H}_{\ZZ_p}) &\otimes& 
V_{g}(\mathscr{H}_{\ZZ_p}) \ar[u]_{({\rm pr}^\lambda)^*}\ar[rrrrrrrr]_(.6){\langle\,\circ\,,\,\bullet\,\rangle_{N}^\prime}&&&&&&&& L(1).
}
\end{equation}
Here, we have put $\langle\,\circ\,,\,\bullet\,\rangle_{Np}^\prime:=\{\,\circ\,,\,W_{Np}(\bullet)\,\}_{Np}$ and $\langle\,\circ\,,\,\bullet\,\rangle_{N}^\prime:=\{\,\circ\,,\,W_{N}(\bullet)\,\}_{N}$. It is customary to normalize the duality in the bottom row using the Atkin--Lehner involution $\lambda_N(g)^{-1}W_N$ in place of $W_N$. Concretely, one defines $\langle\,\circ\,,\,\bullet\,\rangle_{N}^\prime:=\{\,\circ\,,\,\lambda_N(g)^{-1}W_{N}(\bullet)\,\}_{N}$ so that the self-duality diagram takes the form
\begin{equation}
\label{eqn_xymatrix_self_duality_2}
    \xymatrix@C=.1cm{
V_{g^\lambda}(\mathscr{H}_{\ZZ_p}) \ar[d]_{({\rm pr}^\lambda)_*}&\otimes& 
V_{g^\lambda}(\mathscr{H}_{\ZZ_p})\ar[rrrrrrrr]^(.6){\{\,\circ\,,\,\bullet\,\}_{Np}^\prime}&&&&&&&& L(1)\\
V_{g}(\mathscr{H}_{\ZZ_p}) &\otimes& 
V_{g}(\mathscr{H}_{\ZZ_p}) \ar[u]_{({\rm pr}^\lambda)^*}\ar[rrrrrrrr]_(.6){\langle\,\circ\,,\,\bullet\,\rangle_{N}}&&&&&&&& L(1).\ar[u]_{\lambda_N(g)}
}
\end{equation}}

{\begin{lemma}
\label{lemma_KLZ_2_Prop_10_1_1_last_par}
The following diagram commutes:
$$    \xymatrix@C=.1cm{
V_{g^\lambda}(\mathscr{H}_{\ZZ_p}) &\otimes& 
V_{g^\lambda}(\mathscr{H}_{\ZZ_p})\ar[rrrrrrrr]^(.6){\{\,\circ\,,\,\bullet\,\}_{Np}^\prime}&&&&&&&& L(1)\\
V_{g}(\mathscr{H}_{\ZZ_p}) \ar[u]_{({\rm pr}^\lambda)^*}&\otimes& 
V_{g}(\mathscr{H}_{\ZZ_p}) \ar[u]_{({\rm pr}^\lambda)^*}\ar[rrrrrrrr]_(.6){\langle\,\circ\,,\,\bullet\,\rangle_{N}}&&&&&&&& L(1).\ar[u]_{\lambda_g\lambda_N(g) \cE(g^\lambda)\cE^*(g^\lambda)}
}$$
Here, $\cE(g^\lambda)=\left(1-\frac{p^{\kappa-2}}{\lambda_g^2}\right)$, $\cE^*(g^\lambda)=\left(1-\frac{p^{\kappa-1}}{\lambda_g^2}\right)$.
\end{lemma}
\begin{proof}
This is a restatement of the discussion in the final paragraph of \cite[Proposition 10.1.1]{KLZ2}.
\end{proof}
}

\begin{defn}
Let $\iota: \widetilde{\Gamma}_\ac\to \widetilde{\Gamma}_\ac$ denote the involution given by $\gamma\mapsto \gamma^{-1}$. We let $\ZZ_p[[\widetilde{\Gamma}_\ac]]^\iota$ denote the free $\ZZ_p[[\widetilde{\Gamma}_\ac]]$-module of rank-one on which $G_K$ acts via the composition $G_K\twoheadrightarrow \widetilde{\Gamma}_\ac\stackrel{\iota}{\lra}\widetilde{\Gamma}_\ac$. We put $\TT_g^\ac:=T_g{\otimes}_{\ZZ_p}\ZZ_p[[\widetilde{\Gamma}_\ac]]^\iota$. {We similarly define $\TT_{g^\lambda}^\ac$ for each $\lambda_g\in \{\alpha_g,\beta_g\}$.}
\end{defn}

\subsection{Generalized Heegner cycles} {We now discuss the definition of generalized Heegner cycles.}
Fix a positive integer $c_0$ coprime to $pN$ and a non-negative integer $s$. Set $c:=c_0p^s$. We denote by $K_c$ the ring class extension modulo $c$. For any locally algebraic anticyclotomic character {$\chi$ of conductor $c$, let us write $\chi:\Gal(K_{c_0p^\infty}/K)\rightarrow \Cp^\times$ for the Galois character associated to the $p$-adic avatar of $\chi$ by Class Field Theory (we refer the reader to \cite[Definition~3.4]{CastellaHsiehGHC} for a definition of $p$-adic avatars). When the infinity-type $(j,-j)$ of $\chi$ verifies $-\frac{\kappa}{2}<j<\frac{\kappa}{2}$,} the main construction of Bertolini--Darmon--Prasanna~\cite{bertolinidarmonprasanna13} gives rise to a collection of cohomology classes 
$$z_{g,\chi,c}\in H^1(K_{c},T_{g}\otimes{\chi})$$
(see, for example, \cite[(4.6)]{CastellaHsiehGHC}). As in (4.7) of op. cit., we define the \emph{generalized Heegner class} $z_{g,\chi}\in H^1(K,T_{g}\otimes{\chi})$ by setting $$z_{g,\chi}:={\rm{cor}}_{K_{c}/K}\left(z_{g,\chi,c}\right)\,.$$ 
\begin{defn}
\label{def_various_Heeg_j_stabilized}
\item[i)]{Let $\psi_\p$ and $\psi_{\p^c}$ be the canonical characters $G_K\rightarrow \Zp^\times$ as defined in \cite[\S2.3]{kobayashiGHC}.} {The infinity-types of the associated Hecke characters are $(1,0)$ and $(0,1)$ respectively and their conductor is $\p$ and $\p^c$, respectively.   }
\item[ii)] We write 
\[
T_g\langle j\rangle:= T_g(\psi_\p^j\psi_{\p^c}^{-j}), \quad V_g\langle j\rangle=V_g(\psi_\p^j\psi_{\p^c}^{-j})
\]
for $-\frac{\kappa}{2}<j<\frac{\kappa}{2}$ and denote the corresponding class by
$$z_{g,c}^{(j)}:=z_{g,\psi_\p^j\psi_{\p^c}^{-j},c}\in H^1(K_{c},T_{g}\langle j\rangle).$$ 
\item[iii)]  For  $s\ge1$, we define the $p$-stabilized class
\[
z_{g,c,\lambda_g}^{(j)}=z_{g,c}^{(j)}-p^{\kappa-2}\lambda_g^{-1}\res_{K_{c}/K_{cp^{-1}}}\left(z_{g,cp^{-1}}^{(j)}\right)\in \dfrac{1}{p}H^1(K_{c},T_{g}\langle j\rangle) \subset H^1(K_{c},V_{g}\langle j\rangle).
\]
\item[iv)] { Let us put $\cE_p(g^\lambda,c_0):=\frac{1}{|\cO_{c_0}^\times|}\left(1-\frac{p^{\frac{\kappa}{2}-1}}{\lambda_g}\sigma_\p\right)\left(1-\frac{p^{\frac{\kappa}{2}-1}}{\lambda_g}\sigma_{\p^c}\right)$, where $\sigma_\p$ and $\sigma_{\p^c}$ denote the Frobenii at $\p$ and $\p^c$ respectively. For $s=0$, we define 
$$z_{g,c_0,\lambda_g}^{(j)}:=\cE_p(g^\lambda,c_0)z_{g,c_0}^{(j)}\in \frac{1}{|\cO_{c_0}^\times|p^\kappa} H^1(K_{c},T_{g}\langle j\rangle) \subset H^1(K_{c},V_{g}\langle j\rangle).$$
\item[v)] When $j=0$, we omit the superscript $(j)$ from the  notation.
}
\end{defn}

{For  $s> 1$ and $-\frac{\kappa}{2}<j<\frac{\kappa}{2}$}, these elements  verify 
\begin{align}\label{eq:norm}
&\cor_{K_{c}/K_{cp^{-1}}}\left(z_{g,c,\lambda_g}^{(j)}\right)=\lambda_g\cdot z_{g,cp^{-1},\lambda_g}^{(j)}\\
\label{eq:cong}
&\sum_{j=0}^{\kappa-2}(-1)^j\binom{\kappa-2}{j}\Tw_{\frac{\kappa}{2}-j-1}\circ\res_{K_{c_0p^\infty}/K_{c}}\left(z_{g,c,\lambda_g}^{(j-\frac{\kappa}{2}+1)}\right)\in p^{s(\kappa-2)-C} H^1(K_{c_0p^\infty},T_g),
\end{align}
where 
$$\Tw_{\frac{\kappa}{2}-j-1}: H^1(K_{c_0p^\infty}, T_g\langle j-\frac{\kappa}{2}+1\rangle)\lra H^1(K_{c_0p^\infty}, T_g)$$  is the twisting map and $C$ is a constant which depends only on $v(\lambda_g)$.
{See  \cite[Equations (4.31) and (4.35)]{kobayashiota} as well as  \cite[proof of Proposition~5.3.1]{JLZ} where these relations are discussed}. The properties \eqref{eq:norm} and \eqref{eq:cong} allow one to interpolate $\{z_{g,c,\lambda_g}^{(j)}\}$ along the anticyclotomic tower. We refer the reader to \cite[\S8]{kobayashiGHC}  for a detailed discussion.

\begin{theorem}[Kobayashi]\label{thm_kobayashi_ota_ac_interpolation}
For every integer $c_0$ as before and $\lambda_g\in \{\alpha_g,\beta_g\}$ with $v(\lambda_g)<\kappa-1$, there exists a unique element 
$${\bz_{g,c_0,\lambda_g}^{\rm ac}}\in \frac{1}{p^C}H^1(K_{c_0},\TT_g^\ac)\,\widehat\otimes\, \cH_{v(\lambda_g)}^+({\widetilde{\Gamma}_\ac}),$$
where $C$ is some integer that depends only on $v(\lambda_g)$,
 such that its image under the composite map $\pi_{c_0,j}$
$$\pi_{c_0,j}: H^1(K_{c_0},\TT_g^\ac)\,\widehat\otimes\, \cH_{v(\lambda_g)}(\widetilde{\Gamma}_\ac)\lra H^1(K_{c_0},\TT_g^\ac\langle j\rangle)\,\widehat\otimes\, \cH_{v(\lambda_g)}(\widetilde{\Gamma}_\ac)\lra H^1(K_{c_0p^s},V_g\langle j\rangle)$$
equals $\lambda^{-s}_gz_{g,c_0p^s,\lambda_g}^{(j)}$.
\end{theorem}
\begin{proof}
Under the hypothesis \textbf{(BI)}, \cite[Lemma~3.10]{LV17} tells us that $H^0(K_{c_0p^\infty},T_g)=H^0(K_{c_0p^\infty},\overline{T_g})=0$, where $\overline{T_g}$ denotes the residual representation of $T_g$. By the properties \eqref{eq:norm} and \eqref{eq:cong},   the classes $\lambda^{-s}_gz_{g,c_0p^s,\lambda_g}^{(j)}$ are interpolated by a class 
$${\bz_{g,c_0,\lambda_g}^{\ac}}\in  \left(\frac{1}{p^C} H^1(K_{c_0p^\infty},\TT_g^\ac\,\widehat\otimes\, \cH_{v(\lambda_g)}^+(\Gamma_\ac))\right)^{\Gal(K_{c_0p^\infty}/K_{c_0})}=\frac{1}{p^C} H^1(K_{c_0},\TT_g^\ac)\widehat\otimes\, \cH_{v(\lambda_g)}^+(\Gamma_\ac)$$ 
{ by the integral Perrin-Riou twists developed by Kobayashi and \cite[Propositions~II.2.3 and II.3.1]{colmez98}. }
\end{proof}

\begin{defn}\label{defn_chi_part_of_ac_big_GHC}
Given a {ring class} character $\chi$ of conductor $c_0p^n$, we let 
$\bz_{g,c_0,\lambda_g}^\chi \in H^1(K_{c_0},V_g\otimes \chi)$ denote the image of ${\bz_{g,c_0,\lambda_g}^\ac}$ under the natural specialization morphism and set 
$$\bz_{g,\lambda_g}^\chi:={\rm cor}_{K_{c_0}/K}(\bz_{g,c_0,\lambda_g}^\chi)\in H^1(K,V_g\otimes \chi)\,.$$
\end{defn}

\begin{lemma}
\label{lemma_compare_GHCs_to_deformed_GHCs}
{If $\chi$ is a  ring class character  of conductor $c_0p^n$, with $n>0$, then 
$$\bz_{g,c_0,\lambda_g}^\chi=\lambda_g^{-n}z_{g,\chi,c_0}$$ 
 Consequently, $\bz_{g,\lambda_g}^\chi=\lambda_g^{-n}z_{g,\chi}\,.$ When $n=0$, we have
 \[
 \bz_{g,c_0,\lambda_g}^\chi=z_{g,\chi,c_0,\lambda_g}=\cE_p(g^\lambda,c_0)z_{g,\chi,c_0}\, ,
 \]}
{where $\cE_p(g^\lambda,c_0)$ is given in Definition~\ref{def_various_Heeg_j_stabilized}(iv).}
\end{lemma}

\begin{proof}
The assertions in this lemma follow from the general twisting formalism~\cite[Lemma 2.4.3]{rubin00} and their proofs follow in a manner identical to that of \cite[Lemma 5.4]{CastellaHsiehGHC}. 
\end{proof}

\begin{defn}
\item[i)] Let  $\psi$ be  an anticyclotomic Hecke character of infinity type $(\frac{\kappa}{2},-\frac{\kappa}{2})$ with prime-to-$p$ conductor $c_0$. We set $\cL_g\in {\cO_L\otimes\sW[[\widetilde{\Gamma}_\ac]]}$ to denote the twist of $\sL_{\p,\psi}(g)$ by the character $\widehat{\psi}^{-1}$. More precisely,  it is characterized by 
$$\cL_g(\chi)=\sL_{\p,\psi}(g)(\widehat{\psi}^{-1}\chi),$$ for any anticyclotomic character $\chi$. 
When we wish to emphasize the dependence of $\cL_g$ to $\psi$, we shall denote it by $\cL_{g,\psi}$.

Similarly, if $\f$ is a Coleman family as in \S\ref{subsec_anticyclo_padic_L_in_families}, we define $\cL_\f=\cL_{\f,\psi}$ to be the twist of $\sL_{\p,\psi}(\f)$ by $\widehat{\psi}^{-1}$.
\item[ii)] We define the equivariant Bertolini--Darmon--Prasanna $p$-adic $L$-function 
\begin{align*}
\cL_{g,c_0}&:=\sum_{\widehat\rho\in \widehat{\Gal(K_{c_0}/K)}}e_{\widehat{\rho}}\cdot\cL_{g^\rho}\\
&=\sum_{\widehat\rho\in \widehat{\Gal(K_{c_0}/K)}}e_{\widehat{\rho}}\cdot\cL_{g,\psi\rho}\in \ZZ_p[{\Gal(K_{c_0}/K)}]\otimes_{\ZZ_p}\cO_L\otimes_{}\sW[[\widetilde{\Gamma}_\ac]]\,,
\end{align*}
where $g^\rho:=g_{/K}\otimes \rho$ for each anticyclotomic Hecke character $\rho$ of finite order and conductor $c_0$, and ${\displaystyle e_{\widehat\rho}:=\sum_{\delta\in \Gal(K_{c_0}/K)}\widehat\rho(\delta)\delta^{-1}}$.
\end{defn}
\begin{remark}
\item[i)] There is a bijection $($induced by the Artin map of global class field theory$)$ between anticyclotomic Hecke characters of finite order and conductor $c_0$ and $p$-adic characters of $\Gal(K_{c_0}/K)$. 
\item[ii)] Given an anticyclotomic Hecke character $\psi$ of infinity type $(\frac{\kappa}{2},-\frac{\kappa}{2})$ with prime-to-$p$ conductor $c_0$, note that $\rho\psi$ runs through the set of anticyclotomic Hecke characters of infinity type $(\frac{\kappa}{2},-\frac{\kappa}{2})$ and conductor $c_0$ as $\rho$ runs through anticyclotomic Hecke characters of finite order and conductor $c_0$.
\end{remark}

{\begin{defn} \label{defn_our_omega_eta}
{Let $t$ be the uniformizer of Fontaine's ring $\mathbb{B}_\mathrm{dR}$ attached a system of compatible primitive $p$-power roots of unity $\epsilon=\{\zeta_{p^n}\}_{n\ge1}$. In what follows, we shall write $t^{-1}$ for the non-zero element $\epsilon\otimes t^{-1}$ of $\Dcris(\Qp(1))$. }
\label{def_omega_eta}
\item[i)]  We let $\omega_g\in {\rm Fil^1}{\DD_{\rm cris}(T_g(-\frac{\kappa}{2}))}$ be the vector given as in \cite[Corollary~2.3]{bertolinidarmonprasanna13}\,.
We set $\omega_{g^\lambda}:=({\rm pr}^\lambda)^*(\omega_g)\in {\rm Fil^1}\Dcris(V_{g^{\lambda}}(-\frac{\kappa}{2}))$.
\item[ii)] We let $\eta_g \in {\DD_{\rm cris}(V_g({\frac{\kappa}{2}-1}))}/{\rm Fil}^{{0}}  {\DD_{\rm cris}(V_g({\frac{\kappa}{2}-1}))}$ denote the unique vector such that $\eta_g\otimes t^{{\frac{\kappa}{2}-1}}$ pairs to $1$ with  $\omega_g\otimes t^{{-\frac{\kappa}{2}}}$ under the canonical pairing 
$${\DD_{\rm cris}(V_g)}\otimes {\DD_{\rm cris}(V_g)}\stackrel{\{\,,\,\}_{N}}{\lra} \DD_{\rm cris}(L(1)) \stackrel{\otimes t}{\lra} L$$
induced from the Poincar\'e duality we have introduced in Section~\ref{subsubsec_poincare}.
\item[iii)] We let $\eta_{g^\lambda} \in \Dcris(T_{g^\lambda}({\frac{\kappa}{2}-1}))^{\vp=\lambda_gp^{1-\kappa}}$ denote the unique $\vp$-eigenvector such that $\eta_{g^\lambda}\otimes {t^{\frac{\kappa}{2}-1}}$ pairs with $\omega_{g^\lambda}\otimes t^{{-\frac{\kappa}{2}}}$  to $1$ under the natural pairing induced from  the Poincar\'e duality.
\end{defn}}

{
\begin{remark}\label{rem_compare_eta_g_to_eta_lambda_g}
Only in this remark, we shall also work with the conjugate modular form $g^*$. Under our assumption that $g$ has trivial nebentype, this of course is identical to the original eigenform $g$. However, in order to compare the vectors we have introduced in Definition~\ref{defn_our_omega_eta} with the corresponding objects in \cite{KLZ2}, we will work with $g^*$ wherever appropriate. In the notation of \cite{KLZ2}, we note that $T_{g^*}(-\kappa/2)=M_L(g^*)$, whereas $T_{g^*}(\kappa/2-1)=M_L(g^*)^*$. 
\item[i)] In \cite[Proposition 10.1.1]{KLZ2}, the authors treat $\omega_{g^*}\in {\rm Fil}^1\Dcris(M_L(g^*))$ as a morphism 
$$\omega_{g^*}^{\rm KLZ}: \Dcris(M_L(g^*)^*(1-\kappa))\lra L\,.$$
The morphism $\omega_{g^*}^{\rm KLZ}$ is given as the composition of the arrows
$$\Dcris(M_L(g^*)^*(1-\kappa))\stackrel{\otimes t^{1-\kappa}}{\lra} \Dcris(M_L(g^*)^*)\stackrel{\omega_{g^*}}{\lra} L$$
where the final arrow is the pairing with $\omega_{g^*}\in {\rm Fil}^1\Dcris(M_L(g^*))$. In other words, we have
$$\omega_{g^*}^{\rm KLZ}=\omega_{g^*}\otimes t^{1-\kappa}$$
as elements of $\Dcris(M_L(g^*)(\kappa-1))=\Dcris(M_L(g)^*)$, where the final identification is via the canonical isomorphism $M_L(g^*)(\kappa-1)\stackrel{\sim}{\lra} M_L(g)^*$.
\item[ii)] Let us put 
\begin{align*}
    \eta_{g^\lambda}^{\rm KLZ}:=\eta_{g^\lambda}\otimes t^{\kappa-1} \in \Dcris(M_L(g^*)^*)^{\vp=\lambda_gp^{1-\kappa}}\otimes \Dcris(\QQ_p(1-\kappa))&=\Dcris(M_L(g^*)^*(1-\kappa))^{\vp=\lambda_g}\\
    &=\Dcris(M_L(g))^{\vp=\lambda_g}
\end{align*}
where the final identification is again via the canonical isomorphism $M_L(g^*)(\kappa-1)\stackrel{\sim}{\lra} M_L(g)^*$. By the definition of $\eta_{g^\lambda}$, the vector $\eta_{g^\lambda}^{\rm KLZ}$ pairs with $\omega_{g^{*,\lambda}}^{\rm KLZ}$ to $1$, under the de Rham pairing induced from the Poincar\'e duality $\{\,,\,\}_{Np}$ on the cohomology of the modular curve $Y_0(Np)$ of level $Np$. 
\item[iii)] Building on the discussion in \cite[\S10.1]{KLZ2}, Loeffler and Zerbes explain in \cite[Corollary 6.4.3]{LZ1} that $\eta_{g^\lambda}^{\rm KLZ}$ is the specialization of a vector $\eta_\f^{\rm LZ}$, where $\f$ is the unique Coleman family through the $p$-stabilized eigenform $g^\lambda$. See also Remark~\ref{rem_appendix_explain_C_again} for a detailed discussion concerning $\eta_\f^{\rm LZ}$. In \cite{LZ1}, the vector $\eta_\f^{\rm LZ}$ is plainly denoted by $\eta_\f$. In the present work, the notation $\eta_\f$ is reserved for an alteration of $\eta_\f^{\rm LZ}$, see Definition~\ref{defn_big_omega_eta}.
\end{remark}}

{We next introduce the avatars of the generalized Heegner cycles that we realize in the cohomology of level $Np$ modular curves. They will be denoted by $\frak{z}$ (decorated suitably), as opposed to $z$ or ${\bz}$ which we used above to notate generalized Heegner cycles that are realized in the cohomology of modular curves of level $N$. Recall also that $T_{g^\lambda}$ stands for the lattice $T_{g^\lambda}(\mathscr{H}_{\ZZ_p})$, which is realized as the $g^\lambda$-isotypic quotient of $H^1_{\textup{\'et}}(Y_0(Np)_{\overline{\QQ}},{\rm TSym}^{\kappa-2}(\mathscr{H}_{\ZZ_p}))$.}

{\begin{defn}
\label{defn_pstabilizedGHC_levelNp} Let $c_0$ be a positive integer coprime to $p$.
\item[i)] We let ${\frak{z}}_{g^\lambda,c_0}^\ac\in p^{-C}H^1(K_{c_0},\TT_{g^\lambda}^\ac)\red{\,\otimes_{\LL(\widetilde\Gamma_\ac)}\, \cH_{v(\lambda_g)}^+({\widetilde{\Gamma}_\ac})}$ denote the unique class which verifies 
$$({\rm pr}^\lambda)_*({\frak{z}}_{g^\lambda,c_0}^\ac)=\bz_{g,c_0,\lambda_g}^\ac\,.$$
\item[ii)] For any ring class character $\chi$ of conductor $c_0p^n$, we similarly define ${\frak{z}}_{g^\lambda,c_0}^\chi\in H^1(K_{c_0},V_{g^\lambda}\otimes\chi)$ as the unique class which verifies 
$$({\rm pr}^\lambda)_*({\frak{z}}_{g^\lambda,c_0}^\chi)=\bz_{g,c_0,\lambda_g}^\chi\,.$$
We also put ${\frak{z}}_{g^\lambda}^\chi:={\rm cor}_{K_{c_0}/K}({\frak{z}}_{g^\lambda,c_0}^\chi)\in H^1(K,V_{g^\lambda}\otimes \chi)$.
\item[iii)] We write ${\frak{z}}_{g^\lambda}$ in place of ${\frak{z}}_{g^\lambda,1}^\mathds{1}$ and ${\frak{z}}_{g^\lambda}^\mathds{1}$.
\end{defn}}

{\begin{remark}
\label{rem_LZ_pstabilization}
The $p$-stabilized class ${\frak{z}}_{g^\lambda}^\chi\in H^1(K,V_{g^\lambda}\otimes \chi)$ corresponds to the class denoted by $z_{\textup{\'et}}^{[g_{\lambda_g},\chi]}$ in \cite[\S3.5]{JLZ}; see also Proposition~3.5.2 in op. cit. We explain in Lemma~\ref{lemma_characterize_pstab} how to recover it from the non-$p$-stabilized class when $\chi=\mathds{1}$.
\end{remark}}

We have a map (semi-local Perrin-Riou big exponential map along the anticyclotomic tower) 
$$\Omega^{\varepsilon,{\p}}_{T_{g^\lambda}, \frac{\kappa}{2},c_0}: \DD_{\rm cris}(T_{g^\lambda})\otimes_{\ZZ_p}\sW[\Gal(K_{c_0}/K)][[\widetilde{\Gamma}_\ac]]\lra H^1(K_{c_0,\p},\TT^\ac_{g^\lambda})\widehat\otimes p^{c(\lambda_g)}\cH^+_{v(\lambda_g)}(\widetilde{\Gamma}_\ac)_\sW$$
by Zhang~\cite{zhangLT}. We review this in Appendix~\ref{sec:biglogalaonganticyclotower} (see Section~\ref{subsec_semilocalPR} for their semi-local versions). We refer readers to \eqref{eqn_PR_big_Exp_LT} and \eqref{eq:Xi} in the appendix where we outline the interpolative property which characterizes this map. {Note that the map we are considering here is the one attached to the tower of Lubin--Tate extensions of $K_\p$ given by the localization of $K_{p^\infty}$ at a prime above $\p$ and $\varepsilon$ denotes the canonical norm compatible uniformizers in this tower given as in \cite[\S4.7]{kobayashiGHC}. } 

\begin{theorem}[Kobayashi]
\label{thm_GHC_rec_law_g} {We have
$$\Omega^{\varepsilon,{\p}}_{T_{g^\lambda}, \frac{\kappa}{2},c_0}\left(\left(\eta_{g^\lambda}\otimes t^{\frac{\kappa}{2}-1}\right)\otimes \cL_{g,c_0}^\iota\right)=\res_\p\left({\frak{z}}_{g^\lambda,c_0}^\ac\right),$$
where $\iota$ is the involution on $\sW[\Gal(K_{c_0}/K)][[\widetilde{\Gamma}_\ac]]$ sending elements of $\widetilde{\Gamma}_\ac$ and $\Gal(K_{c_0}/K)$ to their inverses.}
\end{theorem}
\begin{proof}
{Let us write 
$$\mathscr{L}^{ac}_{\p,\lambda_g}(g,A_{c_0};\pmb{e}_{\overline{\pmb{v}}})\in  \DD_{\rm cris}(T_{g^\lambda})^{{\vp=\lambda_g p^{-\frac{\kappa}{2}}}}\otimes_{\ZZ_p}\sW[\Gal(K_{c_0}/K)][[\widetilde{\Gamma}_\ac]]$$
for the $p$-adic $L$-function defined in \cite[\S5.2]{kobayashiGHC}, where we have taken $\alpha$ to be $\lambda_g$, the vector $\overline{\pmb{v}}$ is as in Theorem~5.7 of op. cit. and $\{\pmb{e}_{\overline{\pmb{v}}}\}$ is a basis of the kernel of $\psi$ inside the Robba ring with coefficients in  $\sW$, given as in \S5.1 (just before the statement of Proposition 5.4) in op. cit.} 

{
We recall the element $\omega_{A,{}^t\overline{v}}^\vee\xi_{A,{}^t\overline{v}}^\vee \in \Dcris(\ZZ_p(1))$ in op. cit., which is the image of 
$$\omega_{A,{}^t\overline{v}}^\vee\otimes \xi_{A,{}^t\overline{v}}^\vee \in \Dcris(T_\p(A)^\vee)^\vee \otimes \Dcris(T_{\p^*}(A)^\vee)^\vee$$ 
introduced in \cite[\S4.3]{kobayashiGHC}, under the map induced from the de Rham pairing (which in turn is induced from the Weil pairing $T_{\p}(A)\otimes T_{\p^*}(A) \to \ZZ_p(1)$)
$$ \Dcris(T_\p(A)^\vee)^\vee \otimes \Dcris(T_{\p^*}(A)^\vee)^\vee\lra \Dcris(\ZZ_p(-1))^\vee=
\Dcris(\ZZ_p(1))\,.$$
Then $\omega_{A,{}^t\overline{v}}^\vee\xi_{A,{}^t\overline{v}}^\vee$ corresponds to the canonical basis of $\Dcris(\ZZ_p(1))$, denoted by $t^{-1}$ in Definition~\ref{defn_our_omega_eta}.}

{Theorem~5.7 of op. cit. tells us that the image of the element $\mathscr{L}^{ac}_{\p,\lambda_g}(g,A_{c_0};\pmb{e}_{\overline{\pmb{v}}})\otimes \left(\omega_{g^\lambda}\otimes t^{{-\frac{\kappa}{2}}}\right)$ under the  pairing which extends $\sW[\Gal(K_{c_0}/K)][[\widetilde{\Gamma}_\ac]]$-linearly the natural pairing 
$$\Dcris(V_{g^\lambda})\otimes \Dcris(V_{g^\lambda}) \lra \Dcris(L(1))\lra L$$
is equal to $\cL_{g,c_0}^\iota$. We note that the pairing considered in op. cit. is the one extending linearly  $\Dcris(V_{g^\lambda})\otimes \Dcris(V_{g^\lambda}(-1))\lra L$. This is why in op. cit., $\mathscr{L}^{ac}_{\p,\lambda_g}(g,A_{c_0};\pmb{e}_{\overline{\pmb{v}}})$ is paired with $\omega_{g^\lambda}\otimes t^{-\kappa/2+1}\in \Dcris(V_{g^\lambda}(-1))$, instead of $\omega_{g^\lambda}\otimes t^{-\kappa/2}\in \Dcris(V_{g^\lambda})$.
Since $\eta_{g^\lambda}\otimes t^{\frac{\kappa}{2}-1}$ is a basis of $\DD_{\rm cris}(V_{g^\lambda})^{\vp=\lambda_g p^{-\frac{\kappa}{2}}}$  and it pairs with $\omega_g\otimes t^{-\frac{\kappa}{2}}$ to 1, we have the equality
\[
\left(\eta_{g^\lambda}\otimes t^{\frac{\kappa}{2}-1}\right)\otimes \cL_{g,c_0}^\iota=\mathscr{L}^{ac}_{\p,\lambda_g}(g,A_{c_0};e_v).
\]
The proof now follows from \cite[Lemma~8.2]{kobayashiGHC}.
}
\end{proof}

{We end this section with a description of ${\frak{z}}_{g^\lambda}$ in terms of classical Heegner cycles. This will be useful in the main calculations of \cite{BPSI}. It is also useful when comparing our main interpolation results (Theorem~\ref{thm_main_GHC_in_families_with_ac} below) with those established in \cite{ota2020} in the context of slope-zero families. For simplicity, we will treat for most part the important case $\chi=\mathds{1}$ and $c_0=1$.}

{In what follow,  $V_{g}^*(1)$ stands for the homological Galois representation $V_{g}(\mathscr{H}_{\ZZ_p}^\vee)$, which is realized in $H^1_{\textup{\' et}}(Y_0(N)_{\overline{\QQ}},{\rm Sym}^{\kappa-2}(\mathscr{H}_{\ZZ_p}^{\vee}))(\kappa/2)$. Recall also that the Atkin--Lehner involution $\lambda_N(g)^{-1}W_N$ gives rise to an isomorphism $V_g^*(1)\stackrel{\sim}{\to} V_g$.} 
{\begin{defn}
\label{defn_classical_heegner}
We let ${z}_{g,c_0} \in H^1(K_{c_0},V_g^*(1))$ be the classical Heegner class $($given as in \cite[\S 3]{NekovarGZ}$)$ of conductor $c_0$. We also put ${z}_{g}:={\rm cor}_{K_1/K} \left({z}_{g,1}\right) \in H^1(K,V_{g}^*)$\,.
\end{defn}}

{
\begin{lemma}
\label{lemma_characterize_pstab}
\item[i)] ${z}_{g,\mathds{1},c_0}=\dfrac{\lambda_N(g)^{-1}W_N({z}_{g,c_0})}{(2c_0\sqrt{-D_K})^{\frac{\kappa}{2}-1}}$.
\item[ii)] ${\bf z}_{g,\lambda_g}^{\mathds{1}}={\left(1-\dfrac{p^{\frac{\kappa}{2}-1}}{\lambda_g}\right)^2}\dfrac{u_K^{-1}}{(2\sqrt{-D_K})^{\frac{\kappa}{2}-1}}\cdot \lambda_N(g)^{-1}W_N({z}_{g})$.
\item[iii)] ${\frak{z}}_{g^\lambda}={\left(1-\dfrac{p^{\frac{\kappa}{2}-1}}{\lambda_g}\right)^2}\dfrac{u_K^{-1}}{(2\sqrt{-D_K})^{\frac{\kappa}{2}-1}}\cdot\dfrac{W_{Np}\circ({\rm pr}^\lambda)^*(z_{g})}{\lambda_g\,\lambda_N(g)\cE(g^\lambda)\cE^*(g^\lambda)}$\,.
\end{lemma}
\begin{proof}
The first assertion follows as a consequence of \cite[Proposition 4.1.2]{bertolinidarmonprasanna13}, where the authors compare generalized Heegner cycles and classical Heegner cycles (see also the proof of Theorem~6.5 in \cite{CastellapadicvariationofHeegnerpoints}). The second assertion follows from (i) and the definition of the $p$-stabilized generalized Heegner cycles (where we recall that we have taken $c_0=1$ and $\mathds{1}$ is the trivial character).
We now prove the final assertion. The last paragraph of the proof of \cite[Proposition 10.1.1]{KLZ2} that the composition
$$V_g\xrightarrow{\lambda_N(g)^{-1}W_N}V_g^*(1)\xrightarrow{({\rm pr}^\lambda)_*\circ W_{Np}\circ({\rm pr}^\lambda)^*} V_g$$ 
is the multiplication by $\lambda_g\,\lambda_N(g)\cE(g^\lambda)\cE^*(g^\lambda)$ map. This shows that 
$$({\rm pr}^\lambda)_*\left( \dfrac{W_{Np}\circ({\rm pr}^\lambda)^*(z_{g})}{\lambda_g\,\lambda_N(g)\cE(g^\lambda)\cE^*(g^\lambda)}\right)=\lambda_N(g)^{-1}W_N(z_g)\,.$$
Hence,
\begin{align*}
   ({\rm pr}^\lambda)_*\left(\left(1-\dfrac{p^{\frac{\kappa}{2}-1}}{\lambda_g}\right)^2 \dfrac{u_K^{-1}}{(2\sqrt{-D_K})^{\frac{\kappa}{2}-1}} \dfrac{W_{Np}\circ({\rm pr}^\lambda)^*(z_{g})}{\lambda_g\,\lambda_N(g)\cE(g^\lambda)\cE^*(g^\lambda)}\right)&=\dfrac{u_K^{-1}}{(2\sqrt{-D_K})^{\frac{\kappa}{2}-1}}\left(1-\dfrac{p^{\frac{\kappa}{2}-1}}{\lambda_g}\right)^2 {\lambda_N(g)^{-1}W_N({z}_{g})} \\
   &={\bf z}^{\mathds{1}}_{g,\lambda_g}\,.
\end{align*}
The proof from the defining property of the class $\frak{z}_{g^\lambda}$ in Definition~\ref{defn_pstabilizedGHC_levelNp}.
\end{proof}
}

{The following statement is closely related to \cite[Theorem 1.2]{ota2020}, especially when combined with Theorem~\ref{thm_main_GHC_in_families_with_ac}(i) below.}
{\begin{lemma}
\label{lemma_comparison_with_ota_1} 
$({\rm pr}_1)_*(\frak{z}_{g^\lambda})={\lambda_g}({\lambda_g-\lambda_g^{-1}p^{\kappa-1}})^{-1} \, {\bf z}^{\mathds{1}}_{g,\lambda_g}$.
\end{lemma}
\begin{proof}
By Lemma~\ref{lemma_characterize_pstab}(i),
\begin{equation}
\label{eqn_Ota_reduction_1}
({\rm pr}_1)_*(\frak{z}_{g^\lambda})={\left(1-\dfrac{p^{\frac{\kappa}{2}-1}}{\lambda_g}\right)^2}\dfrac{u_K^{-1}}{(2\sqrt{-D_K})^{\frac{\kappa}{2}-1}}\cdot\dfrac{({\rm pr}_1)_*\circ W_{Np}\circ({\rm pr}^\lambda)^*(z_{g})}{\lambda_g\,\lambda_N(g)\cE(g^\lambda)\cE^*(g^\lambda)}\,.
\end{equation}
As noted in \cite{KLZ2} (the last two lines on Page 92), we have the identities
$$W_{Np}\circ ({\rm pr}_1)^*=({\rm pr}_2)^*\circ W_N \hbox{\,\,\,\, and \,\,\,\,}  W_{Np}\circ ({\rm pr}_2)^*=p^{\kappa-2}({\rm pr}_1)^*\circ W_N\,.$$
This shows that
$$({\rm pr}_1)_*\circ W_{Np}\circ ({\rm pr}^\lambda)^*=\left(({\rm pr}_1)_*\circ ({\rm pr}_2)^*-\frac{p^{\kappa-2}}{\lambda_g}({\rm pr}_1)_*\circ ({\rm pr}_1)^*\right)\circ W_N\,.$$
Since $({\rm pr}_1)_*\circ ({\rm pr}_2)^*$ is the dual Hecke operator $T_p^\prime$, it acts on the $g$-isotypic quotient $V_g$ by multiplication by $a_p(g)=\lambda_g+\dfrac{p^{\kappa-1}}{\lambda_g}$; whereas $({\rm pr}_1)_*\circ ({\rm pr}_1)^*$ acts as multiplication by $p+1$. This shows that 
\begin{align*}
({\rm pr}_1)_*\circ W_{Np}\circ ({\rm pr}^\lambda)^*\left(z_g\right)&=\left(\lambda_g+\dfrac{p^{\kappa-1}}{\lambda_g}-\dfrac{p^{\kappa-2}(p+1)}{\lambda_g}\right)W_N(z_g)\\
&=\left(\lambda_g-\dfrac{p^{\kappa-2}}{\lambda_g} \right)W_N(z_g)=\lambda_g\,\cE(g^\lambda)\,W_N(z_g)\,.
\end{align*}
On plugging this computation in \eqref{eqn_Ota_reduction_1}, we see that
\begin{align*}
({\rm pr}_1)_*(\frak{z}_{g^\lambda})&={\left(1-\dfrac{p^{\frac{\kappa}{2}-1}}{\lambda_g}\right)^2}\dfrac{u_K^{-1}}{(2\sqrt{-D_K})^{\frac{\kappa}{2}-1}}\cdot\dfrac{\lambda_g\,\cE(g^\lambda)\,W_N(z_g)}{\lambda_g\,\lambda_N(g)\cE(g^\lambda)\cE^*(g^\lambda)}\\
&=\cE^*(g^\lambda)^{-1}{\left(1-\dfrac{p^{\frac{\kappa}{2}-1}}{\lambda_g}\right)^2}\dfrac{u_K^{-1}}{(2\sqrt{-D_K})^{\frac{\kappa}{2}-1}}\,\lambda_N^{-1}W_N(z_g)\\
&=\cE^*(g^\lambda)^{-1} {\bf z}^{\mathds{1}}_{g,\lambda_g}
\end{align*}
where the final equality is Lemma~\ref{lemma_characterize_pstab}. This completes the proof, since
$$\cE^*(g^\lambda)^{-1}=\left(1-\dfrac{p^{\kappa-1}}{\lambda_g^2} \right)=\lambda_g({\lambda_g-\lambda_g^{-1}p^{\kappa-1}})^{-1}\,.$$
\end{proof}}



\section{Interpolation of generalized Heegner cycles}
\label{sec_GHC_interpolated}
We start with the definition of a Selmer complex which we will shall make use of in our argument in Section~\ref{subsec_GHC_interpolated}.
\begin{defn}
\label{defn_selmer_complex}
\item[i)] For any complete local Noetherian ring $R$, a free $R$-module $X$ of finite rank which is endowed with a continuous action of $G_{K,\Sigma}$, and integer $c_0$ coprime to $p$ and all primes of $K$ in the set $\Sigma$, we consider the Selmer complex 
$$\widetilde{{\bf R}\Gamma}_{\rm f}(G_{K_{c_0},\Sigma},X;\Delta_X)\in D_{\rm ft} (_{R}{\rm Mod})$$
with local conditions $\Delta_X$, given as in \cite[\S6.1]{nekovar06}. We denote its cohomology by $\widetilde{H}^\bullet(G_{K_{c_0},\Sigma},X;\Delta_X)$.

\item[ii)]We shall write $\Delta(\p,\p^c)$ for local conditions which are unramified for all primes in $\Sigma$ that are coprime to $p$ $($see \cite[\S8]{nekovar06} for details$)$ and which are given by the Greenberg conditions $($see \cite[\S6.7]{nekovar06}$)$ with the choices
$$j_{\p^c}^+:X\stackrel{=}{\longrightarrow} X$$
$$j_{\p}^+:\{0\}\hookrightarrow X$$
at the primes of $K_{c_0}$ above $\p$ and $\p^c$, respectively.
\end{defn}

\subsection{Interpolation of Generalised Heegner Cycles}
\label{subsec_GHC_interpolated}
Section~\ref{subsec_GHC_interpolated} is dedicated to the proof of Theorem~\ref{thm_main_GHCinterpolate} below, which is the main  result of this article. It asserts that the generalized Heegner cycles of Bertolini--Darmon--Prasanna (which we recalled in Section~\ref{sec:GHC}) interpolate along the Coleman family $\f$  to a distribution valued cohomology class. We retain our notation from Section~\ref{subsec_Coleman_families_revisited} concerning the Coleman family $\f$ and various objects associated to it.

We start with an auxiliary lemma. We recall that $P_\kappa$ is a fixed generator of $\ker(\pi_\kappa)$ for any $\kappa\in B^\circ(r_0)_{\rm cl}$. { We put $(P_\kappa):=\LL_{k,r_0}[1/p]\cdot P_\kappa$\,.}

\begin{lemma}
\label{lemma_amice_velu_in_families}
${\displaystyle \bigcap_{\kappa\in B^\circ(r_0)_{\rm cl}}\,(P_\kappa)\,\widehat{\otimes}_{ \QQ_p}\,\mathcal{H}_{v(\lambda)}(\widetilde{\Gamma}_\ac)_L=0\,.}$
\end{lemma}

\begin{proof}
It is easy to see that it suffices to check 
$$\bigcap_{\kappa\in B^\circ(r_0)_{\rm cl}}\,(P_\kappa)\,\widehat{\otimes}_{ \QQ_p}\,\mathcal{H}_{v(\lambda)}(\Gamma)_L=0,$$
where $\Gamma$ is any  topological group  isomorphic to $\ZZ_p$ and 
 $\cH_{v(\lambda)}(\Gamma)_L$ is defined similarly to $\cH_{v(\lambda)}(\Gamma_\ac)_L$.

Let us set $\mathscr{Z}:=k+r_0\ZZ_p$ and write $\iota_{k,r_0}:\mathscr{Z}\stackrel{\sim}{\lra} \ZZ_p$ for the evident  {homeomorphism of topological spaces, which also allows us to consider $\mathscr{Z}$ as an additive topological group}. We then have $\LL_{k,r_0}=\cO_L[[\mathscr{Z}]]$ and  $\LL_{k,r_0}[1/p]=\mathcal{H}_{0}(\mathscr{Z})_L$.

For any non-negative real number $h\ge0$, let us set $C^h(\Gamma,L)$ to denote the Banach space of \emph{order-$h$ $L$-valued functions} on $\Gamma$ as given as in \cite[\S1.5]{Colmez2010Ast}. We define $D^h(\Gamma,L)$, the space of order-$h$ distributions as the continuous dual of $C^h(\Gamma,L)$. For each $h$, Amice transform induces an isomorphism
\begin{align*}
\mathcal{A}:D^h(\Gamma,L)&\stackrel{\sim}{\lra} \cH_{h}(\Gamma)_L\\
\mu&\longmapsto \mathcal{A}_\mu,
\end{align*}

Let
$$\mathcal{A}_\mu\in \bigcap_{\kappa}\,(P_\kappa)\,\widehat{\otimes}_{ \QQ_p}\,\mathcal{H}_{v(\lambda)}({\Gamma})_L \subset \mathcal{H}_{0}(\mathscr{Z})_L\,\widehat{\otimes}\,\mathcal{H}_{v(\lambda)}(\Gamma)_L=:\mathcal{H}_{(0,v(\lambda))}(\mathscr{Z}\times\Gamma)_L$$ 
be any element. Here, $\mathcal{H}_{(0,v(\lambda))}(\mathscr{Z}\times\Gamma)_L$ is a Banach space isomorphic to the space ${D}^{(0,v(\lambda))}(\mathscr{Z}\times \Gamma,L)$ introduced in \cite[Definition 5]{LoefflerHeidelberg2012} via the Amice transform, which is the continuous dual of the Banach space $C^{(0,v(\lambda))}(\mathscr{Z}\times \Gamma,L)$, which is also introduced in \cite[Definition 5]{LoefflerHeidelberg2012}.

 According to Remark~6 in op. cit., the natural containment $$C^{0}(\mathscr{Z},L)\,\widehat{\otimes}_{{L}}\,C^{(v(\lambda))}(\Gamma,L)\subset C^{(0,v(\lambda))}(\mathscr{Z}\times \Gamma,L)$$ 
 gives rise to an identification
\begin{equation}
\label{eqn_decompose_Banach}
C^{(0,v(\lambda))}(\mathscr{Z}\times \Gamma,L)\cong C^{0}(\mathscr{Z},L)\,\widehat{\otimes}_{{L}}\,C^{(v(\lambda))}(\Gamma,L)
\end{equation}
of Banach spaces. For each positive integer $i$, consider the characteristic function
$\mathds{1}_{\mathscr{Z}_i}$ of the subset $\mathscr{Z}_i:=\iota_{k,r_0}^{-1}(i+p^{\ell(i)}\ZZ_p)\subset \mathscr{Z}$, where $\ell(i)$ is as defined in \cite[\S1.3]{Colmez2010Ast}. By Theorem {I.5.14} of op. cit, these form a Banach basis of $C^{0}(\mathscr{Z},L)$. More generally, the collection 
$$\{e_{j,m,v(\lambda)}: j\in \ZZ_{\geq 0} \hbox{ and } 0\leq m\leq v(\lambda)\} \subset C^{v(\lambda)}(\Gamma,L)$$ 
given as in {the proof of Proposition I.5.8 in} op. cit. form a Banach basis  of $C^{v(\lambda)}(\Gamma,L)$, and by the isomorphism \eqref{eqn_decompose_Banach}, the collection
$\{\mathds{1}_{\mathscr{Z}_i}\,{\otimes}\,e_{j,m,v(\lambda)}: i,j\in \ZZ_{\geq0}, 0\leq m\leq v(\lambda)\}$ is a Banach basis of the space $C^{(0,v(\lambda))}(\mathscr{Z}\times \Gamma,L)$. In order to prove the lemma, we need to check that
\begin{equation}\label{eqn_question_mark}
\mu\left(\mathds{1}_{\mathscr{Z}_i}\,{\otimes}\,e_{j,m,v(\lambda)}\right)=0
\end{equation}
for any given $ i,j\in \ZZ_{\geq0}$ and  $0\leq m\leq v(\lambda)$, under our running assumption that  $\pi_\kappa(\mathcal{A}_\mu)=0$ for all crystalline points in $B(r)$ with $r<r_0$. Let us fix $j\in \ZZ_{\geq 0}$ and an integer $0\leq m\leq v(\lambda)$. Let us also consider the distribution $\mu_{j,m}\in D^0(\mathscr{Z},L)$ given by
$$\mu_{j,m}(f)=\mu(f\otimes e_{j,m,v(\lambda)})\,\,,\,\,\, f\in C^{0}(\mathscr{Z},L)$$ 
and set $\mathcal{A}_{\mu,j,m} \in \mathcal{H}_{L,0}(\mathscr{Z})$ its Amice transform. Our assumption that $\pi_\kappa(\mathcal{A}_\mu)=0$  implies that $\pi_\kappa(\mathcal{A}_{\mu,j,m})=0$, and in turn (by the infinitude of crystalline points and the fact that $\mathcal{A}_{\mu,j,m}$ is an Iwasawa function) that $\mathcal{A}_{\mu,j,m}=0$. Since the Amice transform 
\begin{align*}\mathcal{A}: D^0(\ZZ_p,L)&\stackrel{\sim}{\lra} \cH_{0}(\Zp)_L\\
\mu_{j,m}&\longmapsto \mathcal{A}_{\mu,j,m}
\end{align*}
is an isomorphism, it follows that $\mu_{j,m}=0$ for every $j\in \ZZ_{\geq0}$ and integers $0\leq m\leq v(\lambda)$. This in turn implies that \eqref{eqn_question_mark} holds for for any given $ i,j\in \ZZ_{\geq0}$ and  $0\leq m\leq v(\lambda)$, as required.
\end{proof}
We will also need the following version of Lemma~\ref{lemma_amice_velu_in_families} (which is a corollary to the proof of Lemma~\ref{lemma_amice_velu_in_families}) that concerns $\sW$-valued distributions.
{\begin{corollary}
\label{cor_lemma_amice_velu_in_families}
${\displaystyle \bigcap_{\kappa\in B^\circ(r_0)_{\rm cl}}\,(P_\kappa)\,\widehat{\otimes}_{ \QQ_p}\,\mathcal{H}_{v(\lambda)}(\widetilde{\Gamma}_\ac)_L\otimes{\sW}=0\,.}$
\end{corollary}}
{\begin{proof}
We retain the notation of the proof of Lemma~\ref{lemma_amice_velu_in_families}. On noting that $\sW$ is discretely valued, \cite[Remark 6]{LoefflerHeidelberg2012} still applies when we replace $L$ in the proof of Lemma~\ref{lemma_amice_velu_in_families} with $L\widehat{\QQ}_p^{\ur}$ to show that
\begin{equation}
\label{eqn_cor_lemma_amice_velu_in_families_1}
C^{(0,v(\lambda))}(\mathscr{Z}\times \Gamma,L\widehat{\QQ}_p^{\ur})\cong C^{0}(\mathscr{Z},L\widehat{\QQ}_p^{\ur})\,\widehat{\otimes}_{{L\widehat{\QQ}_p^{\ur}}}\,C^{(v(\lambda))}(\Gamma,L\widehat{\QQ}_p^{\ur})
\end{equation}
Moreover, since Mahler coefficients exist over $L\widehat{\QQ}_p^{\ur}$, we also have natural isomorphisms
\begin{equation}
\label{eqn_cor_lemma_amice_velu_in_families_2}
 C^{0}(\mathscr{Z},L)\otimes_L L\widehat{\QQ}_p^{\ur}\stackrel{\sim}{\lra} C^{0}(\mathscr{Z},L\widehat{\QQ}_p^{\ur})
\end{equation}
\begin{equation}
\label{eqn_cor_lemma_amice_velu_in_families_3}
 C^{(v(\lambda))}(\mathscr{Z},L)\otimes_L L\widehat{\QQ}_p^{\ur}\stackrel{\sim}{\lra} C^{(v(\lambda))}(\mathscr{Z},L\widehat{\QQ}_p^{\ur})\,.
\end{equation}
On combining \eqref{eqn_cor_lemma_amice_velu_in_families_1}, \eqref{eqn_cor_lemma_amice_velu_in_families_2} and \eqref{eqn_cor_lemma_amice_velu_in_families_3}, we deduce that
\begin{equation}
\label{eqn_cor_lemma_amice_velu_in_families_4}
\cH_{(0,v(\lambda))}(\mathscr{Z}\times\Gamma)_L\otimes \sW=D^{(0,v(\lambda))}(\mathscr{Z}\times \Gamma,L)\otimes_LL\widehat{\QQ}_p^{\ur} \cong D^{(0,v(\lambda))}(\mathscr{Z}\times \Gamma,L\widehat{\QQ}_p^{\ur})=\cH_{(0,v(\lambda))}(\mathscr{Z}\times\Gamma)_{L\widehat{\QQ}_p^{\ur}}\,.
\end{equation}
With \eqref{eqn_cor_lemma_amice_velu_in_families_4} at hand, one may proceed as in the proof of Lemma~\ref{lemma_amice_velu_in_families}, on replacing all $L$'s therein with $L\widehat{\QQ}_p^{\ur}$. 
\end{proof}}

\begin{remdef}
\label{remdefn_big_omega_eta}
{  Following \cite{perrinriou94} and \cite[Definition 3.12]{Ochiai2003AMJ} with minor alterations, we will explain that $\LL_{\cO_L}(\widetilde{\Gamma}_\cyc)=\cO_L[[\widetilde{\Gamma}_\cyc]]$ $($where $\widetilde{\Gamma}_\cyc=\Gal(\QQ(\mu_{p^{\infty}})/\QQ))$$)$ naturally interpolates $\Dcris(\cO_L(n))$ as $n$ varies in the weight space, given the fixed basis $\epsilon=\{\zeta_{p^m}\}\in \ZZ_p(1)$. We put $\chi_\cyc= \langle \chi_\cyc \rangle\omega$, where $\omega$ is the Teichm\"uller character giving the action of $G_\QQ$ on $p$th roots of unity, whereas $\langle \chi_\cyc \rangle$ is the restriction of $\chi_\cyc$ to the Galois group $\Gamma_\cyc$ of the cyclotomic $\ZZ_p$-extension.}

{\item[i)] The fixed basis $\epsilon\in \ZZ_p(1)$ determines a basis $\epsilon_n:=\epsilon^{\otimes n}\in \cO_L(n)$ for every integer $n$ and an element $t\in \Bcris$ on which $G_{\QQ_p}$ acts via the cyclotomic character, and in turn a basis $\epsilon_n\otimes t^{-n} \in \Dcris(\cO_L(n))$, which we often denote by $t^{-n}$ alone. The basis $\epsilon$ also determines a natural isomorphism 
$$\delta_n: \LL_{\cO_L}(\widetilde{\Gamma}_\cyc)/(\gamma-\chi_\cyc^n(\gamma))\stackrel{\sim}{\lra} \Dcris(\cO_L(n))$$
induced by $\gamma \mapsto (\gamma\cdot\epsilon_n)\otimes t^{-n}=\chi_\cyc^n(\gamma)(\epsilon_n\otimes t^{-n})$, so that $1$ is sent to the basis  $ \epsilon_n\otimes t^{-n}$ of $ \Dcris(\cO_L(n))$.}
{\item[ii)] As before, we will denote the open rigid analytic ball about $k$ of radius $r_0$ by $U$ $($in place of $B^\circ(r_0)$$)$ when we rely on the constructions in \cite{LZ1}. As in \cite{LZ1}, we let $\LL_U$ denote the $L$-valued analytic functions on $U$ bounded by $1$. Then the ``square-root'' $\kappa_U/2$ of the universal weight character is given as the composition
\begin{equation}
\label{eqn_sqrt_universal_wt_char}
    \kappa_U/2:G_\QQ \twoheadrightarrow \widetilde{\Gamma}_\cyc \xrightarrow{\chi_\cyc^{{1}/{2}}} \ZZ_p^\times \hookrightarrow \LL_{\cO_L}(\ZZ_p^\times)^\times \lra  \LL_U^\times\,.
\end{equation}
where square-root $\chi_\cyc^{{1}/{2}}:=\omega^{k/2}\langle \chi_\cyc\rangle^{1/2}$ of the cyclotomic character is well defined $($the attentive reader will realize that our discussion parallels \cite[Definition 2.1.3]{howard2007}$)$. We similarly define 
$$-\kappa_U/2:G_\QQ \twoheadrightarrow \widetilde{\Gamma}_\cyc \xrightarrow{\chi_\cyc^{-{1}/{2}}} \ZZ_p^\times \hookrightarrow \LL_{\cO_L}(\ZZ_p^\times)^\times \lra  \LL_U^\times\,.$$}
{ \item[iii)] We now fix $r \in (0,r_0)$ and consider the Tate algebra $\mathscr{A}^\circ$ of power bounded $L$-valued rigid analytic functions on the closed disc of radius $r$. By the choice of $U$ and $r$, we have a natural morphism $\LL_U\to \mathscr{A}^\circ$. We set
\begin{equation}
\label{eqn_sq_root_twist_big}
    \mathscr{A}^\circ(-\kappa_U/2)=\LL_{\cO_L}(\widetilde{\Gamma}_\cyc)\otimes_{[\chi_\cyc^{-1/2}]}\mathscr{A}^\circ \,.
\end{equation}
 Then $\mathscr{A}^\circ(-\kappa_U/2)$ is a free $\mathscr{A}^\circ$-module of rank one on which $G_\QQ$ acts via the character $(\kappa_U/2)^{-1}$. We further remark that in \eqref{eqn_sq_root_twist_big}, $[\chi_\cyc^{-1/2}]: \LL(\widetilde{\Gamma}_\cyc) \to  \mathscr{A}^\circ$  is the ring homomorphism given as the composition
$$\LL(\widetilde{\Gamma}_\cyc)\xrightarrow{[\omega^{-k/2}\langle \chi_\cyc\rangle^{-1/2}]} \LL(\ZZ_p^\times)\lra \LL_U \lra \mathscr{A}^\circ\,.$$
where the first arrow on the left is induced from the map $\omega^{-k/2}\langle \chi_\cyc\rangle^{-1/2}:\widetilde{\Gamma}_\cyc \to \ZZ_p^\times$ on the group-like elements of the respective completed group rings. 
}
{\item[iv)] We let $\LL(\ZZ_p^\times)^{(k)} \subset \LL(\ZZ_p^\times)$ denote the irreducible component the weight $k$ specialization factors through. For any classical point $\kappa\in B^\circ(r)_{\rm cl}$, we recall that the specialization map
$${\rm sp}_\kappa:\LL(\ZZ_p^\times)\lra \ZZ_p$$
is given by $[\zeta_{p-1}(1+p)]\mapsto \zeta_{p-1}^k(1+p)^\kappa$ $($which factors as $\LL(\ZZ_p^\times)\to \LL(\ZZ_p^\times)^{(k)}\to \LL_U \to \mathscr{A}^\circ\to \ZZ_p$, since $\kappa\in B^\circ(r)_{\rm cl}$$)$, where  $[x]\in \LL(\ZZ_p^\times)$ denotes the group like element for any given $x\in \ZZ_p^\times$ and $\zeta_{p-1}$ is any primitive $(p-1)^{\rm st}$ root of unity. It is then easy to see that ${\rm sp}_\kappa\circ (-\kappa_U/2)=\chi_\cyc^{-\kappa/2}$. 
In other words, $\mathscr{A}^\circ(-\kappa_U/2)$ interpolates $\{\ZZ_p(-\kappa/2)\}_{\kappa\in B^\circ(r)_{\rm cl}}$, in the sense that we have a natural isomorphism
$$\mathscr{A}^\circ(-\kappa_U/2)/\cP_\kappa^\circ \mathscr{A}^\circ(-\kappa_U/2)\xrightarrow[{\rm sp}_\kappa]{\sim} \cO(-\kappa/2)$$
where $\cP_\kappa:=([\zeta_{p-1}(1+p)]-\zeta_{p-1}^\kappa(1+p)^\kappa)\subset \LL_{\cO_L}(\ZZ_p^\times)$ is the kernel of ${\rm sp}_\kappa$ and $\cP_\kappa^\circ$ is its image under the map $\LL_{\cO_L}(\ZZ_p^\times)\to \mathscr{A}^\circ$.
}
\item[v)] {Let us denote by $\gamma_0 \in \widetilde{\Gamma}_\cyc$ the unique generator such that $\chi_\cyc(\gamma_0)=\zeta_{p-1}^k(1+p)$. Consider the $\LL(\ZZ_p^\times)$-module
$$\cD_\infty:=\LL(\widetilde{\Gamma}_\cyc)\widehat{\otimes}_{\ZZ_p}\LL(\ZZ_p^\times)\Big{/} (\gamma_0^{-2}\otimes 1-1\otimes[\zeta_{p-1}(1+p)])\,.$$
For any classical weight $w: \LL(\ZZ_p^\times)\to \ZZ_p$ in the connected component of the weight space containing $k$, the specialization map ${\rm sp}_w$ induces a natural isomorphism
$$\cD_\infty/\cP_w \cD_\infty \xrightarrow[1\otimes {\rm sp}_w]{\sim} \LL(\widetilde{\Gamma}_\cyc)\big{/}(\gamma_0^{-2}-\zeta_{p-1}^k(1+p)^w)\xrightarrow[\delta_{w/2}]{\sim} \Dcris(\ZZ_p(-w/2))\,.$$}
{We set $\cD^\circ(\mathscr{A}^\circ(-\kappa_U/2)):=\cD_\infty\otimes_{\LL(\ZZ_p^\times)}\mathscr{A}^\circ$. This is a free $\mathscr{A}^\circ$-module of rank one and (thanks to the discussion above) for any $\kappa\in B^\circ(r)_{\rm cl}$, the specialization map ${\rm sp}_\kappa$ induces an isomorphism
$$\cD^\circ(\mathscr{A}^\circ(-\kappa_U/2))/\cP^\circ_\kappa \cD^\circ(\mathscr{A}^\circ(-\kappa_U/2))\stackrel{\sim}{\lra} \Dcris(\cO_L(-\kappa/2))$$
mapping $1\otimes 1$ to the preferred generator $\epsilon_{-\kappa/2}\otimes t^{\kappa/2}\in \Dcris(\cO_L(-\kappa/2))$.}
{\item[vi)] We denote the generator $t^{-1}\otimes (1 \otimes 1) \in \Dcris(\ZZ_p(1))\otimes_{\ZZ_p} \cD^\circ(\mathscr{A}^\circ(-\kappa_U/2))$ by $t^{\frac{\kappa_U}{2}-1}$. As we have remarked in {\rm (v)}, $t^{\frac{\kappa_U}{2}-1}$ specializes to the generator of $\Dcris(\cO_L(1-\kappa/2))$ determined by $\epsilon$, for each $\kappa\in B^\circ(r)_{\rm cl}$.}
\end{remdef}

{\begin{defn}
\label{defn_big_omega_eta}
We let $\eta_{\f}:=\eta_\f^{\rm LZ}\otimes t^{\frac{\kappa_U}{2}-1} \in \mathbf{D}_U$ denote the unique vector which interpolates the vectors $\{\eta_{g}\otimes {\epsilon_{1-\frac{w_g}{2}}\otimes t^{\frac{w_g}{2}-1}} \in \Dcris(V_g)\}_{g\in B^\circ(r)_{\rm cl}}$, where $w_g$ stands for the weight of the $p$-stabilized eigenform $g$ and the $\LL_{U}$-module $\mathbf{D}_U$ is {the $\Lambda_U$-adic Dieudonn\'e module coming form the filtration of the $(\vp,\Gamma)$-module of $\TT_\f$,} given as in Definition~\ref{defn_appendix_big_D_twisted}. We also refer the reader to Remark~\ref{rem_appendix_explain_C_again} for a discussion that concerns the element $\eta_\f^{\rm LZ}$.
\end{defn}}

\begin{theorem}
\label{thm_main_GHCinterpolate} 
Suppose $c_0$ is a positive integer prime to $pN$.
 \item[i)] { There exists a unique $($semi-local$)$ class
$${{\pmb\zeta}^\ac_{\f,c_0}}\in H^1\left(K_{c_0,\p},\TT_\f^\ac[1/p]\right)\widehat{\otimes}_{\mathcal{H}_{0}(\widetilde{\Gamma}_\ac)} \,\mathcal{H}_{v(\lambda)}(\widetilde{\Gamma}_\ac)_{\sW}$$
which is characterized by the property that for any $\kappa\in B^\circ(r_0)_{\rm cl}$,
\begin{equation}\label{Eqn_main_BigGHC_local_ac}
{\pmb \zeta}^\ac_{\f,c_0} (\kappa) =\res_{\p}\left({{\frak{z}}_{\f(\kappa),c_0}^\ac}\right) \in H^1(K_{c_0,\p},\TT_{\f(\kappa)}^\ac[1/p])\,.
\end{equation}
Moreover, if the assumption \ref{item_Int_PR} holds, then 
$${{\pmb\zeta}^\ac_{\f,c_0}}\in H^1\left(K_{c_0,\p},\TT_\f^\ac\right)\widehat{\otimes}_{\ZZ_p[[\widetilde{\Gamma}_\ac]]} \,p^{c(\lambda)}\mathcal{H}_{v(\lambda)}^+(\widetilde{\Gamma}_\ac)_{\sW}$$
where $c(\lambda)$ is a constant that depends only on the slope $v(\lambda)$ of the Coleman family.
\item[ii)] If the assumption \ref{item_Int_PR} holds, then there exists a class 
$${{\frak{z}}_{\f,c_0}^\ac}\in H^1(G_{K_{c_0},\Sigma},\TT_{\f}^\ac)\,\widehat{\otimes}\,p^{c(\lambda)}\mathcal{H}_{v(\lambda)}^+(\widetilde{\Gamma}_\ac)_{\sW}$$ 
such that ${\pmb\zeta}_{\f,c_0}^\ac=\res_\p({\frak{z}}_{\f,c_0}^\ac)$. 
\item[iii)] If the $\LL_{(k,r_0)}\widehat{\otimes}\,\ZZ_p[[\widetilde{\Gamma}_\ac]]$-module $\widetilde{H}^2(G_{K_{c_0},\Sigma},\TT_{\f}^\ac;\Delta(\p,\p^c))$ is torsion, 
then there exists at most one class ${\frak{z}}_{\f,c_0}^\ac\in H^1\left(K_{c_0,\Sigma},\TT_\f^\ac[1/p]\right)\widehat{\otimes}_{\mathcal{H}_{0}(\widetilde{\Gamma}_\ac)} \,\mathcal{H}_{v(\lambda)}(\widetilde{\Gamma}_\ac)_{\sW}$ for which we have $\res_\p({\frak{z}}_{\f,c_0}^\ac)={\pmb\zeta}_{\f,c_0}^\ac$. It is characterized by the interpolation property 
\begin{equation}\label{eqn_global_GHC_no_ac_interpolation}
{\frak{z}}_{\f,c_0}^\ac(\kappa)= {{\frak{z}}_{\f(\kappa),c_0}^\ac}
\end{equation}
for any $\kappa\in B^\circ(r_0)_{\rm cl}$ for which $\widetilde{H}^2(G_{K_{c_0},\Sigma},T_{\f(\kappa)}^\ac;\Delta(\p,\p^c))$ is torsion as an $\cO_L[[\widetilde{\Gamma}_\ac]]$-module.}
\end{theorem}

\begin{proof}{
\item[i)] Let us set 
$${\pmb\zeta}^\ac_{\f,c_0}:={\EXP}_{\f,c_0}\left(\eta_{\f}\otimes{{\mathcal{L}_{\f,c_0}^{\iota}}}\right)\in H^1\left(K_{c_0,\p},\TT_\f^\ac[1/p]\right)\widehat{\otimes}_{\mathcal{H}_{0}(\widetilde{\Gamma}_\ac)} \,\mathcal{H}_{v(\lambda)}(\widetilde{\Gamma}_\ac)_{\sW}$$
where $\EXP_{\f,c_0}$ is defined as in \S\ref{subsec_semilocalPR} and 
$${\mathcal{L}_{\f,c_0}}:=\sum_{\widehat\rho\in \widehat{\Gal(K_{c_0}/K)}}e_{\widehat{\rho}}\cdot\cL_{\f,\psi\rho}\in \ZZ_p[{\Gal(K_{c_0}/K)}]\otimes_{\ZZ_p}\sA^\circ\widehat{\otimes_{}}\sW[[\widetilde{\Gamma}_\ac]]\,.
$$
Note that if the assumption \ref{item_Int_PR} holds, then 
$${{\pmb\zeta}^\ac_{\f,c_0}}\in H^1\left(K_{c_0,\p},\TT_\f^\ac\right)\widehat{\otimes}_{\ZZ_p[[\widetilde{\Gamma}_\ac]]} \,p^{c(\lambda)}\mathcal{H}_{v(\lambda)}^+(\widetilde{\Gamma}_\ac)_{\sW}\,.$$
By the interpolative properties of the big exponential map ${\EXP}_{\f,c_0}$ and $\mathcal{L}_{\f,c_0}$, it follows by Theorem~\ref{thm_GHC_rec_law_g} that ${{\pmb\zeta}^\ac_{\f,c_0}}$ defined in this manner verifies \eqref{Eqn_main_BigGHC_local_ac}. Its uniqueness follows once we check that 
\begin{equation}\label{eqn_required_vanishing_4_3}
\bigcap_{\kappa \in B^\circ(r_0)_{\rm cl}}\ker\left(H^1\left(K_{c_0,\p},\TT_\f^\ac[1/p]\right)\widehat{\otimes}_{\mathcal{H}_{0}(\widetilde{\Gamma}_\ac)} \,\mathcal{H}_{v(\lambda)}(\widetilde{\Gamma}_\ac)_{\sW}\stackrel{\pi_\kappa}{\lra} H^1(K_{c_0,\p},T_{\f(\kappa)}^\ac[1/p])_{\sW}\right)=0\,.
\end{equation}
Notice that 
\begin{align*}
\ker\Big(H^1\left(K_{c_0,\p},\TT_\f^\ac[1/p]\right)\widehat{\otimes}_{\mathcal{H}_{0}(\widetilde{\Gamma}_\ac)}  &\mathcal{H}_{v(\lambda)}(\widetilde{\Gamma}_\ac)_{\sW}\stackrel{\pi_\kappa}{\lra} H^1(K_{c_0,\p},T_{\f(\kappa)}^\ac)_{\sW}\Big)\\
&=(P_\kappa)\,H^1(K_{c_0,\p},\TT_\f^\ac[1/p])\widehat{\otimes}_{\mathcal{H}_{0}(\widetilde{\Gamma}_\ac)}  \mathcal{H}_{v(\lambda)}(\widetilde{\Gamma}_\ac)_{\sW}\,.
\end{align*}
Since the $\LL_{k,r_0}\,\widehat{\otimes}\,\ZZ_p[[\widetilde{\Gamma}_\ac]]$-module $H^1\left(K_{c_0,\p},\TT_\f^\ac\right)$ is finitely generated and torsion-free
, the required vanishing in \eqref{eqn_required_vanishing_4_3} follows from Corollary~\ref{cor_lemma_amice_velu_in_families}.}

{
\item[ii)] Fix $\kappa\in B^\circ(r_0)_{\rm cl}$, let us consider the following commutative diagram with exact columns:
$$\xymatrix{H^1(G_{K_{c_0},\Sigma},\TT_\f^\ac)\widehat\otimes_{\ZZ_p[[\widetilde{\Gamma}_\ac]]}p^{c(\lambda)}\mathcal{H}_{v(\lambda)}^+(\widetilde{\Gamma}_\ac)_{\sW} \ar[d]^(.5){\res_\p}\ar[r]^(.48){\pi_\kappa}&H^1(G_{K_{c_0},\Sigma},T_{\f(\kappa)}^\ac)\widehat\otimes_{\ZZ_p[[\widetilde{\Gamma}_\ac]]}p^{c(\lambda)}\mathcal{H}_{v(\lambda)}^+(\widetilde{\Gamma}_\ac)_{\sW} \ar[d]^(.5){\res_\p}\\
H^1(K_{c_0,\p},\TT_\f^\ac)\widehat\otimes_{\ZZ_p[[\widetilde{\Gamma}_\ac]]}p^{c(\lambda)}\mathcal{H}_{v(\lambda)}^+(\widetilde{\Gamma}_\ac)_{\sW}\ar[d]^(.5){\delta}\ar[r]& H^1(K_{c_0,\p},T_{\f(\kappa)}^\ac)\ar[d]^(.5){\delta(\kappa)}\widehat\otimes_{\ZZ_p[[\widetilde{\Gamma}_\ac]]}p^{c(\lambda)}\mathcal{H}_{v(\lambda)}^+(\widetilde{\Gamma}_\ac)_{\sW} \\
\widetilde{H}^2(G_{K_{c_0},\Sigma},\TT_\f^\ac;\Delta(\p,\p^c))\widehat\otimes_{\ZZ_p[[\widetilde{\Gamma}_\ac]]}p^{c(\lambda)}\mathcal{H}_{v(\lambda)}^+(\widetilde{\Gamma}_\ac)_{\sW} \ar[r]_(.49){\pi_\kappa}&\widetilde{H}^2(G_{K_{c_0},\Sigma},T_{\f(\kappa)}^\ac;\Delta(\p,\p^c))\widehat\otimes_{\ZZ_p[[\widetilde{\Gamma}_\ac]]}p^{c(\lambda)}\mathcal{H}_{v(\lambda)}^+(\widetilde{\Gamma}_\ac)_{\sW}
}$$
where $\widetilde{H}^2(G_{K_{c_0},\Sigma},\TT_\f^\ac;\Delta(\p,\p^c))$ and $ \widetilde{H}^2(G_{K_{c_0},\Sigma},T_{\f(\kappa)}^\ac;\Delta(\p,\p^c))$ are the finitely generated $\LL_{(k,r_0)}\,\widehat{\otimes}\,\ZZ_p[[\widetilde{\Gamma}_\ac]]$-modules given as in Definition~\ref{defn_selmer_complex}. Notice that we have
$$\pi_\kappa\circ\delta\left({\pmb\zeta}_{\f,c_0}^\ac\right)=\delta(\kappa)\circ \res_\p({\bf z}_{\f(\kappa)^\circ,c_0,\bblambda(\kappa)})=0\,.$$
We conclude that
$$\delta\left({\pmb\zeta}_{\f,c_0}^\ac\right)\in \bigcap_{\kappa\in B^\circ(r_0)_{\rm cl}}\,\, P_\kappa\cdot \widetilde{H}^2(G_{K_{c_0},\Sigma},\TT_{\f}^\ac;\Delta(\p,\p^c))\widehat\otimes_{\ZZ_p[[\widetilde{\Gamma}_\ac]]}p^{c(\lambda)}\mathcal{H}_{v(\lambda)}^+(\widetilde{\Gamma}_\ac)_{\sW}=\{0\}$$
where the vanishing of the intersection follows from Corollary~\ref{cor_lemma_amice_velu_in_families}, thanks to the fact that the $\LL_{(k,r_0)}\,\widehat{\otimes}\,\ZZ_p[[\widetilde{\Gamma}_\ac]]$-module $\widetilde{H}^2(G_{K_{c_0},\Sigma},\TT_{\f}^\ac;\Delta(\p,\p^c))$ is finitely generated. This concludes the proof that 
$${\pmb\zeta}_{\f,c_0}^\ac\in \ker(\delta)={\rm im}(\res_\p)\,$$
so that we have ${\pmb\zeta}_{\f,c_0}=\res_\p({\frak{z}}_{\f,c_0}^\ac)$ for some ${\frak{z}}_{\f,c_0}^\ac\in H^1(G_{K_{c_0},\Sigma},\TT_\f^\ac)\widehat\otimes_{\ZZ_p[[\widetilde{\Gamma}_\ac]]}p^{c(\lambda)}\mathcal{H}_{v(\lambda)}^+(\widetilde{\Gamma}_\ac)_{\sW}$, as required. 
\item[iii)] We need  to verify that $\widetilde{H}^1(G_{K_{c_0},\Sigma},\TT_{\f}^\ac;\Delta(\p,\p^c))=0$ under the assumption that $\widetilde{H}^2(G_{K_{c_0},\Sigma},\TT_{\f}^\ac;\Delta(\p,\p^c))$ is $\LL_{(k,r_0)}\,\widehat{\otimes}\,\ZZ_p[[\widetilde{\Gamma}_\ac]]$-torsion. This is immediate by global duality. For all $\kappa\in B^\circ(r_0)_{\rm cl}$ such that the $\cO_L[[\widetilde{\Gamma}_\ac]]$-module $\widetilde{H}^2(G_{K_{c_0},\Sigma},T_{\f(\kappa)}^\ac;\Delta(\p,\p^c))$ is torsion, we infer by global duality that $\widetilde{H}^1(G_{K_{c_0},\Sigma},T_{\f(\kappa)}^\ac;\Delta(\p,\p^c))=0$. Since we have an exact sequence 
$$0\lra \widetilde{H}^1(G_{K_{c_0},\Sigma},T_{\f(\kappa)}^\ac;\Delta(\p,\p^c))\lra H^1(G_{K_{c_0},\Sigma},T_{\f(\kappa)}^\ac)\stackrel{\res_\p}{\lra} H^1(K_\p,T_{\f(\kappa)}^\ac)$$
by the defining property of the Selmer complex $\widetilde{{\bf R}\Gamma}_{\rm f}(G_{K_{c_0},\Sigma},T_{\f(\kappa)}^\ac;\Delta(\p,\p^c))$, it follows that for all such $\kappa$, there exists at most one class in $H^1(G_{K_{c_0},\Sigma},T_{\f(\kappa)}^\ac)\,\widehat\otimes_{\ZZ_p[[\widetilde{\Gamma}_\ac]]}\,p^{c(\lambda)}\mathcal{H}_{v(\lambda)}^+(\widetilde{\Gamma}_\ac)_{\sW}$ that maps under $\res_\p$ to the element 
$$\Omega^{\varepsilon,{\p}}_{T_{\f(\kappa)}, \frac{\kappa}{2}}\left(\eta_{\f}(\kappa) \otimes {\cL_{\f(\kappa)^\circ}^\iota}\right)\in H^1(K_\p,T_{\f(\kappa)}^\ac)\,\widehat\otimes_{\ZZ_p[[\widetilde{\Gamma}_\ac]]}\,p^{c(\lambda)}\mathcal{H}_{v(\lambda)}^+(\widetilde{\Gamma}_\ac)_{\sW}.$$ 
Since 
{${{\frak{z}}_{\f(\kappa),c_0}^\ac}$ is already one such class thanks to Theorem~\ref{thm_GHC_rec_law_g}, the proof of the interpolative property \eqref{eqn_global_GHC_no_ac_interpolation} follows.}}
\end{proof}

{
\begin{remark}
\label{rem_JLZ}
If one extends the reciprocity law \cite[Theorem 8.2.4]{JLZ} of Jetchev--Loeffler--Zerbes to the $p$-adic $L$-function $\sL_{\p,\psi}(\f)$ to general $\psi$ (modulo space concerns, this appears to be a straightforward formal task),  then Remark~\ref{rem_compare_BDP_of_JLZ} together with the ``uniqueness'' assertion in Theorem~\ref{thm_main_GHCinterpolate}(i) will allow us to conclude, without assuming \ref{item_Int_PR}, that the semi-local class ${\pmb\zeta}^\ac_{\f,c_0}$ arises as the image of a global class\footnote{ As we will see in the proof of Theorem~\ref{thm_main_GHC_in_families_with_ac}(i), this global class is uniquely determined with this property.
} under the map $\res_\p$. 

We also remind the reader that \ref{item_Int_PR} holds true when $\f$ has slope zero.
\end{remark}}

{\begin{remark}
\label{rem_H2vanishing_will_be_verified}
As we have noted in the proof of Theorem~\ref{thm_main_GHCinterpolate}{(iii)},  the hypothesis that the $\LL_{(k,r_0)}\widehat{\otimes}\,\ZZ_p[[\widetilde{\Gamma}_\ac]]$-module $\widetilde{H}^2(G_{K_{c_0},\Sigma},\TT_{\f}^\ac;\Delta(\p,\p^c))$ be torsion is equivalent to the vanishing $\widetilde{H}^1(G_{K_{c_0},\Sigma},\TT_{\f}^\ac;\Delta(\p,\p^c))=0$, thanks to global duality. In Theorem~\ref{thm_main_GHC_in_families_with_ac}(i), we verify (relying on results of Hsieh and Kobayashi) that this required vanishing statement is always true. As a result, the hypothesis that $\widetilde{H}^2(G_{K_{c_0},\Sigma},\TT_{\f}^\ac;\Delta(\p,\p^c))$ is torsion can be dropped from the statement of Theorem~\ref{thm_main_GHCinterpolate}{(iii)}, as we do so in the statement of Theorem~\ref{thm_main_GHC_in_families_with_ac}(i).
\end{remark}}

When $r_0$ is sufficiently small, the main results of \cite{liu-CMH} equip us with
\begin{itemize}
\item a $(\vp,\Gamma)$-module $\DD(\TT_\f)$ over the relative Robba ring $\RR_{\sA}$, where $\sA=\sA(r_0)$ and $\Gamma=\Gal(\Qp(\mu_{p^\infty})/\Qp)$;
\item a saturated triangulation of $(\vp,\Gamma)$-modules 
\begin{equation}
\label{eq:filtration}
0\lra \sF^+\DD(\TT_\f)\lra \DD(\TT_\f)\lra \sF^-\DD(\TT_\f)\lra 0.
\end{equation}
\end{itemize}

\begin{defn}\label{defn_pottharst_selmer_on_B_r}
For any $r<r_0$, let us define $\TT_{\f}\vert_{B(r)}:=\TT_{\f}\otimes_{\LL_{k,r_0}} \mathscr{A}(r)$, which is a  free $\mathscr{A}(r)$-module of rank $2$. We define the Pottharst Selmer complex $$\widetilde{{\bf R}\Gamma}_{\rm f}(G_{K_{c_0},\Sigma},\TT_{\f}\vert_{B(r)};\Delta_{\bblambda})\in D_{\rm perf}^{[0,3]} (_{\mathscr{A}(r)}{\rm Mod})$$
with the local conditions $\Delta_{\bblambda}$ which are unramified at primes coprime to $p$ $($given as in \cite[Example 1.18]{jayanalyticfamilies}$)$ and given by strict ordinary conditions $($in the sense of \S3B in op. cit.$)$ arising from the triangulation \eqref{eq:filtration}. We denote its cohomology by $\widetilde{H}^\bullet (G_{K_{c_0},\Sigma},\TT_{\f}\vert_{B(r)};\Delta_{\bblambda})$.

\end{defn}
\begin{theorem}
\label{thm_main_GHC_in_families_with_ac}
Let $c_0$ be a positive integer coprime to $pN$. 
\item[i)] { There exists a unique class ${\frak{z}}_{\f,c_0}^\ac\in H^1(G_{K_{c_0},\Sigma},\TT_\f^\ac[1/p])\,\widehat{\otimes}_{\mathcal{H}_{0}(\widetilde{\Gamma}_\ac)}\mathcal{H}_{v(\lambda)}(\widetilde{\Gamma}_\ac)_{\sW}$ such that
$${\frak{z}}_{\f,c_0}^\ac(\kappa)={\frak{z}}_{\f(\kappa),c_0}^\ac\quad, \qquad  \forall \kappa\in B^\circ(r_0)_{\rm cl}\,. $$
Moreover, if \ref{item_Int_PR} holds true, then ${\frak{z}}_{\f,c_0}^\ac\in H^1(G_{K_{c_0},\Sigma},\TT_\f^\ac)\,\widehat{\otimes}_{\ZZ_p[[\widetilde{\Gamma}_\ac]]}\,p^{c(\lambda)}\mathcal{H}_{v(\lambda)}^+(\widetilde{\Gamma}_\ac)_{\sW}$.}
\item[ii)] Let us set ${\frak{z}}_{\f,c_0}:=\mathds{1}({\frak{z}}_{\f,c_0}^\ac)\in H^1(G_{K_{c_0},\Sigma},\TT_\f)\otimes_{}{\sW}$. For any $r<r_0$, we have
$${\frak{z}}_{\f,c_0}\vert_{B(r)} \in \widetilde{H}^1(G_{K_{c_0},\Sigma},\TT_{\f}\vert_{B(r)};\Delta_{\bblambda})\otimes_{}{\sW[1/p]}\,.$$
\end{theorem}

{Before we prove this theorem, we first explain how to descend the coefficients of the universal Heegner cycle ${\frak{z}}_{\f,c_0}^\ac$  from $\sW[1/p]$ to $L$:
\begin{proposition}
\label{prop_main_GHC_in_families_with_ac_descended_to_L}
Let $c_0$ be a positive integer coprime to $pN$. Let ${\frak{z}}_{\f,c_0}^\ac$ and ${\frak{z}}_{\f,c_0}\vert_{B(r)}$ be as in Theorem~\ref{thm_main_GHC_in_families_with_ac}.
\item[i)] { We have ${\frak{z}}_{\f,c_0}^\ac\in H^1(G_{K_{c_0},\Sigma},\TT_\f^\ac[1/p])\,\widehat{\otimes}_{\mathcal{H}_{0}(\widetilde{\Gamma}_\ac)}\mathcal{H}_{v(\lambda)}(\widetilde{\Gamma}_\ac)$. Moreover, if \ref{item_Int_PR} holds true, then we have ${\frak{z}}_{\f,c_0}^\ac\in H^1(G_{K_{c_0},\Sigma},\TT_\f^\ac)\,\widehat{\otimes}_{\ZZ_p[[\widetilde{\Gamma}_\ac]]}\,p^{c(\lambda)}\mathcal{H}_{v(\lambda)}^+(\widetilde{\Gamma}_\ac)$.}

\item[ii)] ${\frak{z}}_{\f,c_0}\vert_{B(r)} \in \widetilde{H}^1(G_{K_{c_0},\Sigma},\TT_{\f}\vert_{B(r)};\Delta_{\bblambda}).$
\end{proposition}
\begin{proof}
Let $\tau \in \Gal(L\QQ_p^\ur/\QQ_p)$ denote the Frobenius element. Proposition~\ref{prop_main_GHC_in_families_with_ac_descended_to_L} follows once we verify that 
$$(1-\tau){\frak{z}}_{\f,c_0}=0\,.$$ 
Since we have 
$${ {\frak{z}}_{\f(\kappa),c_0}^\ac\in H^1(G_{K_{c_0},\Sigma},T_{\f(\kappa)}^{\ac}[1/p])\widehat{\otimes}_{\mathcal{H}_{0}(\widetilde{\Gamma}_\ac)}\mathcal{H}_{v(\lambda)}(\widetilde{\Gamma}_\ac),}$$ 
it follows from Theorem~\ref{thm_main_GHC_in_families_with_ac}(ii) that
$$(1-\tau){\frak{z}}_{\f,c_0}(\kappa)=0$$
for every $\kappa\in  B^\circ(r_0)_{\rm cl}$. This shows that 
$${ (1-\tau){\frak{z}}_{\f,c_0} \in \bigcap_{\kappa\in B^\circ(r_0)_{\rm cl}}\,\, P_\kappa\cdot H^1(G_{K_{c_0},\Sigma},\TT_\f^\ac[1/p])\,\,\widehat{\otimes}_{\mathcal{H}_{0}(\widetilde{\Gamma}_\ac)}\mathcal{H}_{v(\lambda)}(\widetilde{\Gamma}_\ac)_{\sW}=0}$$
where the vanishing of the intersection follows from Corollary~\ref{cor_lemma_amice_velu_in_families}, thanks to the fact that the $\LL_{(k,r_0)}\,\widehat{\otimes}\,\ZZ_p[[\widetilde{\Gamma}_\ac]]$-module  ${H}^1(G_{K_{c_0},\Sigma},\TT_{\f}^\ac)$ is finitely generated and torsion-free. 
\end{proof}}
We introduce a set of Hecke characters which will be useful in the proof of Theorem~\ref{thm_main_GHC_in_families_with_ac}.
\begin{defn}
\label{defn_non_vanish_locus_of_BDP}
For any $\kappa\in B^\circ(r_0)_{\rm cl}$, let us write $\mathscr{N}_{\kappa}$ as the set of ring class characters $\chi$ of $p$-power conductor for which we have $\mathcal{L}_{\f(\kappa)^\circ}(\chi)\neq 0$. More generally, for a positive integer $c_0$,
let us also define the subset $\mathscr{N}_{\kappa}^{(c_0)}$ of $\mathscr{N}_{\kappa}$ as the set of ring class characters $\chi$ of $p$-power conductor for which we have $\mathcal{L}_{\f(\kappa)^\circ,c_0}(\chi\eta)\neq 0$ for each ring class character $\eta$ of conductor $c_0$. 
\end{defn}

It follows from \cite[Theorem C]{hsiehnonvanishing} that all but finitely many ring class characters $\chi$ of $p$-power conductor belongs to $\mathscr{N}_{\kappa}^{(c_0)}$. We recall the family $\Psi$ of anticyclotomic Hecke characters we have introduced in Section~\ref{subsec_anticyclo_padic_L_in_families} above, as well as the anticyclotomic Hecke character $\Psi_\kappa$ with infinity type $(\frac{\kappa}{2},-\frac{\kappa}{2})$ denotes its specialization in weight $\kappa$.
\begin{defn}
For any $\kappa\in B^\circ(r_0)_{\rm cl}$, we set $\mathscr{S}_{\kappa}:=\{\chi\Psi_\kappa: \chi\in \mathscr{N}_\kappa\}$\,. Also for each positive integer $c_0$, we define the subset $\mathscr{S}_{\kappa}^{(c_0)}:=\{\chi\Psi_\kappa: \chi\in \mathscr{N}_\kappa^{(c_0)}\}\subset \mathscr{S}_{\kappa}$.
\end{defn}

\begin{proof}[Proof of Theorem~\ref{thm_main_GHC_in_families_with_ac}]
\item[i)] This portion follows from Theorem~\ref{thm_main_GHCinterpolate}{(iii) and Remark~\ref{rem_JLZ}} once we verify that $\widetilde{H}^1(G_{K_{c_0},\Sigma},\TT_{\f}^\ac;\Delta(\p,\p^c))=0$. It suffices to show that 
\begin{equation}
\label{equation_desired_vanishing_opposite_Greenberg}
\widetilde{H}^1(G_{K_{c_0},\Sigma},T_{\f(\kappa)}\otimes\widehat{\phi}_0;\Delta(\p,\p^c))=0
\end{equation}
for every $\kappa\in B^\circ(r_0)_{\rm cl}$ and for $\phi_0\in \mathscr{S}_\kappa^{(c_0)}$. 
To that end, we fix an arbitrary $\kappa \in B^\circ(r_0)_{\rm cl}$ and a ring class character $\eta$ of conductor $c_0$. Let us set $\phi:=\phi_0\eta$, $T_{\phi}:=T_{\f(\kappa)}\otimes\widehat{\phi}$ and $A_\phi:=T_{\phi}\otimes\QQ_p/\ZZ_p$ to ease notation. We will prove below that 
\begin{equation}
\label{equation_desired_vanishing_opposite_Greenberg_1}
\widetilde{H}^1(G_{K,\Sigma},T_{\f(\kappa)}\otimes\widehat{\phi};\Delta(\p,\p^c))=0.
\end{equation}
for all $\phi\in \mathscr{S}_\kappa^{(\eta)}:=\{\eta\phi_0:\phi_\in \mathscr{S}_\kappa^{(c_0)}\}$. As $\eta$ varies in \eqref{equation_desired_vanishing_opposite_Greenberg_1}, the equation \eqref{equation_desired_vanishing_opposite_Greenberg} follows.

For each anticyclotomic Hecke character $\phi$ of infinity type $(\frac{\kappa}{2},-\frac{\kappa}{2})$, let us consider the canonical Selmer structure $\FFF_{\rm can}$ on $T_{\phi}$, given as in \cite[Definition 3.2.1]{mr02} (which requires no conditions at primes above $p$ and $p$-saturation of unramified conditions at primes away from $p$). Let us also write $\FFF_{({\p},{\p^c})}$ for the Selmer structure on $T_{\phi}$, which one obtains by replacing the local conditions at $\p$ by the strict local conditions and relaxing the local conditions at $\p^c$. By the defining property of $\widetilde{H}^1(G_{K,\Sigma},T_{\phi};\Delta(\p,\p^c))$, we have an injection
$$\widetilde{H}^1(G_{K,\Sigma},T_{\phi};\Delta(\p,\p^c))\hookrightarrow H^1_{\FFF_{({\p},{\p^c})}}(K,T_{\phi})$$ 
into the Selmer group attached to $\FFF_{({\p},{\p^c})}$, with finite cokernel (which is annihilated by the Tamagawa numbers for the Galois representation $T_{\phi}$ away from $p$). It therefore suffices to prove that $H^1_{\FFF_{({\p},{\p^c})}}(K,T_{\phi})=0$  for all $\phi \in \mathscr{S}_\kappa^{(\eta)}$.

Let $\FFF_{({\p},{\p^c})}^*$ denote the dual Selmer structure on 
$$T_{\phi}^{\vee}(1)=T_{\f(\kappa)}^{\vee}(1)\otimes\widehat{\psi}^{-1}=A_{\phi^{-1}}$$
where $(-)^{\vee}:={\rm Hom}(-,\QQ_p/\ZZ_p)$ is the Pontryagin duality functor and the second equality follows thanks to the self-duality of $T_{\f(\kappa)}$.

It follows from \cite[Theorem 5.2.15]{mr02} that we have $\chi(\FFF_{\rm can},T_{\phi})=2$ for the core Selmer rank of the Selmer structure $\FFF_{\rm can}$ on $T_{\phi}$. Combining this fact with \cite[Proposition 1.6]{wiles}, one deduces that 
$$\chi(\FFF_{(\p,\p^c)},T_{\phi})=0=\chi(\FFF_{(\p,\p^c)}^*,A_{\phi^{-1}}).$$ 
By \cite[Corollary 5.2.6]{mr02}, it therefore suffices to prove that $H^1_{\FFF_{({\p},{\p^c})}^*}(K,A_{\phi^{-1}})$ has finite cardinality for all $\phi \in \mathscr{S}_\kappa^{(\eta)}$, which we carry next. Let us write $L_\phi$ for the extension of $L$ generated by the values of $\widehat{\phi}$ and fix a uniformizer $\varpi_\phi\in L_\phi$ and notice that 
$$H^1_{\FFF_{({\p},{\p^c})}^*}(K,A_{\phi^{-1}})[\varpi_\phi^s]=H^1_{\FFF_{({\p},{\p^c})}^*}(K,A_{\phi^{-1}}[\varpi_\phi^s])=H^1_{\FFF_{({\p},{\p^c})}^*}(K,T_{\phi^{-1}}/\varpi^s_\psi T_{\phi^{-1}})$$
for each positive integer $s$, where the first equality is  \cite[Lemma 3.5.3]{mr02}. In order to prove that $H^1_{\FFF_{({\p},{\p^c})}^*}(K,A_{\phi^{-1}})$ has finite cardinality for all $\phi \in \mathscr{S}_\kappa^{(\eta)}$, it therefore suffices to show that 
 $${\rm length}_{\ZZ_p}\left(H^1_{\FFF_{({\p},{\p^c})}^*}(K,T_{\phi^{-1}}/\varpi^s_\psi T_{\phi^{-1}})\right)=O(1)$$
is bounded independently of $s$. 

Let us now consider the Selmer structure $\FFF_{(\p^c,\p)}$ on $T_{\phi^{-1}}$, which is obtained by replacing the local conditions determined by $\FFF_{\rm can}$ on $T_{\phi^{-1}}$ by strict local conditions at $\p^c$ and relaxed local condition at $\p$. We shall compare
\begin{itemize}
\item the induced Selmer structure $\FFF_{(\p^c,\p)}$ on $T_{\phi^{-1}}/\varpi_\phi^sT_{\phi^{-1}}$ via the natural map 
$$\pi_{\phi,\Sigma}:T_{\phi^{-1}}\twoheadrightarrow T_{\phi^{-1}}/\varpi_\phi^sT_{\phi^{-1}},$$ 
\item and the Selmer structure $\FFF_{(\p^c,\p)}^*$ on $T_{\phi^{-1}}/\varpi_\phi^sT_{\phi^{-1}}$ induced via the map 
$$\iota_{\phi,\Sigma}: T_{\phi^{-1}}/\varpi_\phi^sT_{\phi^{-1}}\stackrel{\sim}{\lra} \varpi_\phi^{-s}T_{\phi^{-1}}/T_{\phi^{-1}}\hookrightarrow A_{\phi^{-1}}=T_{\phi}^\vee(1)\,.$$
\end{itemize}
For any prime $\lambda$ of $K$ not dividing $p$, it follows from \cite[Lemma I.3.8(i)]{rubin00} that 
$$H^1_{\FFF_{({\p},{\p^c})}^*}(K_\lambda,T_{\phi^{-1}}/\varpi_\phi^sT_{\phi^{-1}}) =H^1_{\FFF_{(\p^c,\p)}}(K_\lambda,T_{\phi^{-1}}/\varpi_\phi^sT_{\phi^{-1}})\,.$$
{\bf The prime $\p^c$:} We have
\begin{align*}
H^1_{\FFF_{({\p},{\p^c})}^*}(K_{\p^c},T_{\phi^{-1}}/\varpi_\phi^sT_{\phi^{-1}})&\stackrel{\rm def}{=}\iota_{\phi,\Sigma}^{-1}\left(H^1_{\FFF_{({\p},{\p^c})}^*}(K_{\p^c},A_{\phi^{-1}})\right)\\
&=\ker\left(H^1(K_{\p^c},T_{\phi^{-1}}/\varpi_\phi^sT_{\phi^{-1}}) \stackrel{\iota_{\phi,\Sigma}}{\lra} H^1(K_{\p^c},A_{\phi^{-1}})\right)\\
&={\rm coker}\left(H^0(K_{\p^c},A_{\phi^{-1}}) \stackrel{\varpi_{\phi}^s}{\lra} H^0(K_{\p^c},A_{\phi^{-1}})\right)
\end{align*}
has size bounded independently of $s$ (since $H^0(K_{\p^c},T_{\phi^{-1}})=0$ and therefore, $H^0(K_{\p^c},A_{\phi^{-1}})$ has finite cardinality). Moreover,
\begin{align*}
H^1_{\FFF_{(\p^c,\p)}}(K_{\p^c},T_{\phi^{-1}}/\varpi_\phi^sT_{\phi^{-1}})&=\pi_{\phi,\Sigma}\left(H^1_{\FFF_{(\p^c,\p)}}(K_{\p^c},T_{\phi^{-1}})\right)\stackrel{\rm def}{=}0\,.\end{align*}
{\bf The prime $\p$:} We have
\begin{align*}
H^1_{\FFF_{({\p},{\p^c})}^*}(K_{\p},T_{\phi^{-1}}/\varpi_\phi^sT_{\phi^{-1}})&\stackrel{\rm def}{=}\iota_{\phi,\Sigma}^{-1}\left(H^1_{\FFF_{({\p},{\p^c})}^*}(K_{\p},A_{\phi^{-1}})\right)\\
&\stackrel{\rm def}{=}\iota_{\phi,\Sigma}^{-1}\left(H^1(K_{\p},A_{\phi^{-1}})\right)\\
&=H^1(K_{\p},T_{\phi^{-1}}/\varpi_\phi^sT_{\phi^{-1}})
\end{align*}
and 
\begin{align*}
H^1_{\FFF_{(\p^c,\p)}}(K_{\p},T_{\phi^{-1}}/\varpi_\phi^sT_{\phi^{-1}})&\stackrel{\rm def}{=}\pi_{\phi,\Sigma}\left(H^1_{\FFF_{(\p^c,\p)}}(K_{\p},T_{\phi^{-1}})\right)\\
&\stackrel{\rm def}{=}\pi_{\phi,\Sigma}\left(H^1(K_{\p},T_{\phi^{-1}})\right)\,.\\
\end{align*}
In summary, we have an exact sequence 
$$ 0\lra H^1_{\FFF_{(\p^c,\p)}}(K,T_{\phi^{-1}}/\varpi_\phi^sT_{\phi^{-1}})\lra H^1_{\FFF_{({\p},{\p^c})}^*}(K,T_{\phi^{-1}}/\varpi_\phi^sT_{\phi^{-1}})\lra \mathscr{C}$$
where 
\begin{align*}
\mathscr{C}&=\frac{H^1(K_{\p},T_{\phi^{-1}}/\varpi_\phi^sT_{\phi^{-1}})}{\pi_{\phi,\Sigma}\left(H^1(K_{\p},T_{\phi^{-1}})\right)}\oplus {\rm coker}\left(H^0(K_{\p^c},A_{\phi^{-1}}) \stackrel{\varpi_{\phi}^s}{\lra} H^0(K_{\p^c},A_{\phi^{-1}})\right)\\
&\hookrightarrow H^2(K_\p,T_{\phi}^{-1})[\varpi_{\phi}^s]\oplus {\rm coker}\left(H^0(K_{\p^c},A_{\phi^{-1}}) \stackrel{\varpi_{\phi}^s}{\lra} H^0(K_{\p^c},A_{\phi^{-1}})\right).
\end{align*}
has cardinality bounded independently of $s$. So, we are reduced to proving that 
 $${\rm length}_{\ZZ_p}\left(H^1_{\FFF_{(\p^c,\p)}}(K,T_{\phi^{-1}}/\varpi^s_\psi T_{\phi^{-1}})\right)=O(1)$$
for all $\phi \in \mathscr{S}_\kappa^{(\eta)}$. Since  $H^0(K,A_{\phi^{-1}})=0$, this is equivalent to showing that 
\begin{equation}
\label{eqn_desired_vanishing_Greenberg}
H^1_{\FFF_{(\p^c,\p)}}(K,T_{\phi^{-1}})=0 
\end{equation}
for all $\phi \in \mathscr{S}_\kappa^{(\eta)}$ by \cite[Lemma 3.7.1]{mr02}. The equation \eqref{eqn_desired_vanishing_Greenberg} follows from the main theorem of \cite{kobayashiGHC} since $H^1_{\FFF_{(\p^c,\p)}}(K,T_{\phi^{-1}})$ is precisely the Bloch-Kato Selmer group associated to $T_{\phi^{-1}}$ (as explained in the proof of \cite[Theorem~6.2]{CastellaHsiehGHC}) and   $\mathcal{L}_{\f(\kappa),c_0}(\chi\eta)\neq 0$ for all $\phi \in  \mathscr{S}_\kappa^{(\eta)}$.

\item[ii)] For each crystalline point $\kappa\in B(r)$, consider the commutative diagram
$$\xymatrix{{H}^1(G_{K_{c_0},\Sigma},\TT_{\f}\vert_{B(r)})\ar[r]^(.47){\res_p^{/+}}\ar[d]_{\pi_\kappa}&H^1(K_{c_0,p},\sF^-\DD(\TT_\f\vert_{B(r)}))\ar[d]^{\pi_\kappa}\ar[rr]^(.5){\langle \EXP^*(-)\,,\,\omega_{\f} \rangle}&&K_{c_0}\otimes_K\left(\mathscr{A}(r)\oplus \mathscr{A}(r)\ar[d]^{\pi_\kappa^{\oplus2}}\right)\\
H^1(G_{K_{c_0},\Sigma},V_{\f(\kappa)})\ar[r]_{\res_p^{/\rm f}}&H^1_s(K_{c_0,p},V_{\f(\kappa)})\ar[rr]_(.5){\langle \exp^*(-)\,,\,\omega_{\f(\kappa)} \rangle}^(.52){\sim}&& K_{c_0}\otimes_K\left(E\oplus E\right),
}$$
where {$\EXP^*=\bigoplus_{\p_i|p} \EXP_{\p_i}^*$ is the direct sum of Nakamura's dual exponential map associated to the (dual of the) $(\varphi,\Gamma)$-module $\sF^-\DD(\TT_\f\vert_{B(r)})$ of rank one, given as  in \cite[P.360 (just after equation (23))]{nakamuraepsilon} as $\p_i$ run over primes of $K_{c_0}$ above $p$, whereas $\exp^*$ is  the Bloch--Kato dual exponential map as in \cite[\S II.1]{kato93} and} the tensor products on the right are over the embeddings $K\stackrel{\iota_p}{\hookrightarrow}\QQ_p\subset E$. This diagram tells us that 
$$\pi_\kappa^{\oplus 2}\left( \langle \EXP^*({\frak{z}}_{\f,c_0}\vert_{B(r)})\,,\,\omega_{\f} \rangle\right)=\langle \exp^*({{\frak{z}}_{\f(\kappa),c_0}})\,,\,\omega_{\f(\kappa)} \rangle=0\,.$$ 
This shows, by the density of such $\kappa$ in $B(r)$, that
$$\langle \EXP^*({\frak{z}}_{\f,c_0}\vert_{B(r)})\,,\,\omega_{\f} \rangle=0\,.$$
{The assertion in (ii) is equivalent to the requirement that 
${\frak{z}}_{\f,c_0}\vert_{B(r)} \in \ker(\res_p^{/+})$. To verify that, we need to check that the map $\langle \EXP^*(-)\,,\,\omega_{\f} \rangle$ is injective. } 

{For any ring class character $\chi$ with conductor dividing $c_0$, we let $\langle \EXP^*(-)\,,\,\omega_{\f} \rangle^{(\chi)}$ denote the $\chi$-isotypic component of this map. To prove that the map $\langle \EXP^*(-)\,,\,\omega_{\f} \rangle$ is injective, it suffices to show that $\langle \EXP^*(-)\,,\,\omega_{\f} \rangle^{(\chi)}$ is injective for every $\chi$ as above. Thanks to the right half of the commutative diagram above, the map $\langle \EXP^*(-)\,,\,\omega_{\f} \rangle^{(\chi)}$ is non-zero. To verify that it is injective, it remains to check that 
\begin{itemize}
    \item[$(\dagger)$] the $\sA$-module $H^1(\QQ_p,\sF^-\DD(\TT_\f\vert_{B(r)}\otimes\chi))$ is flat with rank $1$.
\end{itemize}
Here, we treat $\chi$ as a character of $G_{\QQ_p}$ via Class Field Theory and via any one of the embeddings $K\hookrightarrow \QQ_p$, which induce $G_{\QQ_p}\hookrightarrow G_K$. Let us put $D:=\sF^-\DD(\TT_\f\vert_{B(r)}\otimes\chi)$ to ease our notation. Observe that $D=\cR_{\sA}(\delta)$ is a $(\varphi,\Gamma)$-module of rank one, where the character $\delta:\QQ_p^\times \to \sA^\times$ is given as
$$\delta(p)=A_p(\f)^{-1}\chi(p)\,,$$
$$\delta(u)=u^{1-\frac{\kappa_U}{2}}, \hbox{ for } u \in \ZZ_p^\times\,.$$
With this description, it is easy to see that $H^0(\QQ_p,D)=0=H^2(\QQ_p,D)$, and it follows from the Euler characteristic formula in \cite[Theorem 4.4.5]{KPX} that the $\sA$-module $H^1(\QQ_p,D)$ has rank one. It remains to prove its flatness.}

{Let $C^{\bullet}_{\varphi,\gamma}(D)$ denote the Fontaine--Herr complex associated to $D$, as in \cite[\S2.3]{KPX}, so that $H^{i}(\QQ_p,D)=H^i(C^{\bullet}_{\varphi,\gamma}(D))$ for $i=0,1,2$. It follows from Theorem 4.4.5(i) in op. cit. that $C^{\bullet}_{\varphi,\gamma}(D)$ can be represented by a complex
$$[0\lra P_0\lra P_1\lra P_2\lra 0]$$
of projective $\sA$-modules. Since $H^0(\QQ_p,D)=0=H^2(\QQ_p,D)$, it follows that the map $P_0\to P_1$ is injective and $P_1\to P_2$ is surjective. Since $P_1$ and $P_2$ are both projective $\sA$-modules, it follows that $\ker(P_1\to P_2)$ is a projective $\sA$-direct summand of $P_1$. To conclude the proof that $H^1(\QQ_p,D)=\ker(P_1\to P_2)/{\rm im}(P_0\to P_1)$ is flat, it suffices to show that $X:={\rm coker}(P_0\to P_1)$ is flat. In other words, we are reduced to proving that 
\begin{equation}
\label{eqn_tor_X_vanishes}
    {\rm Tor}_1^{\sA}(X,\sA/\frak{m})\stackrel{?}{=}0
\end{equation}
for every $\frak{m} \in {\rm Max}(A)$. }

{Let $x=x_{\frak m}:\sA \to K_x\subset \overline{\QQ}_p$ denote the ring homomorphism that corresponds to $\frak{m}$. Let us put $D_x$ for the specialization of $D$ to $x$. Then $D_x=\cR_{K_x}(\delta_x)$ is a $(\varphi,\Gamma)$-module of rank one associated to the character $\delta_x:\QQ_p^\times \to K_x^\times$ given by
$$\delta_x(p)=A_p(\f)(x)^{-1}\chi(p)\,,$$
$$\delta_x(u)=u^{1-\frac{\kappa_U(x)}{2}}, \hbox{ for } u \in \ZZ_p^\times\,.$$
It follows from \cite[Theorem 0.2]{colmez2008} and the Euler characteristic formula for $D_x$ that $H^0(\QQ_p,D_x)=0$ so long as $A_p(\f)(x)\chi^{-1}(p)$ is not a power of $p$. We can ensure that this condition holds for every $\frak{m}_x\in {\rm Max}(\sA)$ on shrinking $r$ and assume without loss of generality that $H^0(\QQ_p,D_x)=0$. Moreover, 
\begin{align*}
    0=H^0(\QQ_p,D_x)=H^0(C^{\bullet}_{\varphi,\gamma}(D_x))&\stackrel{(\ast)}{=}H^0\left([0\lra P_0\lra P_1\lra P_2\lra 0]\stackrel{\mathbb{L}}{\otimes}_{\sA}\sA/{\frak{m}}\right)\\
    &=H^0\left([0\lra P_0\lra P_1\lra X\lra 0]\stackrel{\mathbb{L}}{\otimes}_{\sA}\sA/{\frak{m}}\right)\\
    &={\rm Tor}_1^{\sA}(X,\sA/\frak{m}),
\end{align*}
where the equality $(\ast)$ follows from \cite[Theorem 2.8]{jayanalyticfamilies}. This confirms the validity of \eqref{eqn_tor_X_vanishes} and concludes our proof that $(\dagger)$ holds true.}
\end{proof}

\begin{remark}
\label{rem_Castella_Hsieh}
An Iwasawa theoretic vanishing statement closely related to \eqref{eqn_desired_vanishing_Greenberg} was proved in  \cite[Theorem 4.14(ii)]{BFSuper}. Note however that this vanishing statement does not cover our case of interest in the current article since both factors in the Rankin--Selberg products considered in op. cit. are required to be $p$-distinguished. This additional hypotheses allows us in op. cit. to construct a full-fledged Beilinson--Flach Euler systems for $g_{/K}\otimes\theta$ over the imaginary quadratic field $K$, where $g_{/K}$ denotes the base change of the elliptic modular form $g$ to $K$ and $\theta$ for a $p$-distinguished Hecke character of $K$.  This condition was weakened in \cite{BuyukbodukLei2020_02}.
\end{remark}

\subsection{Relation to classical Heegner cycles}
Let $\f$ be the Coleman family we have fixed before. 
{For any $r<r_0$, recall the class ${\frak{z}}_{\f,c_0}\vert_{B(r)} \in \widetilde{H}^1(G_{K_{c_0},\Sigma},\TT_{\f}\vert_{B(r)};\Delta_{\bblambda})$ 
we have introduced as part of the statement of Theorem~\ref{thm_main_GHC_in_families_with_ac}. We also put 
${\frak{z}}_{\f}\vert_{B(r)}:={\rm cor}_{K_1/K}\left({\frak{z}}_{\f,1}\vert_{B(r)}\right)\in \widetilde{H}^1(G_{K,\Sigma},\TT_{\f}\vert_{B(r)};\Delta_{\bblambda})$. Recall also the classical Heegner classes $z_{\f(\kappa)^\circ}\in H^1(K,V_{\f(\kappa)^\circ}^*(1))$ given as in Definition~\ref{defn_classical_heegner}. We record below the precise relation between the specializations of the ``universal Heegner class'' $\frak{z}_{\f}$ and the classical Heegner classes. This relation plays a crucial role in \cite{BPSI}.
\begin{proposition}
\label{prop_compareGHCtononpstabilizedHeegCycle}
For every positive integer $c_0$ coprime to $pN$, $r<r_0$ and $\kappa\in B(r)_{\rm cl}$, we have:
\item[i)]${\frak{z}}_{\f}\vert_{B(r)}(\kappa)=\left(1-\dfrac{p^{\frac{\kappa}{2}-1}}{\bblambda(\kappa)}\right)^2\cdot\dfrac{u_K^{-1}}{(2\sqrt{-D_K})^{\frac{\kappa}{2}-1}}\cdot\dfrac{W_{Np}\circ({\rm pr}^{\bblambda(\kappa)})^*(z_{\f(\kappa)^\circ})}{\bblambda(\kappa)\,\lambda_N(\f(\kappa)^\circ)\cE(\f(\kappa))\cE^*(\f(\kappa))}\,.$
\item[ii)]$({\rm pr^{\bblambda(\kappa)}})_*\left({\frak{z}}_{\f}\vert_{B(r)}(\kappa)\right)=\left(1-\dfrac{p^{\frac{\kappa}{2}-1}}{\bblambda(\kappa)}\right)^2\cdot\dfrac{u_K^{-1}}{(2\sqrt{-D_K})^{\frac{\kappa}{2}-1}}\cdot{\lambda_N(\f(\kappa)^\circ)^{-1}W_{N}(z_{\f(\kappa)^\circ})}$\,.
\end{proposition}
\begin{proof}
The first identity follows on combining Theorem~\ref{thm_main_GHC_in_families_with_ac}(i) and Lemma~\ref{lemma_characterize_pstab}(iii). The second identity is Theorem~\ref{thm_main_GHC_in_families_with_ac}(i) combined with the defining property of the class $\frak{z}_{g^\lambda}$ (Definition~\ref{defn_pstabilizedGHC_levelNp}) and Lemma~\ref{lemma_characterize_pstab}(ii).
\end{proof}
$\,$}


\appendix

\section{Big logarithm map along the anticyclotomic tower}
\label{sec:biglogalaonganticyclotower}

{Let $\f$ be a Coleman family as in the introduction over the affinoid disc $B(r_0)$. As in \cite{LZ1}, we denote the wide-open disc of radius $r_0$ by $U$ and we have $\LL_U=\LL_{(k,r_0)}$. We recall that $\TT_\f:=p^{-\mathscr{C}}M_U^{\circ}(\f)^*(1-\kappa_U/2)$, so that $\TT_\f(\kappa_U/2-1)$ is a $\LL_U$-lattice inside the representation $M_U(\f)^*$ of \cite{LZ1} with coefficients in $\LL_U[\frac{1}{p}]$-module.}

{ The goal of this appendix is to outline the construction of the $\Lambda_U{[1/p]}$-adic Perrin-Riou map {over relative Lubin--Tate extensions} for the representation $\TT_\f^\ac$, using the theory of $(\vp,\Gamma)$-modules from \cite{liu-CMH} and \cite{nakamura2014Jussieu}. This will allow us to define the semi-local   $\Lambda_U{[1/p]}$-adic Perrin-Riou map utilized in the main body of the article. We note that the map we define here is essentially dual to the one constructed in \cite[Proposition~8.2.1]{JLZ}. However, for our purposes, we have to  study variation in tame levels, whereas such variation has not been considered in loc. cit. }

\begin{remark}
\label{rem_appendix_explain_C_again}
In this remark, we explain the choice of $\mathscr{C}$, which is related to the vector $\eta_\f^{\rm LZ}$ of \cite[Corollary~6.4.3]{LZ1}. We first review the choice of this vector. Over an affinoid disc $V=B(r)\subset U$, one has the Galois representation $M_V(\f)^*$, which is simply given as the base change: $M_V(\f)^*=M_U(\f)^*\otimes_{\LL_U[\frac{1}{p}]}\mathscr{A}(r)$. We note that $\mathscr{A}(r)$ given as in the introduction is the ring of rigid analytic functions on $V$ and corresponds to $\cO(V)$ in \cite{LZ1}. As such, it admits a natural map $\LL_U[\frac{1}{p}]\to \mathscr{A}(r)$ and the tensor product above is given with respect to this map. When $r$ is sufficiently small, $\left(\sF^+\DD(M_V(\f)^*)\right)^{\Gamma=1}$ is a free $\mathscr{A}(r)$ module of  rank one and $\eta_{\f}^{\rm LZ}\in \left(\sF^+\DD(M_V(\f)^*)\right)^{\Gamma=1}$ is an $\sA(r)$-basis vector. On restricting to the wide-open ball $B^\circ(r)$ about $k$ of radius $r$ $($which amounts to a base change $\mathscr{A}(r)\to \LL_{(k,r)}[1/p]$$)$ and $U$ with the smaller open disc $B^\circ(r)$, we see that there exists an integer $\mathscr{C}$ such that 
$$p^{\mathscr{C}}\eta_\f^{\rm LZ}{}_{\vert_{U}}\in \cD^\circ \left(\sF^+\DD(M_U^\circ(\f)^*)\right)^{\Gamma=1}\,.$$
Elsewhere in our article, we ignore this process to decide on a good choice of an open ball $U$, and simply work with this open ball $U$ and set $r_0=r$. As a result, we write $\eta_\f^{\rm LZ}$ in place $\eta_\f^{\rm LZ}{}_{\vert_{U}}$. This also applies with the choice of $\TT_\f$ as the $\LL_U$-lattice $p^{-\mathscr{C}}M_U^{\circ}(\f)^*(1-\kappa_U/2)$, where $U$ is again the smaller open disc $B^\circ(r)$, so that we have
$\eta_\f^{\rm LZ}\in (\sF^+\DD(\TT_\f({\kappa_U}/{2}-1)))^{\Gamma=1}$.
\end{remark}
Recall the free  $\sA^\circ(r_0)$-module $\cD^\circ(\sA^\circ(r_0)(1-{\kappa_U}/{2}))$ of rank one introduced in Remark/Definition~\ref{remdefn_big_omega_eta}. Recall also that $\LL_{U}=\LL_{(k,r_0)}=\cO[[\frac{\kappa_U-k}{e_0}]]$ with $p^{-\ord_p(e_0)}=r_0$, where the radius $r_0$ chosen through the discussion in Remark~\ref{rem_appendix_explain_C_again}, and there exists a natural map $\sA^\circ(r_0)\to \LL_U$. 
{\begin{defn} 
\label{defn_appendix_big_D_twisted}
We { define the Dieudonn\'e module attached to $\sF^+\DD(\TT_\f)$ as} 
\[
\mathbf{D}_U =(\sF^+\DD(\TT_\f({\kappa_U}/{2}-1)))^{\Gamma=1}\otimes_{\LL_{U}} \left(\cD^\circ(\sA^\circ(r_0)(1-{\kappa_U}/{2}))\otimes_{\sA^\circ(r_0)} \Lambda_{U}\right),
\]
which is a free $\Lambda_{U}$-module of rank 1 and interpolates $\Dcris(T_{\f(\kappa)})^{\vp=p^{-\frac{\kappa}{2}}a_p(\f(\kappa))}$ as $\kappa \in B^\circ(r_0)_{\rm cl}$ varies.
\end{defn}}

\subsection{Perrin-Riou's big exponential map over relative Lubin--Tate extensions}\label{sec:PRLT}

It is proved in  \cite{zhangLT} (and revisited in \cite[\S9]{kobayashiGHC}) that the construction of the Perrin-Riou map attached to the representation $T_{\f(\kappa)}$ given in \cite{perrinriou94} for a classical weight $\kappa\in B^\circ(r_0)_{\rm cl}$ can be generalized to relative Lubin--Tate extensions. We outline how to carry this out for $\TT_\f$.

 Let $F$ be a finite unramified extension of $\Qp$ with ring of integers $W$ and Frobenius $\sigma$. Let $\cG$ be a relative Lubin--Tate group over $H$ of height 1. The corresponding uniformizer of $W$ is denoted by $\pi$.

Let $T_\pi$ be the Tate module associated to $\cG$. Fix $\eG=(\varepsilon_n)_{n\ge0}\in T_\pi$ and put $\pi_n=\varepsilon_n^{\sigma^{-n}}$. Let  $F_n=F(\pi_n)$, $F_\infty=\cup_{n\ge1}F_n$, $G_\infty=\Gal(F_\infty/\Qp)$ and $\Gamma_\infty=\Gal(F_\infty/F)\cong\Zp^\times$. Let $T$ be a free $\Zp$-module equipped with a continuous crystalline $G_H$-action. Let $V=T\otimes_{\Zp}\Qp$. For an integer $n$, we write
\[
T\langle n\rangle=T\otimes_{\Zp}T_\pi^{\otimes n},
\]
and define $V\langle n\rangle$  similarly.

Let $\Lambda_W(\Gamma_\infty)$ be the Iwasawa algebra of $\Gamma_\infty$ over $W$. Via the Amice transform, we may   identify it with $W[[X]]^{\psi=0}$. Define $\cH_{h}(\Gamma_\infty)$ and $\cH_h^+(\Gamma_\infty)$ similarly to $\cH_h(\Gamma_\cyc)$ and $\cH_h^+(\Gamma_\cyc)$, respectively. Their tensor products with $W$ will be denoted by  $\cH_h(\Gamma_\cyc)_W$ and $\cH_h^+(\Gamma_\cyc)_W$, respectively. Define 
\begin{align*}
\widetilde{\Delta}:\Dcris(T)\otimes_{\Zp} W[[X]]^{\psi=0}&\lra \bigoplus_{r\in \ZZ}\Dcris(T\langle r\rangle)/(1-\vp) \\
g&\mapsto \left((1\otimes D^r)g|_{X=0}\right)_{r\in \ZZ}.
\end{align*}
It follows from the work of Zhang~\cite{zhangLT} that for $h,j\gg0$ (namely, $h\ge1$ such that $\Fil^{-h}\Dcris(T)=\Dcris(T)$; in other words, all the Hodge--Tate weights of $T$ are $\le h$, and $j\ge 1-h$), there is a family of $\Lambda_W(\Gamma_\infty)$-morphisms
\begin{equation}\label{eqn_PR_big_Exp_LT}
\Omega_{T\langle j\rangle,h}^{\eG}:\left(\Dcris(T\langle j\rangle)\otimes_{\Zp} W[[X]]^{\psi=0}\right)^{\widetilde{\Delta}=0}\lra H^1(F,T\langle j\rangle\otimes\Lambda(\Gamma_\infty)^\iota)\otimes_{\Lambda(\Gamma_\infty)}\cH_h(\Gamma_\infty)_W
\end{equation}
such that the following diagram commutes
\[
\xymatrix{
\left(\Dcris(T\langle j\rangle)\otimes_{\Zp}W[[X]]^{\psi=0}\right)^{\widetilde{\Delta}=0}\ar[d]^{\Xi_{h,n}^{(j)}}\ar[rr]^{\ \ \ \ \ \ \ \Omega_{T\langle j\rangle,h}^{\eG}\ \ \ \  \ }&&  H^1(F,T\langle j\rangle\otimes\Lambda(\Gamma_\infty)^\iota)\otimes_{\Lambda(\Gamma_\infty)}\cH_h(\Gamma_\infty)_W\ar[d]\\
\Dcris(T\langle j\rangle)\otimes_{\Zp} F_n\ar[rr]^{\exp_{F_n,V\langle j\rangle}}&&H^1(F_n,V\langle j\rangle)
}
\]
where $\Xi_{h,n}^{(j)}$ sends $g$ to 
\begin{equation}\label{eq:Xi}
(h+j-1)!\left.\left(\left(p^{-n}\vp^{-n}\otimes 1\right)G^{\sigma^{-n}}\right)\right|_{X=\pi_n}
\end{equation}
with $(1-\vp\otimes \Phi)G=g$  and the second vertical map is the natural projection. 

{
Given a $(\vp,\Gamma_\infty)$-module $D$ over the Robba ring $\cR_L$, let $\HIw(F_\infty,D)$ denote $\varprojlim H^1(F,D\widehat{\otimes}_L\tilde\Lambda_n^\iota)$  where $\tilde\Lambda_n^\iota$ is defined as in \cite[\S3.1]{nakamura2014Jussieu} with $\Gamma_K$ in loc. cit. replaced by $\Gamma_\infty$. Note that  $\HIw(F_\infty,D)$ is a module over $\cH(\Gamma_\infty):=\cup_{r>0}\cH_r(\Gamma_\infty)$
and it follows from \cite[Theorem~3.3]{nakamura2014Jussieu} that there is an isomorphism 
\[
\HIw(F_\infty,D)\cong D^{\psi=1}.
\]
  Let us write $\Dcris(D)$ for the Dieudonn\'e module as defined in Definition~2.5 in op. cit. We assume that the eigenvalues of $\vp$ on $\Dcris(D)$ are not powers of  $p$. When $D$ is crystalline with Hodge--Tate weights $\ge0$, Nakamura gave a generalization of  the Perrin-Riou map attached to $D$ over the $p$-cyclotomic extension  of a finite unramified extension of $\Qp$ (see Definition~3.19 of op. cit.). We may define a similar map over $F_\infty$:
\[
\Omega_{D,h}^{\eG}	=\nabla_{h-1}\circ\cdots\nabla_1\circ\nabla_0\circ (1-\vp\otimes\Phi)^{-1}:\left(\Dcris(D)\otimes_{\Zp} W[[X]]^{\psi=0}\right)^{\widetilde{\Delta}=0}\lra \HIw(F_\infty,D)\otimes_{\Zp}W,
\]where $h$ is an integer such that all the Hodge--Tate weights of $D$ are $\le h$ and the operators $\nabla_i$ are defined as in \cite[\S3.2]{nakamura2014Jussieu}.
Indeed, we may solve the equation  $(1-\vp\otimes \Phi)G=g$ via Lemma~3.18 of op. cit., which says that  under our assumption on the $\vp$-eigenvalues in $\Dcris(D)$, there is an isomorphism
\begin{equation}\label{eq:1-vp}
1-\vp:\left(\Brig\otimes \Dcris(D)\right)^{\psi=1}\stackrel{\sim}{\longrightarrow}\left(\Brig\otimes \Dcris(D)\right)^{\psi=0},
\end{equation}
where $\Brig$ is the set of power series in $F[[X]]$ that converge on the open unit disc and we may identify $W[[X]]^{\psi=0}$ as a subset of $(\Brig)^{\psi=0}\otimes_{\Zp}W$. 
Parallel to \cite[Proposition~8.2.1]{JLZ} (where the operator $1-\vp$ is used), we may define a $\Lambda_U[1/p]$-adic version of $\Omega_{D,h}^{\eG}	$.
}

\begin{theorem}
Let $\f$ be a Coleman family as in the introduction.   { Fix an integer $h\ge0$ that is greater than or equal to the slope of $\f$. }There exists a unique ${\Lambda_U}[1/p]\widehat{\otimes}_{\Zp}\Lambda_{W}(\Gamma_\infty)$-morphism
\[
\EXP_\f^{\eG}:{\mathbf{D}_U[1/p]\widehat{\otimes}_{\Zp}\Lambda_{W}(\Gamma_\infty)}\lra H^1(F,\TT_\f\widehat\otimes \Lambda(\Gamma_\infty)^\iota){\otimes}_{\Lambda(\Gamma_\infty)} { \cH_{h}(\Gamma_\infty)_W}
\]
interpolating the Perrin-Riou big exponential maps. More explicitly, if $\kappa\in B^\circ(r_0)_{\rm cl}$ is a classical weight, then we have
$$\EXP_\f^{\eG}(\kappa)=\Omega_{T_{\f(\kappa)},\frac{\kappa}{2}}^{\eG}\big|_{\left(\Dcris(T_{\f(\kappa)})^{\vp=p^{-\frac{\kappa}{2}}a_p(\f(\kappa))}\right)\otimes \Lambda_{W}(\Gamma_\infty)}$$  
under the weight-$\kappa$ specialization map on $\mathbf{D}_U[1/p]$.
\end{theorem}
\begin{proof}
{
Recall from \cite[Theorem~6.3.2]{LZ1}
(combined with our discussion in Remark~\ref{rem_appendix_explain_C_again}) that  $\DD^+_{\ur}:=\sF^+\DD(\TT_{\f})[1/p](1-\kappa_U) $ is an unramified  rank-one $(\vp,\Gamma)$-module  over the relative Robba ring $\cR_{\LL_U}:=\cR_{\sA(r_0)}\otimes_{\sA(r_0)}\LL_U$. In more explicit terms, $\DD^+_\ur=\cR_{\LL_U}\cdot e$, where  $\Gamma$ acts on $e$ trivially, whereas the action of  $\vp$ on $e$ is given by the multiplication by $a_p(\f)$, where $a_p(\f)$ is the $U_p$-eigenvalue of $\f$ (since  the central character is trivial in our current setting).  Let $\Dcris(\DD^+_\ur)=\left(\DD^+_\ur\right)^{\Gamma=1}=\Lambda_U[1/p]\cdot e$ denote the Dieudonn\'e module of $\DD^+_\ur$. It interpolates $\Dcris(T_{\f(\kappa)})[1/p]^{\vp=a_p(\f(\kappa))}$ as $\kappa \in B^\circ(r_0)_{\rm cl}$ varies. Furthermore, we have \begin{equation}\label{eq:twist-Dieudonne}
\mathbf{D}_U[1/p] =\Dcris(\DD^+_{\ur})\otimes_{\Lambda_U} \Lambda_U(\kappa_U/2).
\end{equation}

The proof of \cite[Lemma~3.18]{nakamura2014Jussieu} can be generalized to give the $\cR_{\Lambda_U}$-adic version of \eqref{eq:1-vp}:
\[
1-\vp:\left(\Brig\otimes \Dcris(\DD_\ur^+)\right)^{\psi=1}\stackrel{\sim}{\longrightarrow}\left(\Brig\otimes \Dcris(\DD_\ur^+)\right)^{\psi=0}.
\]
This allows us to define
\[
\Omega_{\DD_\ur^+,h}^{\eG}	=\nabla_{h-1}\circ\cdots\nabla_1\circ\nabla_0\circ (1-\vp\otimes\Phi)^{-1}:\left(\Dcris(\DD_\ur^+)\otimes_{\Zp} W[[X]]^{\psi=0}\right)^{\widetilde{\Delta}=0}\lra \HIw(F_\infty,\DD_\ur^+)\otimes_{\Zp}W,
\]
after replacing $\Gamma$ by $\Gamma_\infty$ (which we can do since the action of $\Gamma$ on $e$ is trivial).
Since $\vp$ acts on $\bD_\ur^+$ via $a_p(\f)$, which is not a constant, we have
\[
\bigoplus_{r\in\ZZ}\Dcris(\DD_\ur^+\langle r\rangle)/(1-\vp)=0.
\]
In particular, this tells us that
\[
\left(\Dcris(\DD_\ur^+)\widehat{\otimes}_{\Zp} W[[X]]^{\psi=0}\right)^{\widetilde{\Delta}=0}=\Dcris(\DD_\ur^+)\widehat{\otimes}_{\Zp} W[[X]]^{\psi=0},
\]
giving us the map
\[
\Omega_{\DD_\ur^+,h}^{\eG}:\Dcris(\DD_\ur^+)\otimes_{\Zp} W[[X]]^{\psi=0}\lra \HIw(F_\infty,\DD_\ur^+)\otimes_{\Zp}W,
\]

  Similar to \cite[Theorem~8.2.8]{KLZ2}, we may  define, via the identification \eqref{eq:twist-Dieudonne},
\[
\Omega_{\DD_\ur^+\otimes_{ \LL_U}\Lambda_U(\kappa_U/2),h}^{\varepsilon_\cG}:\bD_U[1/p]\widehat{\otimes}_ {\Zp} W[[X]]^{\psi=0}\lra \HIw(F_\infty,\DD_\ur^+\otimes_{\LL_U}\Lambda_U(\kappa_U/2))\otimes_{\Zp}W,
\]
after tensoring the construction above with $\DD_\ur^+$ by $\Lambda_U(\kappa_U/2)$ . 

Consider the composition $\Xi$ given by
\begin{align*}
\HIw(F_\infty,\DD_\ur^+\otimes_{\LL_U}\Lambda_U(\kappa_U/2))\otimes_{\Zp}W\hookrightarrow &\, \HIw(F_\infty,\DD(\TT_\f))\otimes_{\Zp}W\\
\cong&\, H^1(F,\TT_\f\,\widehat\otimes\cH(\Gamma_\infty)^\iota)\otimes_{\Zp}W\\
\cong&\, H^1(F,\TT_\f\,\widehat\otimes\,\Lambda(\Gamma_\infty)^\iota)\otimes_{\Lambda(\Gamma_\infty)}\cH(\Gamma_\infty)_W,
\end{align*}
where the first inclusion is given by the inclusion of $(\vp,\Gamma_\infty)$-modules, the second isomorphism is given by \cite[Theorem~ 5.13]{SV} (after twisting by the Lubin--Tate character $\chi_{LT}$ attached to $\cG$ multiplied by the inverse of the cyclotomic character, which is an unramified character). { The final isomorphism can be proved arguing as in the proof of \cite[Proposition~II.3.1]{colmez98} relying crucially on the vanishing $H^2(F,\TT_\f\,\widehat{\otimes}\,\LL(\Gamma_\infty)^\iota)=0$, a fact that follows from our running hypothesis\footnote{ This is in fact the only point in the present work where we need the hypothesis \ref{item_irr}.} \ref{item_irr}}.
We can now define $\EXP_\f^{\varepsilon_\cG}$ by composing $\Omega_{\DD_\ur^+\otimes_{ \LL_U}\Lambda_U(\kappa_U/2),h}^{\varepsilon_\cG}$ with $\Xi$.

Note that for each  $\kappa\in B^\circ(r_0)_\mathrm{cl}$, the specialization $\EXP_\f^{\varepsilon_\cG}(\kappa)=\EXP_{\f(\kappa)}^{\varepsilon_\cG}$ takes values inside the module $H^1(F,T_{\f(\kappa)}\otimes\Lambda(\Gamma_\infty)^\iota)\otimes_{\Lambda(\Gamma_\infty)}\cH_h(\Gamma_\infty)_W$. This tells us that the image of $\EXP_\f^{\varepsilon_\cG}$ is contained in $H^1(F,\TT_\f\,\widehat\otimes\,\Lambda(\Gamma_\infty)^\iota)\otimes_{\Lambda(\Gamma_\infty)} { \cH_{h}(\Gamma_\infty)_W}$, as required. 
}
\end{proof}

{
Consider the following integrality condition on the Perrin-Riou map $\EXP_\f^{\eG}$: 
\begin{itemize}
\item[\mylabel{item_Int_PR}{{\bf (IntPR)}}] 
There exists an integer $c(h)$, which depends only on $h$, such that the image of {$\mathbf{D}_U\widehat{\otimes}_{\Zp}\Lambda_{W}(\Gamma_\infty)$ under $\EXP_\f^{\eG}$ is contained} in the module $ \HIw(F_\infty,\TT_\f)\widehat{\otimes}_{\Lambda_W(\Gamma_\infty)} p^{c(h)}\cH^+_{h}(\Gamma_\infty)_W$.
\end{itemize}

This condition holds true if $\f$ has slope-zero. In the general set up, it is work in progress of Ochiai.

}

\subsection{Semi-local Perrin-Riou maps}
\label{subsec_semilocalPR}

Let $L=K_{{c_0}p^\infty}$ be the ring class field of conductor $c_0p^\infty$ for some integer $c_0$ coprime to $p$. We write $H_{cp^\infty}$ for the corresponding ring class group. Let $\fP$ be a place of $L$ dividing $\p$ and $\widetilde\fP$ the place of $K[c]$ lying below $\fP$. Then, $K_{c_0\tilde\fP}/K_\p$ can be realized as a finite unramified extension of $\Qp$ and $L_\fP/K_{c_0,\tilde\fP}$ is an extension given by the  torsions of a relative Lubin--Tate extension of height one $\cG_{\fP}$. 

Let $\fP_1,\cdots, \fP_r$ be the primes of $L$ above $\p$. Without loss of generality, we suppose that $\fP_1$ is the prime corresponding to our fixed embeddings. Let $D_{\fP_i}$ be the decomposition group of ${\fP_i}$ and $\fP_i=\sigma_i(\fP_1)$, where $\sigma_i\in H_{c_0p^\infty}$. Then, $H_{c_0p^\infty}=\bigsqcup \sigma_iD_{\fP_1} $.  Let us write $W_i$ for the ring of integers of $K_{c_0,\tilde\fP_i}$, $\Delta_i=\Gal(K_{c_0,\tilde\fP_i}/K_\p)$, $\Gamma_i=\Gal(L_{\fP_i}/K_{c_0,\tilde\fP_i})\cong \widetilde\Gamma_\ac$ and $\cG_i=\cG_{\fP_i}$. We identify $D_{\fP_i}$ with the local Galois group $\Gal(L_{\fP_i}/K_\p)$. Given a modular form $g$ as in the main body of the article, we define the 
$\Lambda(H_{c_0p^\infty})$-morphism
\begin{align*}
\Omega_{T_g\langle j\rangle,h,c_0}^{(\cG_i)_i}:\Dcris(T_g\langle j\rangle)\otimes\Lambda(H_{c_0p^\infty})&\cong\bigoplus_{i=1}^r\Dcris(T_g\langle j\rangle)\otimes\sigma_i\Lambda(D_{\fP_i})\\
&\cong\bigoplus_{i=1}^r\Dcris(T_g\langle j\rangle)\otimes_{\Zp}\sigma_i\Lambda_{W_i}(\Gamma_i)\\
&\rightarrow\bigoplus_{i=1}^r \HIw(F_{\fP_i},T_g\langle j\rangle)\,{\otimes}\,p^{c(h)}\sigma_i\cH^+_h(\Gamma_i)_{W_i}\\
&\cong \bigoplus_{i=1}^rH^1(K_{c_0,\tilde\fP_i},T_g^\ac\langle j\rangle)\,\,{\otimes}\,p^{c(h)}\,\sigma_i\cH^+_h(\Gamma_i)_{W_i}\\
&\cong H^1(K_{c_0,\p},T_g^\ac\langle j\rangle)\,\otimes\, p^{c(h)}\cH^+_h(\tilde\Gamma_\ac)_{W_1}
\end{align*}
given by $\bigoplus_{i=1}^r\Omega_{T\langle j\rangle,h}^{\varepsilon_{\cG_i}}$. Here, we have identified $W_i$ with $W_1$ and
\[H^1(K_{c_0,\p},T_g^\ac\langle j\rangle):=\bigoplus_{i=1}^rH^1(K_{c_0,\tilde{\fP}_i},T_g^\ac\langle j\rangle).
\]
Similarly, we have the 
${\Lambda_U}\,\widehat\otimes_{\Zp}\Lambda(H_{c_0p^\infty})$-morphism
\[
{\EXP}_{\f,c_0}^{(\cG_i)_i}:\mathbf{D}_U\,\widehat\otimes_{\Zp}\Lambda(H_{c_0p^\infty})\cong\bigoplus_{i=1}^r\mathbf{D}_U\,\widehat\otimes_{\Zp}\sigma_i\Lambda(D_{\fP_i})\longrightarrow H^1(K_{c_0,\p},\TT_\f^\ac)\,\otimes\,\cH(\tilde\Gamma_\ac)_{W_1}
\]
given by $\bigoplus_{i=1}^r{\EXP}_{\f}^{\varepsilon_{\cG_i}}$.

\section{Generalized Heegner Cycles over wide open discs}\label{appB}
Our goal in this appendix is to give an unconditional proof of a weaker form of Theorem~\ref{thm_main_GHCinterpolate} (Theorem~\ref{thm_main_GHCinterpolate_appendix_B} below), which concerns the variation of Generalized Heegner Cycles over wide open discs in the anticyclotomic weight space $\mathcal{W}_\ac:={\rm Hom}_{\rm cts}(\widetilde{\Gamma}_\ac, \mathbb{G}_{\rm m})$. Until the end of the present section, we consider $\mathcal{W}_\ac$ as $(p-1)$ copies of the rigid analytic unit balls. 

In what follows, we will work with fixed Lubin--Tate groups $\cG_i$ and simply write ${\EXP}_{\f,c_0}$ in place of ${\EXP}_{\f,c_0}^{(\cG_i)_i}$ (in the notation of \S\ref{subsec_semilocalPR}). 

\subsection{Integrality over small wide open discs} 
\label{subsec_appendix_B_1}  Let $V\subset \mathcal{W}_\ac$ be any wide open disc of radius $r<1$ about the point of $\mathcal{W}_\ac(\QQ_p)$ that corresponds to the trivial character $\mathds{1}\in \mathcal{W}_\ac(\QQ_p)$. Let $\LL_V$ denote the ring of analytic functions on $V$ bounded by $1$. Note then that we have a natural inclusion of topological rings
$$\iota_V: \cH(\widetilde\Gamma_\ac) \twoheadrightarrow \cH(\Gamma_\ac)\hookrightarrow \LL_V[1/p]$$
 induced by restricting functions\footnote{Since $V$ has radius $<1$, by considering the coefficients of an element in $\cH(\Gamma_\ac)$ as a power series, it is clear that each element of $\cH(\Gamma_\ac)$ restricts to a bounded function on $V$.} on $\mathcal{W}_\ac$ to $V$. We write $\iota_V$ also to denote the map $\cH(\Gamma_\ac)\hookrightarrow \LL_V[1/p]$. Note in particular that $\iota_V\left(\Lambda(\widetilde \Gamma_\ac)\right)\subset \LL_V$.
 
For any $c_0$ as in \S\ref{subsec_semilocalPR}, let us put $\LL_V^{(c_0)}:=\Lambda(H_{c_0p^\infty})\otimes_{\Lambda(H_{p^\infty})}\LL_V$ and set $\LL_{V,W_1}:=\LL_{V}\otimes_{\ZZ_p}W_1$. We recall that $H_{c_0p^\infty}$ denotes the Galois group of the ring class extension of $K$ modulo $c_0p^\infty$ and in the special case when $c_0=1$, we have $H_{p^\infty}=\widetilde \Gamma_\ac$.  Note that we have 
 \begin{align*}
     \left(\mathbf{D}_U\widehat{\otimes}_{\Zp}\LL(H_{c_0p^\infty})\right)\otimes_{\Lambda(H_{p^\infty})}\LL_V&=  \mathbf{D}_U\widehat{\otimes}_{\Zp}\LL_V^{(c_0)}, \\
     \left(H^1(K_{c_0,\p},\TT_\f^\ac)\,\otimes\,\cH(\tilde\Gamma_\ac)_{W_1}\right)\otimes_{\Lambda(H_{p^\infty})}\LL_V&= H^1(K_{c_0,\p},\TT_\f^\ac)\,\otimes_{\Lambda(H_{p^\infty})}\, (\LL_{V,W_1})[1/p].
 \end{align*}
Therefore,  tensoring the morphism $\EXP_{\f,c_0}$ with $\Lambda_V$ gives a map of $\LL_U\widehat\otimes\LL_V^{(c_0)}$-modules
\begin{align*}
    \EXP_{\f,c_0}^{(V)}:  \mathbf{D}_U\widehat{\otimes}_{\Zp}\LL_V^{(c_0)} \rightarrow H^1(K_{c_0,\p},\TT_\f^\ac)\,\otimes_{\LL(H_{p^\infty})}\, (\LL_{V,W_1})[1/p] .
\end{align*}

\begin{lemma}
\label{lemma_appendix_B_almost_integrality}
In the setting and notation of \S\ref{subsec_appendix_B_1}, there exists $m=m(c_0,V)\in \ZZ$ so that 
$$p^m\cdot \EXP_{\f,c_0}^{(V)}\left(\mathbf{D}_U\,\widehat{\otimes}_{\Zp}\,\LL_V^{(c_0)}\right)\subset H^1(K_{c_0,\p},\TT_\f^\ac)\,\otimes_{\LL(\widetilde\Gamma_\ac)}\,  (\LL_{V,W_1})\,.$$

\end{lemma}
\begin{proof}
This follows from the definition of the map $\EXP_{\f,c_0}^{(V)}$ together with the observation that the natural containments
$$(\LL_U\widehat\otimes_{\QQ_p}\LL_{V,W_1})[1/p]\subset \LL_U\widehat\otimes_{\ZZ_p}(\LL_{V,W_1}[1/p]) \subset \LL_U[1/p]\widehat\otimes_{\QQ_p}(\LL_{V,W_1})[1/p]$$
are all equalities. Indeed, it suffices to check that the left-most ring equals the right most-ring and this follows from dualizing \eqref{eqn_decompose_Banach}, which we apply after picking $v(\lambda)$ to be $0$ and fixing a homoemorphism $V\xrightarrow{\sim}\Gamma$ (see also the discussion in \cite[P.2092]{LZ0}).
\end{proof}

\subsection{Main result over wide open discs} 
Suppose $c_0$ is a positive integer prime to $pN$ and let 
$${\pmb\zeta}^\ac_{\f,c_0}\in H^1\left(K_{c_0,\p},\TT_\f^\ac[1/p]\right){\otimes}_{\mathcal{H}_{0}(\widetilde{\Gamma}_\ac)} \,\mathcal{H}_{v(\lambda)}(\widetilde{\Gamma}_\ac)_{\sW}$$
denote the semi-local class given as in Theorem~\ref{thm_main_GHCinterpolate}(i). Employing our argument in the proof of Proposition~\ref{prop_main_GHC_in_families_with_ac_descended_to_L} with this semi-local class, we infer that 
$${\pmb\zeta}^\ac_{\f,c_0}\in H^1\left(K_{c_0,\p},\TT_\f^\ac[1/p]\right){\otimes}_{\mathcal{H}_{0}(\widetilde{\Gamma}_\ac)} \,\mathcal{H}_{v(\lambda)}(\widetilde{\Gamma}_\ac)\,.$$

\begin{defn}
\label{defn_app_B_14_06_2021}
For $V$ and $\iota_V$ given as in \S\ref{subsec_appendix_B_1}, we denote by
$${\pmb\zeta}^{(V)}_{\f,c_0}\in H^1\left(K_{c_0,\p},\TT_\f^\ac[1/p]\widehat{\otimes}_{\mathcal{H}_{0}(\widetilde{\Gamma}_\ac)}\LL_V[1/p]\right)$$
the image of 
$$({\rm id}\otimes \iota_V)\left({\pmb\zeta}^\ac_{\f,c_0}\right)\in H^1\left(K_{c_0,\p},\TT_\f^\ac[1/p]\right)\widehat{\otimes}_{\mathcal{H}_{0}(\widetilde{\Gamma}_\ac)}\LL_V[1/p]$$
under the natural morphism
\begin{equation}
    \label{eqn_App_B_1_15_05_2021_10_01}
    H^1\left(K_{c_0,\p},\TT_\f^\ac[1/p]\right)\widehat{\otimes}_{\mathcal{H}_{0}(\widetilde{\Gamma}_\ac)}\LL_V[1/p]\lra H^1\left(K_{c_0,\p},\TT_\f^\ac[1/p]\widehat{\otimes}_{\mathcal{H}_{0}(\widetilde{\Gamma}_\ac)}\LL_V[1/p]\right)\,.
\end{equation}
\end{defn}

\begin{lemma}
\label{lemma_appendix_B_almost_integrality_semi_local_classes}
For $m=m(c_0,V)$ as in Lemma~\ref{lemma_appendix_B_almost_integrality}, we have
$$p^m \cdot {\pmb\zeta}^{(V)}_{\f,c_0}\in H^1\left(K_{c_0,\p},\TT_\f^\ac\widehat{\otimes}_{\LL(\widetilde{\Gamma}_\ac)}\LL_V\right)\,.$$
\end{lemma}
\begin{proof}
In the notation of the proof of Theorem~\ref{thm_main_GHCinterpolate} and \S\ref{subsec_appendix_B_1}, note that we have
$$({\rm id}\otimes \iota_V)=\EXP_{\f,c_0}^{(V)}\left(\eta_{\f}\otimes{{\mathcal{L}_{\f,c_0}^{\iota}}}\right)$$
as per the definition of the class ${\pmb\zeta}^\ac_{\f,c_0}$ in the proof of Theorem~\ref{thm_main_GHCinterpolate}. It follows from Lemma~\ref{lemma_appendix_B_almost_integrality} that we have
$$p^m \cdot({\rm id}\otimes \iota_V) \in H^1\left(K_{c_0,\p},\TT_\f^\ac\right)\widehat{\otimes}_{\LL(\widetilde{\Gamma}_\ac)}\LL_V$$
for $m$ as in the statement of our lemma. Observing that the morphism \eqref{eqn_App_B_1_15_05_2021_10_01} restricts to a map
\begin{equation}
    \label{eqn_App_B_1_15_05_2021_10_01_bis}
    H^1\left(K_{c_0,\p},\TT_\f^\ac\right)\widehat{\otimes}_{\LL(\widetilde{\Gamma}_\ac)}\LL_V\lra H^1\left(K_{c_0,\p},\TT_\f^\ac\widehat{\otimes}_{\LL(\widetilde{\Gamma}_\ac)}\LL_V\right)\,,
\end{equation}
which concludes the proof of the lemma.
\end{proof}

\begin{defn}
For any $\kappa\in B^\circ(r_0)_{\rm cl}$, we denote by
$${\frak{z}}_{\f(\kappa),c_0}^{(V)}:=({\rm id} \otimes \iota_V)\left({\frak{z}}_{\f(\kappa),c_0}^{\ac}\right)\in H^1(K_{c_0},\TT_{g^\lambda}^\ac)\,\otimes_{\LL(\widetilde\Gamma_\ac)}\, \LL_V[1/p]=H^1(K_{c_0},\TT_{g^\lambda}^\ac\,\otimes_{\LL(\widetilde\Gamma_\ac)}\, \LL_V[1/p])$$
the restriction classical generalized Heegner cycle ${\frak{z}}_{\f(\kappa),c_0}^{\ac}$ over ${\rm Spm}(\cH(\widetilde \Gamma_\ac))$, given as in Definition~\ref{defn_pstabilizedGHC_levelNp}(i), to the wide open disc $V$. We remark that the last equality in the displayed equation above follows from \cite[Proposition~II.3.1]{colmez98}.
\end{defn}
\begin{theorem}
\label{thm_main_GHCinterpolate_appendix_B} 
Suppose $c_0$ is a positive integer prime to $pN$. Then there exists a unique class
$${\frak{z}}^{(V)}_{\f,c_0}\in H^1\left(K_{c_0},\TT_\f^\ac{\otimes}_{\LL(\widetilde{\Gamma}_\ac)} \,\LL_V[1/p]\right)$$
such that ${\pmb\zeta}_{\f,c_0}^{(V)}=\res_\p({\frak{z}}_{\f,c_0}^{(V)})$, which is characterized by the interpolation property 
\begin{equation}\label{eqn_global_GHC_no_ac_interpolation_appnedix_B}
{\frak{z}}_{\f,c_0}^{(V)}(\kappa)= {{\frak{z}}_{\f(\kappa),c_0}^{(V)}}
\end{equation}
for any $\kappa\in B^\circ(r_0)_{\rm cl}$.
\end{theorem}

\begin{proof}
Our argument in the proof of Theorem~\ref{thm_main_GHCinterpolate}(ii) with   $\cH_{v(\lambda)}^+(\widetilde{\Gamma}_\ac)_{\sW}$ replaced by $\LL_V$  reduces the proof of the existence of the global classes ${\frak{z}}^{(V)}_{\f,c_0}$ verifying 
\begin{equation}
\label{eqn_app_B1_15_06_2021_10_37}
    {\pmb\zeta}_{\f,c_0}^{(V)}=\res_\p({\frak{z}}_{\f,c_0}^{(V)})
\end{equation}
 to the vanishing of
$$\bigcap_{\kappa\in B^\circ(r_0)_{\rm cl}}\,\, P_\kappa\cdot \widetilde{H}^2(G_{K_{c_0},\Sigma},\TT_\f^\ac{\otimes}_{\LL(\widetilde{\Gamma}_\ac)} \,\LL_V;\Delta(\p,\p^c))\,.$$
This is clear, since $\bigcap_{\kappa\in B^\circ(r_0)_{\rm cl}}\,\, P_\kappa =\{0\}$ and the $\LL_U\widehat\otimes\LL_V$-module $\widetilde{H}^2(G_{K_{c_0},\Sigma},\TT_\f^\ac{\otimes}_{\LL(\widetilde{\Gamma}_\ac)} \,\LL_V;\Delta(\p,\p^c))$ is finitely generated.

In order to prove the uniqueness of ${\frak{z}}^{(V)}_{\f,c_0}$ given by the property \eqref{eqn_app_B1_15_06_2021_10_37}, we may first reduce this claim to the vanishing $\widetilde{H}^1(G_{K_{c_0},\Sigma},\TT_\f^\ac{\otimes}_{\LL(\widetilde{\Gamma}_\ac)} \,\LL_V;\Delta(\p,\p^c))=0$, following the proof of Theorem~\ref{thm_main_GHCinterpolate}(iii). Only in this proof, let us put ${}^\circ\TT_\f^\ac:=\TT_\f^\ac{\otimes}_{\LL({\widetilde\Gamma}_\ac)}\LL(\Gamma_\ac)$ and observe that ${}^\circ\TT_\f^\ac\,\otimes_{\LL(\widetilde{\Gamma}_\ac)}\,\LL_V=\TT_\f^\ac\,\otimes_{\LL({\Gamma}_\ac)}\LL_V$. It follows from the proof of Theorem~\ref{thm_main_GHC_in_families_with_ac}(i) (on passing to the $\mathds{1}_{\Delta}$-isotypic components for the analogous assertion concerning $\TT_\f^\ac$, where $\mathds{1}_{\Delta}$ is the trivial representation of $\Delta_\ac$) that  $\widetilde{H}^1(G_{K_{c_0},\Sigma},{}^\circ\TT_\f^\ac;\Delta(\p,\p^c))=0$. This vanishing result combined with the global Euler-Poincar\'e characteristic formulae and the fact that the complete local Noetherian ring $\LL_U\,\widehat\otimes\,\LL(\Gamma_\ac)$ is regular tell us that the Selmer complex $\widetilde{R\Gamma}(G_{K_{c_0},\Sigma},{}^\circ\TT_\f^\ac;\Delta(\p,\p^c))$ can be represented by a perfect complex of the form 
\begin{equation}
    \label{eqn_app_B_complex}
   [ M \xrightarrow{u} M ] 
\end{equation} 
concentrated in degrees $1$ and $2$, where $M$ is a projective $\LL_U\,\widehat\otimes\,\LL(\Gamma_\ac)$-module of rank one and $u$ is injective. By the base change property of Selmer complexes (c.f. \cite[\S8]{nekovar06} and \cite[Theorem 1.4]{jaycyclotmotives}), we have a quasi-isomorphism
$$[ M \xrightarrow{u} M ]\otimes_{\LL({\Gamma}_\ac)}^{\mathbf L}\LL_V=\widetilde{R\Gamma}(G_{K_{c_0},\Sigma},\TT_\f^\ac;\Delta(\p,\p^c))\otimes_{\LL({\Gamma}_\ac)}^{\mathbf L}\LL_V \xrightarrow{\sim} \widetilde{R\Gamma}(G_{K_{c_0},\Sigma},\TT_\f^\ac{\otimes}_{\LL({\Gamma}_\ac)} \,\LL_V;\Delta(\p,\p^c)).$$
This fact combined with the exactness of the sequence
$$0\lra {\rm Tor}_1^{\LL(\Gamma_\ac)}({\rm coker}(u),\LL_V)\xrightarrow{u_0} M\otimes_{\LL({\Gamma}_\ac)}\LL_V \xrightarrow{u\otimes_{\LL({\Gamma}_\ac)}\LL_V} M \otimes_{\LL({\Gamma}_\ac)}\LL_V\lra {\rm coker}(u)\otimes_{\LL({\Gamma}_\ac)}\LL_V\lra 0$$
(where the injectivity of $u_0$ follows from the fact that ${\rm Tor}_1^{\LL(\Gamma_\ac)}(M,\LL_V)=0$, as $M$ is a flat $\LL(\Gamma_\ac)$-module), we infer that
 $$\widetilde{H}^1(G_{K_{c_0},\Sigma},{}^\circ\TT_\f^\ac{\otimes}_{\LL({\Gamma}_\ac)} \,\LL_V;\Delta(\p,\p^c))=\ker\left( u\otimes_{\LL({\Gamma}_\ac)} \LL_V\right)={\rm im}(u_0)\cong {\rm Tor}_1^{\LL(\Gamma_\ac)}({\rm coker}(u),\LL_V)\,.$$ 
Since ${\rm Tor}_1^{\LL(\Gamma_\ac)}({\rm coker}(u),\LL_V)$ is torsion as a $\LL(\Gamma_\ac)$-module, it is also torsion as a $\LL_V$-module (as the $\LL_V$-module structure on ${\rm Tor}_1^{\LL(\Gamma_\ac)}({\rm coker}(u),\LL_V)$ is compatible with its $\LL(\Gamma_\ac)$-module structure, via the morphism $\iota_V$). This shows that the $\LL_U\,\widehat\otimes\,\LL_V$-module $\widetilde{H}^1(G_{K_{c_0},\Sigma},{}^\circ\TT_\f^\ac{\otimes}_{\LL({\Gamma}_\ac)} \,\LL_V;\Delta(\p,\p^c))$ is torsion. On the other hand, since the $G_\QQ$-representation $\TT_\f^\ac$ is residually irreducible, it follows that (c.f. \cite{mr02}, the proof of Proposition 2.1.5)  $\widetilde{H}^1(G_{K_{c_0},\Sigma},{}^\circ\TT_\f^\ac{\otimes}_{\LL({\Gamma}_\ac)} \,\LL_V;\Delta(\p,\p^c))$ is torsion-free. This concludes the proof that $$\widetilde{H}^1(G_{K_{c_0},\Sigma},\TT_\f^\ac{\otimes}_{\LL({\Gamma}_\ac)} \,\LL_V;\Delta(\p,\p^c))=\widetilde{H}^1(G_{K_{c_0},\Sigma},{}^\circ\TT_\f^\ac{\otimes}_{\LL({\Gamma}_\ac)} \,\LL_V;\Delta(\p,\p^c))=0\,,$$ 
as required.

Our argument in the proof of Theorem~\ref{thm_main_GHCinterpolate}(iii), in light of our observations concerning the Selmer complex $\widetilde{R\Gamma}(G_{K_{c_0},\Sigma},\TT_\f^\ac{\otimes}_{\LL(\widetilde{\Gamma}_\ac)} \,\LL_V;\Delta(\p,\p^c))$, applies to verify the interpolation property \eqref{eqn_global_GHC_no_ac_interpolation_appnedix_B}.
\end{proof}

\bibliographystyle{amsalpha}
\bibliography{references}

\end{document}